\ifdefined\draft
  \documentclass[draft]{article}
\else
  \documentclass{article}
\fi

\usepackage{amsmath}         
\usepackage{amsthm,thmtools} 
\usepackage{enumitem}        

\usepackage[mathscr]{eucal} 

\usepackage{tikz}
  \usetikzlibrary{matrix,arrows}

\usepackage{amssymb}
\usepackage{stmaryrd}

\usepackage{titlesec}

\usepackage{fixme}             
  \fxsetup{targetlayout=color}

\usepackage[nobysame,alphabetic,initials]{amsrefs} 
\usepackage[colorlinks,citecolor=blue]{hyperref}
\usepackage[capitalize,nosort]{cleveref}           

\newcommand{\unit}{section}

\author{Karol Szumi{\l}o}
\title{Two Models for the Homotopy Theory\\of Cocomplete Homotopy Theories}

\newcommand{\from}{\colon} 

\newcommand{\acto}{\stackrel{\we}{\cto}} 
\newcommand{\cto}{\rightarrowtail}       
\newcommand{\ito}{\hookrightarrow}       
\newcommand{\fto}{\twoheadrightarrow}    
\newcommand{\sto}{\rightarrowtriangle}   
\newcommand{\weto}{\stackrel{\we}{\to}}  

\newcommand{\Cat}{\ncat{Cat}}   
\newcommand{\sSet}{\ncat{sSet}} 

\newcommand{\spn}{{\raisebox{-.5\height}{\scalebox{1.5}{\ensuremath{\mathord{\ulcorner}}}}}} 

\tikzset
{
  diagram/.style=
  {
    matrix of math nodes,
    column sep=2.5em,
    row sep=2.5em,
    text height=1.5ex,
    text depth=.25ex
  },
  cross line/.style={preaction={draw=white,-,line width=6pt}},
  every to/.style={font=\footnotesize},
  cof/.style={>->},               
  fib/.style={->>},               
  inj/.style={right hook->},      
  surj/.style={-open triangle 60} 
}

\newcommand{\dfn}{(\arabic*)} 
\newcommand{\thm}{(\arabic*)} 

\DeclareFontFamily{OT1}{bpzc}{}
\DeclareFontShape{OT1}{bpzc}{m}{it}{<-> s * [1.2] pzcmi8z}{}

\DeclareMathAlphabet{\mathpzc}{OT1}{bpzc}{m}{it}

\makeatletter
\newenvironment{ctikzpicture}
{
  \begingroup
  \smallskip
  \par
  \centering
  \begin{tikzpicture}
}{
  \end{tikzpicture}
  \par
  \smallskip
  \endgroup
  \aftergroup\@afterindentfalse
  \aftergroup\@afterheading
}
\makeatother

\newcommand{\enumhack}{\leavevmode} 

\newcommand{\Ho}{\operatorname{Ho}}        
\newcommand{\op}{{\mathord\mathrm{op}}}    
\newcommand{\slice}{\mathbin\downarrow}    
\newcommand{\textop}{\textsuperscript{op}} 

\let\bigcoprod\coprod                                      
\newcommand{\bigprod}{\prod}                               
\newcommand{\bigpull}[2]{\sideset{}{_{#1}}\bigprod_{#2}}   
\newcommand{\bigpush}[2]{\sideset{}{_{#1}}\bigcoprod_{#2}} 
\newcommand{\colim}{\operatorname{colim}}                  
\renewcommand{\coprod}{\amalg}                             
\newcommand{\diag}{\operatorname{diag}}                    
\newcommand{\ev}{{\mathord\mathrm{ev}}}                    
\newcommand{\id}{{\mathord\mathrm{id}}}                    
\newcommand{\Lan}{\operatorname{Lan}}                      
\newcommand{\push}{\amalg}                                 
\newcommand{\quot}{\mathbin{/}}                            
\newcommand{\Ran}{\operatorname{Ran}}                      

\newcommand{\uslice}{\mathbin{\backslash}}

\newcommand{\bigunion}{\bigcup}               
\newcommand{\inter}{\cap}                     
\newcommand{\set}[2]{\left\{#1\mid#2\right\}} 
\newcommand{\union}{\cup}                     

\newcommand{\bd}{\mathord\partial} 

\renewcommand{\iff}{if and only if}
\newcommand{\st}{such that}
\newcommand{\wrt}{with respect to}

\newcommand{\Tfae}{The following are equivalent}

\newcommand{\rlp}{right lifting property}

\newcommand{\face}{\delta} 
\newcommand{\dgn}{\sigma}  

\newcommand{\simp}[1]{\Delta[#1]}       
\newcommand{\bdsimp}[1]{\bd \Delta[#1]} 

\makeatletter
\def\horn#1{\expandafter\horn@i#1,,\@nil}
\def\horn@i#1,#2,#3\@nil{\Lambda^{#2}[#1]}
\makeatother

\newcommand{\join}{\star}           
\newcommand{\Sd}{\operatorname{Sd}} 
\newcommand{\Sk}{\operatorname{Sk}} 

\newcommand{\init}{0} 

\renewcommand{\epsilon}{\varepsilon} 
\renewcommand{\phi}{\varphi}         

\newcommand{\htp}{\simeq} 
\newcommand{\iso}{\cong}  
\newcommand{\we}{\sim}    

\renewcommand{\emptyset}{\varnothing} 
\newcommand{\nat}{\mathbb{N}}         

\renewcommand{\hat}{\widehat}         
\renewcommand{\tilde}{\widetilde}     
\newcommand{\uvar}{\mathord{\relbar}} 

\providecommand{\unit}{chapter}
\declaretheorem[style=definition,within=\unit]{definition}
\declaretheorem[style=definition,numberlike=definition]{example}
\declaretheorem[style=definition,numberlike=definition]{remark}

\declaretheorem[style=plain,numberlike=definition]{corollary}
\declaretheorem[style=plain,numberlike=definition]{lemma}
\declaretheorem[style=plain,numberlike=definition]{proposition}
\declaretheorem[style=plain,numberlike=definition]{theorem}
\declaretheorem[style=plain,numbered=no,name=Theorem]{theorem*}

\Crefname{corollary}{Corollary}{Corollaries}
\Crefname{definition}{Definition}{Definitions}
\Crefname{example}{Example}{Examples}
\Crefname{lemma}{Lemma}{Lemmas}
\Crefname{proposition}{Proposition}{Propositions}
\Crefname{remark}{Remark}{Remarks}
\Crefname{theorem}{Theorem}{Theorems}

\newcommand{\cat}[1]{\mathcal{#1}}  
\newcommand{\ncat}[1]{\mathsf{#1}}  
\newcommand{\qcat}[1]{\mathscr{#1}} 

\newcommand{\dgrm}{\operatorname{Dg}} 
\newcommand{\nf}{\operatorname{N_f}}  

\newcommand{\asSet}{{\ncat{asSet}}}    
\newcommand{\CofCat}{\ncat{CofCat}}    
\newcommand{\QCat}{\ncat{QCat}}        

\newcommand{\ucone}{\rhd}         
\newcommand{\univ}{\mathrm{univ}} 

\newcommand{\Rllp}{Reedy left lifting property}
\newcommand{\Rrlp}{Reedy right lifting property}

\newcommand{\totl}[1]{\langle#1]} 
\newcommand{\tott}[1]{[#1\rangle} 

\newcommand{\hornl}[1]{\Lambda^0\totl{#1}}
\newcommand{\hornt}[1]{\Lambda^{#1}\tott{#1}}

\makeatletter
\def\hornm#1{\expandafter\hornm@i#1,,\@nil}
\def\hornm@i#1,#2,#3\@nil{\Lambda^{#2}\hat{[#1]}}
\makeatother

\newcommand{\pDelta}{\Delta_\sharp}

\newcommand{\Reedy}{\mathrm{R}}

\makeatletter
\def\pfilt#1{\expandafter\pfilt@i#1,,\@nil}
\def\pfilt@i#1,#2,#3\@nil{D^{#1}#2}
\makeatother

\makeatletter
\def\filt#1{\expandafter\filt@i#1,,\@nil}
\def\filt@i#1,#2,#3\@nil{D^{(#1)}#2}
\makeatother

\newcommand{\wiso}{E(1)}    
\newcommand{\nwiso}{{E[1]}} 

\newcommand{\eCofCat}{\overline{\CofCat}} 

\newcommand{\hore}{homotopical}
\newcommand{\HoRe}{Homotopical}

\newcommand{\aug}{\mathrm{a}}

\makeatletter
\newcommand \DotFill {\leavevmode \cleaders \hb@xt@ .33em{\hss .\hss }\hfill \kern \z@}
\makeatother

\titleformat*{\subsection}{\normalsize\bfseries}

\begin{document}

  \maketitle

  \begin{abstract}
    We prove that the homotopy theory of cofibration categories
    is equivalent to the homotopy theory of cocomplete quasicategories.
    This is achieved by presenting both homotopy theories
    as fibration categories and constructing an explicit equivalence
    between them.
  \end{abstract}

  \section*{Introduction}

There are a few notions that formalize the concept of
a cocomplete homotopy theory, but it is not clear
how they compare to each other.
We consider two of them: cofibration categories and cocomplete quasicategories
and prove that they are indeed equivalent.
More precisely, our main result
(\Cref{fibcat-of-cofcats,fibcat-of-fc-quasicats,nf-equivalence}) is as follows.

\begin{theorem*}
  Both the category of cofibration categories
  and the category of cocomplete quasicategories
  carry structures of fibration categories
  and these two fibration categories are equivalent.
\end{theorem*}

There are two particularly noteworthy steps in the proof.
One is the existence of a fibration category of cofibration categories
which gives a positive answer to a version of a question posed by Hovey
who asked whether there is a ``model $2$-category of model categories''
\cite{ho}*{Problem 8.1}.
(The main result itself can be seen as an answer to a version of
\cite{ho}*{Problem 8.2}.)

The second one is a construction of a new functor
that to every cofibration category associates a cocomplete quasicategory
called the \emph{quasicategory of frames}
and has a number of convenient properties compared to other known functors
of this type.
This functor implements simplicial localization of cofibration categories
as will be proven in an upcoming paper \cite{ks}.

As an example of an application of this construction,
it can be shown that the simplicial localization of
any categorical model of dependent type theory
is a locally cartesian closed quasicategory \cite{k}.
This problem has proven difficult when working with known models of
simplicial localization.
However, every categorical model of type theory is a fibration category
\cite{akl}*{Theorem 2.2.5} and hence, by results of the present paper,
its localization is a quasicategory with finite limits.
Moreover, having an explicit description of the localization
in terms of frames makes quasicategories arising from type theory convenient
to work with.
This result can be seen as a step towards describing internal languages
of higher categories.

In the remainder of the introduction we will discuss some background
on the \emph{homotopy theory of homotopy theories}.
In order to explain what exactly we mean by a ``homotopy theory''
and the ``homotopy theory of homotopy theories'' we will give
a brief overview of various approaches to abstract homotopy theory.
They will be very roughly classified into two types: the classical ones
in the spirit of Quillen's \emph{homotopical algebra}\footnote{
Usually, the phrase ``homotopical algebra'' is used to refer
to Quillen model categories.
Here, we extend its meaning to various related notions such as
Brown's categories of fibrant objects or Thomason model categories.}
and the modern ones in the spirit of \emph{higher category theory}.

\subsection*{Homotopical algebra: classical models of homotopy theories}

In the past 50 years many different approaches to abstract homotopy theory
have been introduced.
Perhaps surprisingly, the first such approach,
the theory of \emph{model categories},
remains one of the most intricate ones to the present day.
Model categories were introduced by Quillen \cite{q}.
He defined a model category as a category equipped with
three classes of morphisms: weak equivalences, cofibrations and fibrations
subject to certain conditions that axiomatized well-known methods
of algebraic topology and put them into an abstract framework.
This framework proved to be very powerful and widely applicable
and today it constitutes one of the main tool-sets of homotopy theory.
An important feature of the theory of model categories is that
it allows for comparisons between different homotopy theories via the notion
of a \emph{Quillen adjoint pair}.
A typical example of a problem that can be solved using model categories is
that classical colimits are usually not homotopy invariant
and hence they have to be replaced by better behaved \emph{homotopy colimits}.
If $\cat{M}$ is a model category and $J$ is a small category and we can find
a model structure on the category of diagrams $\cat{M}^J$ \st{}
the colimit functor $\colim_J \from \cat{M}^J \to \cat{M}$
is a left Quillen functor (i.e.\ the left part of a Quillen adjoint pair),
then we can define the associated homotopy colimit functor
as the \emph{left derived functor} of $\colim_J$.
Dually, homotopy limit functors can be defined as the right derived functors
of classical limit functors.
This is achieved by replacing ill-behaved diagrams by better ones,
i.e.\ their (co)fibrant replacements, before applying (co)limit functors.
Contributions to the theory of model categories made by various authors
are far too numerous to be listed here.
Let us just recommend \cite{hi}, \cite{ho} and \cite{jo}*{Appendix E}
as general references.

Even though model categories are very versatile it was not long until
mathematicians realized that not every theory with homotopical content
fits easily into this framework.
K. Brown \cite{br} was the first to propose an alternative approach,
namely \emph{categories of fibrant objects}
(which will be referred to as \emph{fibration categories} in this paper).
Brown observed that the abstract notions of cofibrations and fibrations
remain to be useful under a weaker axiomatization
than the one used to define model categories.\footnote{
Brown's motivating example was the homotopy theory of sheaves of spectra.
A model category presenting this homotopy theory was eventually constructed
in \cite{ja}.}
A fibration category is a category equipped with two classes of morphisms:
weak equivalences and fibrations subject to conditions that follow from
the axioms of a model category but are, in fact,
satisfied by a larger class of examples as discussed in \Cref{sec:examples}.
There is, of course, the dual theory of \emph{cofibration categories} and
this is the notion that we will concentrate on throughout most of this paper.
Moreover, so called \emph{exact functors} are a counterpart to Quillen functors
and it is still possible to construct homotopy colimit functors
as left derived functors in the case of cofibration categories
and dually for fibration categories.
(The construction is similar to but not quite the same as for model categories
as explained in \Cref{sec:hocolims-and-diagrams,sec:infinite-hocolims}.)
Cofibration and fibration categories never became nearly as popular as
model categories, but since they were first introduced a number of contributions
has been made by, among the others, Anderson \cite{a},
Baues \cites{ba-ah,ba-cf}, Cisinski \cite{c-cd}
and R\u{a}dulescu-Banu \cite{rb}.
Moreover, Waldhausen \cite{wa} introduced a closely related notion of
a \emph{category with cofibrations and weak equivalences}
(nowadays usually called a \emph{Waldhausen category}) for the purpose of
developing a general framework for algebraic K-theory.
Subsequently, a close connection to abstract homotopy theory was made
by Cisinski \cite{c-ik}.

It is also worth pointing out that more approaches in a similar spirit
are possible.
For example, in 1995 Thomason \cite{we} introduced a modification of the notion
of a model category that addressed certain technical shortcoming\footnote{
This shortcoming is that it is not known in general how to construct
a model structure on the category of diagrams in a model category.
(Co)fibration categories also alleviate this problem to some extent as discussed
in \Cref{sec:hocolims-and-diagrams}.}
of Quillen's original axioms.

While abstract homotopy theory in the spirit of Quillen's homotopical algebra
was being developed throughout the years, an important conceptual progress
has been made by realizing that in model categories
(and other similar structures) all the homotopical information is contained in
the class of weak equivalences and the remaining structure
plays only an auxiliary role.
A \emph{relative category} is a category equipped with a class of morphisms,
called weak equivalences, subject to no special conditions other than
being closed under composition and containing all the identities.
The first important contribution to the theory of relative categories was made
by Gabriel and Zisman \cite{gz} who introduced a useful method of constructing
the homotopy category of a (nice enough) relative category called
the \emph{calculus of fractions}.
This method is an important motivation for the central construction
of this paper as explained on p.\ \pageref{fractions-motivation}.
Later, Dwyer and Kan \cites{dk1,dk2,dk3} defined
the \emph{simplicial localization} of an arbitrary relative category $\cat{C}$,
i.e.\ certain simplicial category $L \cat{C}$ that enhances
the homotopy category of $\cat{C}$ in the sense that
$\pi_0 L \cat{C} \iso \Ho \cat{C}$.
They also verified that if $\cat{C}$ carries a model structure,
then the mapping spaces obtained this way are weakly equivalent to
the mapping spaces coming from the model structure
via so called \emph{framings}.
Thus they have indeed demonstrated that all the homotopical content of
a model category is contained in its weak equivalences.
This statement was made into a sharp result
(that will be later stated more precisely) by Barwick and Kan \cite{bk1}.
Morphisms of relative categories are \emph{relative functors},
i.e.\ functors that preserve weak equivalences, but this formalism
is not structured enough to yield a reasonable theory of derived functors.
However, \emph{homotopical categories} were introduced in \cite{dhks}
as relative categories satisfying the ``2 out of 6 property''
where it was observed that they are much better behaved than
general relative categories.
In fact, it is possible to use homotopical categories as an abstract framework
for derived functors, but constructing derived functors still requires using
richer structures of homotopical algebra.

\subsection*{Higher category theory: modern models of homotopy theories}

Since Quillen introduced homotopical algebra, a completely new approach
to abstract homotopy theory has been invented
coming from higher category theory.
It would be unrealistic to adequately summarize the history
of higher category theory here.
We will only briefly mention the aspects most
relevant to the topic at hand.
A broader historical perspective can be found in \cite{si}*{Chapter 1}
and concise mathematical overviews in \cite{be} and \cite{p}.

Informally speaking, a \emph{higher category} is a category-like structure
that, in addition to objects and morphisms between them, has $2$-morphisms between
morphisms, $3$-morphisms between $2$-morphisms etc., possibly ad infinitum.
Moreover, these higher morphisms are equipped with composition operations
which are associative but only in a weak sense, i.e.\ up to natural equivalences
specified by higher morphisms.
Making this casual description into a precise definition is a big challenge
which is still not resolved in full generality.

Fortunately, in abstract homotopy theory we are not forced to consider
arbitrary higher categories but only so called \emph{$(\infty,1)$-categories},
i.e.\ the ones were all morphisms above dimension $1$ are weakly invertible.
Such structures can serve as models of homotopy theories
where we think of objects as homotopy types in a given homotopy theory,
morphisms as maps of these homotopy types,
$2$-morphisms as homotopies between maps
and higher morphisms as higher homotopies.
One of the most important reasons why it should be fruitful
to think of homotopy theories in terms of higher category theory is that
it should provide a good framework for stating universal properties of
various homotopy theoretic constructions (e.g.\ homotopy colimits)
which are difficult to express in the language of homotopical algebra.
A result of Barwick and Kan discussed in the next subsection demonstrates
that $(\infty,1)$-categories indeed capture
the classical notion of a homotopy theory.
The problem of formalizing the notion of an $(\infty,1)$-category
has been solved in multiple ways,
we will mention a few of the most notable ones.

The best developed notion of an $(\infty,1)$-category
(and the one used in this paper) is that of a \emph{quasicategory}.
It was introduced by Boardman and Vogt in \cite{bv} under the name
\emph{simplicial set satisfying the restricted Kan condition}.
The original purpose of this definition was to provide a good context for
the treatment of homotopy coherent diagrams as was done by Vogt \cite{v73}
and Cordier and Porter \cites{c,cp86,cp97}.
However, it took quite a long time before the full potential of quasicategories
was realized mostly by Joyal and Lurie in the work culminating in \cite{jo}
and \cite{l}.
In \Cref{ch:qcats} we give a brief treatment of the basic theory
of quasicategories.
One of the crucial advantages of quasicategories is that they make it easy
to state universal properties of homotopy colimits.
Informally, a homotopy colimit of a diagram in an $(\infty,1)$-category
should be given as a \emph{universal cone}, i.e.\ a cone \st{}
the \emph{mapping space} into any other cone is contractible.
Using quasicategories, this definition can be formalized in a practical way
as explained in \Cref{sec:colimits}.

Another early definition of $(\infty,1)$-categories was via
\emph{simplicially enriched categories} (or \emph{simplicial categories})
although it was not initially presented as such.
Simplicial categories were considered by Dwyer and Kan \cites{dk1,dk2}
as a part of their work on simplicial localization mentioned above,
but it was not until much later when Bergner \cite{be-mcsc} established
simplicial categories as models of $(\infty,1)$-categories.
This may seem rather surprising at the first glance since simplicial categories
come with strict composition operations.
However, as it turns out, when seen from the correct homotopical perspective
these strict composition operations already represent all possible
``weak composition operations''.
A drawback of this approach is that, unlike quasicategories,
simplicial categories make it difficult to express universal properties
of homotopy colimits and other homotopy theoretic constructions.
In fact, such difficulties could be seen as motivations for the development of
the theory of homotopy coherent diagrams using quasicategories
cited in the previous paragraph.

As an attempt to rectify the problem of composition operations
of simplicial categories being too strict, Dwyer, Kan and Smith \cite{dks}
introduced \emph{Segal categories} (but they did not give them a name).
Roughly speaking, a Segal category is a category ``weakly enriched''
in simplicial sets.
The theory of Segal categories and their generalizations
was developed extensively by Hirschowitz and Simpson \cite{hs}.
A comprehensive exposition can be found in \cite{si}.

Segal categories are more flexible than simplicial categories.
However, they are not quite as flexible as one could hope
and the difficulties can be traced to the fact that
the underlying $\infty$-groupoid of an $(\infty,1)$-category
is not easily accessible from its presentation as a Segal category.
A modified approach has been proposed by Rezk \cite{r} who defined
\emph{complete Segal spaces} where the underlying $\infty$-groupoid is
explicitly built into the structure of an $(\infty,1)$-category.
The theory of complete Segal spaces has various advantages, e.g.\ it is
presented by a model category (see the next subsection)
with unusually good properties compared to other models.
It is also suitable for internalizing into homotopy theories other than
the homotopy theory of spaces.

The original problem of the lack of a precise mathematical definition
of an $(\infty,1)$-category has been replaced by the problem of
having too many such definitions all of which look equally reasonable.
However, the multitude of notions of higher categories is not really a problem
since they have different advantages.
Simplicial categories and Segal categories serve as sources of examples
which may not be easy to construct directly as quasicategories
or complete Segal spaces which in turn provide good contexts for carrying out
higher categorical arguments.

\subsection*{The homotopy theory of homotopy theories}

We have argued that the abundance of notions of $(\infty,1)$-categories
can be helpful provided that we can properly address the question
of comparison between various definitions.
As it turns out, abstract homotopy theory itself provides a framework
for such comparisons.
The homotopy theories of each of the four types of $(\infty,1)$-categories
discussed above have been described as model categories.
(Which typically means that these models have been exhibited as fibrant objects
of a model category.)
This was done by Joyal for quasicategories \cite{jo},
by Bergner for simplicial categories \cite{be-mcsc},
by Hirschowitz and Simpson for Segal categories \cite{hs}
and by Rezk for complete Segal spaces \cite{r}.
It was subsequently proven that all these model categories
are \emph{Quillen equivalent}, i.e.\ that they present the same homotopy theory
which we call the \emph{homotopy theory of $(\infty,1)$-categories}.
Quillen equivalences between simplicial categories, Segal categories
and complete Segal spaces were established by Bergner \cite{be-3m}.
Moreover, Joyal and Tierney \cite{jt} constructed a Quillen equivalence
(two different ones, in fact) between quasicategories and complete Segal spaces.

Since we introduced $(\infty,1)$-categories as models of homotopy theories,
this leads us to consider the ``homotopy theory of homotopy theories''.
However, even though we already know that various definitions of
an $(\infty,1)$-category encode the same notion of a homotopy theory,
the two occurrences of ``homotopy theory'' in the phrase above
still have seemingly different meanings.

In order to address this issue we recall from the preceding discussion
that the actual content of the model categories above depends on
the notions of their weak equivalences and not on the model structures as such.
This means that in order to talk about ``homotopy theory of homotopy theories''
we have to fix a notion of equivalence of homotopy theories.
Dwyer and Kan \cite{dk3} proved that a Quillen functor
between model categories is a Quillen equivalence \iff{} it induces
an equivalence of their homotopy categories and weak homotopy equivalences
of the mapping spaces in their simplicial localizations
(i.e.\ it is a \emph{Dwyer--Kan equivalence} in the modern language).
By combining these observations we arrive at the conclusion that
if we want to think of model categories or relative categories
as homotopy theories they always have to be accompanied by the notions
of Quillen equivalences or Dwyer--Kan equivalences.
(Similarly, we will define weak equivalences of cofibration categories
in \Cref{ch:cof-cats}.)

This means that there is a way of giving the same meaning to both occurrences
of ``homotopy theory'' in the phrase ``homotopy theory of homotopy theories'',
namely, by interpreting it as the ``relative category of relative categories''
with Dwyer--Kan equivalences as weak equivalences.
Moreover, it is now a well posed question whether this notion of homotopy theory
is equivalent to the higher categorical ones.
Namely, we can ask whether the underlying relative category of any of
the four model categories above is Dwyer--Kan equivalent to
the relative category of relative categories.
This is indeed true by the result of Barwick and Kan \cites{bk1,bk2}.
More precisely, they constructed a model structure on the category of
relative categories and proved that it is Quillen equivalent to
the Rezk model structure for complete Segal spaces.

All these considerations suggest that we should be able to talk about
the ``$(\infty,1)$-category of $(\infty,1)$-categories'' as an alternative
to the ``homotopy theory of homotopy theories''.
This is indeed possible and leads to a very interesting result that
the ``$(\infty,1)$-category of $(\infty,1)$-categories''
can be characterized axiomatically.
This was first done by To\"{e}n \cite{t} in the language of homotopical algebra.
Namely, he gave sufficient conditions for a model category
to be Quillen equivalent to the Rezk model category for complete Segal spaces.
Later, Barwick and Schommer-Pries \cite{bsp} formulated
an alternative axiomatization purely in the language of higher category theory.
(In fact, their theory applies to $(\infty,n)$-categories,
i.e.\ the ones where morphisms are only required to be weakly invertible above
a fixed finite dimension $n$.)

\subsection*{New results}

Just as different notions of $(\infty,1)$-categories have different advantages,
higher category theory as such has different advantages
than homotopical algebra.
A good exemplification of these differences is the way both theories approach
homotopy invariant constructions such as homotopy colimits.
In higher category theory we define them via universal properties,
but such definitions do not address the problem of
actually constructing homotopy colimits and it seems that every proof
of cocompleteness of an $(\infty,1)$-category reduces in one way or another to
homotopical algebra.
On the other hand, while homotopical algebra provides useful tools for
explicit constructions of homotopy colimits, it makes it next to impossible to
talk about their universal properties.
Thus both approaches play important and complementary roles
in abstract homotopy theory.

The state of affairs presented above does not explain how homotopical algebra
(which we can now understand as structured theory of relative categories)
fits into the context of higher category theory.
The purpose of this paper is to solve this very problem.

It should be apparent that while general relative categories present
a wide variety of homotopy theories (in fact all of them),
model categories and cofibration categories only present
some special homotopy theories, i.e.\ the ones having some specific properties
(or perhaps equipped with some specific structure).
One of the main results of this paper is that the homotopy theories presented
by cofibration categories are precisely the cocomplete ones.
Similar remarks apply to morphisms of homotopy theories.
As mentioned, each of the notions discussed above has associated with it
a natural notion of a morphism: Quillen functors for model categories,
exact functors for cofibration categories
and relative functors for relative categories.
Again, relative functors present arbitrary morphism of homotopy theories,
but Quillen functors and exact functors are more special.
In this paper we prove that exact functors between cofibration categories
correspond to homotopy colimit preserving morphisms
of cocomplete homotopy theories.

It is important to realize that the comparison of homotopical algebra
to higher category theory is an entire family of problems, one for each notion
of homotopical algebra.
That is because different notions will present different types of
homotopy theories, e.g.\ in contrast to cofibration categories homotopy theories
presented by model categories are both complete and cocomplete.
This paper addresses only the case of cofibration categories
(and dually fibration categories) and does not seem to apply to
model categories.
However, our individual techniques are potentially useful even
in the theory of model categories.

The main result is that the homotopy theory of cofibration categories
is equivalent to the homotopy theory of cocomplete quasicategories.
The examples of equivalences of homotopy theory discussed so far suggest that
while model categories and Quillen equivalences
do not carry more homotopical information than relative categories
and Dwyer--Kan equivalences, it is usually much easier
to exploit homotopical algebra to construct Quillen equivalences rather than
construct Dwyer--Kan equivalences by hand.
Unfortunately, the categories of cofibration categories
and cocomplete quasicategories do not carry model structures
(e.g.\ since they have no initial objects).
We will circumvent this problem by showing that
they are both fibration categories.

In \Cref{ch:cof-cats} we introduce cofibration categories and summarize
the well known techniques of homotopical algebra that will be use throughout
this thesis.
We introduce morphisms and weak equivalences of cofibration categories
which specifies the homotopy theory of cofibration categories.
Then we define fibrations of cofibration categories and prove that
they make the category of (small) cofibration categories into
a fibration category.
Finally, we discuss some basic techniques of constructing fibrations
and weak equivalences of cofibration categories
and we mention some examples which demonstrate versatility of this approach
to homotopical algebra.

\Cref{ch:qcats} contains the basic theory of quasicategories
which is mostly cited from \cite{jo} and \cite{ds}.
In particular, we establish fibration categories of quasicategories
and of cocomplete quasicategories.
This section contains no new results, except possibly for the existence of
the latter fibration category.
(The completeness of the homotopy theory of cocomplete quasicategories
is discussed in \cite{l}, but it is not stated in terms of
fibration categories.)

We start \Cref{ch:qcats-of-frames} by constructing a functor
from cofibration categories to cocomplete quasicategories.
To each cofibration category $\cat{C}$ we associate a nerve-like simplicial set
denoted by $\nf \cat{C}$ and called
the \emph{quasicategory of frames in $\cat{C}$}.
(The letter $\mathrm{f}$ in $\nf$ stands either for \emph{frames}
since those are the objects in $\nf \cat{C}$ or for \emph{fractions}
since the morphisms in $\nf \cat{C}$ are certain generalizations
of left fractions.)
The first step in the proof of the main theorem is to show that $\nf$
is an exact functor between the fibration categories mentioned above.
(And in particular that it takes values in cocomplete quasicategories
since it is not apparent from the definition.)
This proof is somewhat involved
and occupies the entire \Cref{ch:qcats-of-frames}.

The second step, presented in \Cref{ch:cofcats-of-dgrms},
is to prove that $\nf$ is a weak equivalence of fibration categories.
To this end we associate with every cocomplete quasicategory $\qcat{D}$
a cofibration category $\dgrm \qcat{D}$ called
the \emph{category of diagrams in $\qcat{D}$}.
This yields a functor $\dgrm$ which is not exact
but is an inverse to $\nf$ up to weak equivalence.
This suffices to conclude that $\nf$ is an equivalence of homotopy theories.

We should explain that parts of the arguments outlined above depend on
certain set theoretic assumptions.
Most of the results are parametrized by a regular cardinal number $\kappa$
and concern small $\kappa$-cocomplete cofibration categories
and small $\kappa$-cocomplete quasicategories, i.e.\ the ones admitting
$\kappa$-small (homotopy) colimits.
We will suppress this parameter as much as possible, but there are situations
where referring to it is unavoidable.
In the first two and a half chapters we set $\kappa = \aleph_0$,
i.e.\ we consider finitely cocomplete homotopy theories.
This is done merely to simplify the exposition,
the arguments for $\kappa > \aleph_0$ require only minor modifications
which are explained in \Cref{sec:infinite-hocolims}.
However, from this point on the distinction between these two cases starts
playing a significant role.
As it turns out, the case of $\kappa > \aleph_0$ is much easier
for technical reasons discussed in the beginning
of \Cref{sec:cocompleteness-finite}.
The rest of \Cref{ch:qcats-of-frames} is split into
\Cref{sec:cocompleteness-infinite} which deals with $\kappa > \aleph_0$
and \Cref{sec:cocompleteness-finite} which deals with $\kappa = \aleph_0$.
Similarly, the main part of \Cref{ch:cofcats-of-dgrms} is split into
\Cref{sec:proof-infinite} which deals with $\kappa > \aleph_0$
and \Cref{sec:proof-finite} which deals with $\kappa = \aleph_0$.
The reader is encouraged to read the arguments for $\kappa > \aleph_0$ first.

We work only with small cofibration categories and quasicategories
and do not explicitly mention Grothendieck universes, but it is easy
to interpret all the results in any higher universe of interest.
It suffices to fix a Grothendieck universe $\mathcal{U}$
with $\kappa \in \mathcal{U}$ and substitute ``$\mathcal{U}$-small''
for ``small''.
The only non-$\mathcal{U}$-small categories under consideration
are the categories of
$\mathcal{U}$-small $\kappa$-cocomplete cofibration categories,
of $\mathcal{U}$-small quasicategories
and of $\mathcal{U}$-small $\kappa$-cocomplete quasicategories.
They can be taken to be $\mathcal{V}$-small
for some larger universe $\mathcal{V}$ if desirable.

  \section*{Acknowledgments}

This paper is a version of my thesis \cite{sz} which was written
while I was a doctoral student
in \emph{Bonn International Graduate School in Mathematics}
and, more specifically,
\emph{Graduiertenkolleg 1150 ``Homotopy and Cohomology''}
and \emph{International Max Planck Research School for Moduli Spaces}.
I want to thank everyone involved for creating an excellent working
environment.\footnote{Moreover,
this material is partially based upon work supported by
the National Science Foundation under Grant No. 0932078 000 while the author
was in residence at the Mathematical Sciences Research Institute in Berkeley,
California, during the Spring 2014 semester.
}

I want to thank Clark Barwick, Bill Dwyer, Andr\'{e} Joyal, Chris Kapulkin,
Lennart Meier, Thomas Nikolaus, Chris Schommer-Pries, Peter Teichner
and Marek Zawadowski for conversations on various topics
which were very beneficial to my research.

I am especially grateful to Viktoriya Ozornova and Irakli Patchkoria
for reading an early draft of my thesis.
Their feedback helped me make many improvements and avoid numerous errors.

Above all, I want to express my gratitude to my supervisor Stefan Schwede
whose expertise was always invaluable and without whose support
this thesis could not have been written.

  \section{Cofibration categories}
  \label{ch:cof-cats}
  We start this section by introducing cofibration categories.
The definition stated here is almost the same as (the dual of)
Brown's original definition \cite{br}*{p.\ 420}.
(What he called \emph{categories of fibrant objects}
we call fibration categories.)
We do not commit much space to the discussion of basic properties
of cofibration categories, we refer the reader to \cite{rb} for these.
Instead, the purpose of this section is to establish the homotopy theory
of cofibration categories in the form of a fibration category.
This means that we will consider the category of cofibration categories
with exact functors as morphisms and we will define weak equivalences
and fibrations in this category and verify that they satisfy the duals
of the axioms given below.

\subsection{Definitions and basic properties}

\begin{definition}
  A \emph{cofibration category} is a category $\cat{C}$
  equipped with two subcategories: the subcategory of \emph{weak equivalences}
  (denoted by $\weto$) and the subcategory of \emph{cofibrations}
  (denoted by $\cto$) \st{} the following axioms are satisfied.
  (Here, an \emph{acyclic cofibration} is a morphism that is
  both a weak equivalence and a cofibration.)
  \begin{itemize}
  \item[(C0)] \index{2 out of 6} Weak equivalences satisfy
    the ``2 out of 6'' property, i.e.\ if
    \begin{ctikzpicture}
      \matrix[diagram]
      {
        |(W)| W & |(X)| X & |(Y)| Y & |(Z)| Z \\
      };

      \draw[->] (W) to node[above] {$f$} (X);
      \draw[->] (X) to node[above] {$g$} (Y);
      \draw[->] (Y) to node[above] {$h$} (Z);
    \end{ctikzpicture}
    are morphisms of $\cat{C}$ \st{} both $g f$ and $h g$ are weak equivalences,
    then so are $f$, $g$ and $h$ (and thus also $h g f$).
  \item[(C1)] Every isomorphism of $\cat{C}$ is an acyclic cofibration.
  \item[(C2)] An initial object exists in $\cat{C}$.
  \item[(C3)] Every object $X$ of $\cat{C}$ is cofibrant,
    i.e.\ if $\init$ is the initial object of $\cat{C}$,
    then the unique morphism $\init \to X$ is a cofibration.
  \item[(C4)] Cofibrations are stable under pushouts along arbitrary morphisms
    of $\cat{C}$ (in particular these pushouts exist in $\cat{C}$).
    Acyclic cofibrations are stable under pushouts along arbitrary morphisms
    of $\cat{C}$.
  \item[(C5)] Every morphism of $\cat{C}$ factors as a composite
    of a cofibration followed by a weak equivalence.
  \end{itemize}
\end{definition}

Definitions of (co)fibration categories found throughout the literature
vary in details.
Since we use \cite{rb} as our main source we point out that in the terminology
of this paper the definition above corresponds to ``precofibration categories
with all objects cofibrant and the ``2 out of 6'' property''.
Comparisons to other definitions can be found in \cite{rb}*{Chapter 2}.

The above axioms describe \emph{finitely cocomplete} cofibration categories.
Here, cocompleteness really means ``homotopy cocompleteness''
since cofibration categories do not necessarily have all finite strict colimits,
but they have all finite homotopy colimits.
Their construction will be discussed in \Cref{sec:hocolims-and-diagrams}.
If we want to consider cofibration categories with more homotopy colimits
we need to assume some extra axioms which will be discussed
in \Cref{sec:infinite-hocolims}.

Cofibration categories can be seen as generalizations of model categories.
Namely, if $\cat{M}$ is a model category, then its full subcategory
of cofibrant objects $\cat{M}_{\mathrm{cof}}$ with weak equivalences
and cofibrations inherited from $\cat{M}$ satisfies the above axioms.
Many of the standard tools of homotopical algebra (that do not refer
to fibrations, e.g.\ left homotopies, cofiber sequences or homotopy colimits)
depend only on these axioms and hence are available for cofibration categories,
although they sometimes differ in technical details.
These techniques are discussed in great detail in \cite{rb}.
There are examples of (co)fibration categories that do not come
from model categories.
Some of those are presented in \Cref{sec:examples}.

Before discussing new results about homotopy theory of cofibration categories,
we collect some preliminaries, mostly following \cite{rb}.
We fix a cofibration category $\cat{C}$.

\begin{definition}
  \enumhack
  \begin{enumerate}[label=\dfn]
  \item A \emph{cylinder} of an object $X$ is a factorization
    of the codiagonal morphism $X \coprod X \to X$
    as $X \coprod X \cto IX \weto X$.
  \item A \emph{left homotopy} between morphisms $f, g \from X \to Y$
    via a cylinder $X \coprod X \cto IX \weto X$ is a commutative square
    of the form
    \begin{ctikzpicture}
      \matrix[diagram]
      {
        |(X)|  X \coprod X & |(Y)| Y \\
        |(IX)| IX          & |(Z)| Z \text{.} \\
      };

      \draw[->]  (X) to node[above] {$[f, g]$} (Y);
      \draw[cof] (Y) to node[right] {$\we$}    (Z);

      \draw[cof] (X)  to (IX);
      \draw[->]  (IX) to (Z);
    \end{ctikzpicture}
  \item Morphisms $f, g \from X \to Y$ are \emph{left homotopic}
    (notation: $f \htp_l g$) if there exists a left homotopy between them
    via some cylinder on $X$.
  \end{enumerate}
\end{definition}

The definition of left homotopies differs from the standard definition
as usually given in the context of model categories
where the morphism $Y \acto Z$ is required to be the identity.
This modification is dictated by the lack of fibrant objects
in cofibration categories and makes the definition well-behaved
for arbitrary $Y$ while the standard definition in a model category
is only well-behaved for a fibrant $Y$.

We denote the homotopy category of $\cat{C}$ (i.e.\ its localization \wrt{}
weak equivalences) by $\Ho \cat{C}$ and for a morphism $f$ of $\cat{C}$
we write $[f]$ for its image under the localization functor
$\cat{C} \to \Ho \cat{C}$.
The homotopy category can be constructed in two steps: first dividing out
left homotopies and then applying the calculus of fractions.

\begin{proposition}
  The relation of left homotopy is a congruence on $\cat{C}$.
  Moreover, every morphism of $\cat{C}$ that becomes
  an isomorphism in $\cat{C} \quot \mathord{\htp_l}$ is a weak equivalence.
  Thus left homotopic morphisms become equal in $\Ho \cat{C}$
  and $\cat{C} \quot \mathord{\htp_l}$ comes equipped with a canonical functor
  $\cat{C} \quot \mathord{\htp_l} \to \Ho \cat{C}$.
\end{proposition}

\begin{proof}
  The first statement is \cite{rb}*{Theorem 6.3.3(1)}.
  The remaining ones follow by straightforward ``2 out of 3'' arguments.
\end{proof}

The next theorem is a crucial tool in the theory of cofibration categories
and can be used to verify many of their fundamental properties.
It says that up to left homotopy all cofibration categories satisfy
the left calculus of fractions in the sense of Gabriel and Zisman
\cite{gz}*{Chapter I}.
This fact was first proven by Brown \cite{br}*{Proposition I.2} and can be seen
as an abstraction of the classical construction of the derived category
of a ring, see e.g.\ \cite{gm}*{Theorem III.4.4}.
In general, constructing $\Ho \cat{C}$ may involve using
arbitrarily long zig-zags
of morphisms in $\Ho \cat{C}$ and identifying them via arbitrarily long chains
of relations.
However, the previous proposition implies that
$\cat{C} \quot \mathord{\htp_l} \to \Ho \cat{C}$ is also a localization functor
and in that case \Cref{calculus-of-fractions} says that it suffices to consider
two-step zig-zags (called \emph{left fractions}) up to a much simplified
equivalence relation.
Our main construction, i.e.\ the \emph{quasicategory of frames}, can be seen
as an enhancement of the calculus of fractions as discussed
on p.\ \pageref{fractions-motivation}.

\begin{theorem}\label{calculus-of-fractions}
  A cofibration category $\cat{C}$ satisfies the left calculus of fractions
  up to left homotopy, i.e.
  \begin{enumerate}[label=\thm]
  \item\label{fraction} Every morphism $\phi \in \Ho \cat{C}(X, Y)$
    can be written as a left fraction $[s]^{-1} [f]$
    where $f \from X \to \tilde Y$ and $s \from Y \weto \tilde Y$ are morphisms
    of $\cat{C}$.
  \item\label{fractions-eq} Two fractions $[s]^{-1} [f]$ and $[t]^{-1} [g]$
    are equal in $\Ho \cat{C}(X, Y)$ \iff{} there exist weak equivalences
    $u$ and $v$ \st{}
    \begin{align*}
      u s \htp_l v t \text{ and } u f \htp_l v g \text{.}
    \end{align*}
  \item\label{composition} If $\phi \in \Ho \cat{C}(X, Y)$
    and $\psi \in \Ho \cat{C}(Y, Z)$ can be written as $[s]^{-1} [f]$
    and $[t]^{-1} [g]$ respectively and a square
    \begin{ctikzpicture}
      \matrix[diagram]
      {
        |(Y)|  Y        & |(tZ)| \tilde Z \\
        |(tY)| \tilde Y & |(hZ)| \hat Z \\
      };
      \draw[->] (Y)  to node[above] {$g$} (tZ);
      \draw[->] (tY) to node[below] {$h$} (hZ);

      \draw[->] (Y)  to node[left] {$s$} node[right] {$\we$} (tY);
      \draw[->] (tZ) to node[left] {$u$} node[right] {$\we$} (hZ);
    \end{ctikzpicture}
    commutes up to homotopy, then $\psi \phi$ can be written as $[ut]^{-1}[hf]$.
  \end{enumerate}
\end{theorem}

\begin{proof}
  Parts \ref{fraction} and \ref{fractions-eq} follow from
  \cite{rb}*{Theorem 6.4.4(1)} and \ref{composition} from
  the proof of \cite{rb}*{Theorem 6.4.1}.
\end{proof}

In order to define the homotopy theory of cofibration categories we first need
a good notion of a morphism between cofibration categories.
We will use \emph{exact} functors which (according to the definition
and the lemma below) are essentially homotopy invariant functors
that preserve basic finite homotopy colimits, i.e.\ initial objects
and homotopy pushouts.
It will follow from the discussion in \Cref{sec:hocolims-and-diagrams}
that they actually preserve all finite homotopy colimits.

\begin{definition}
  A functor $F \from \cat{C} \to \cat{D}$ between cofibration categories
  is \emph{exact} if it preserves cofibrations, acyclic cofibrations,
  initial objects and pushouts along cofibrations.
\end{definition}

Finally, we recall a standard method of verifying homotopy invariance
of functors between cofibration categories.

\index{K. Brown's Lemma}
\begin{lemma}[K. Brown's Lemma]\label{Ken-Brown}
  If a functor between cofibration categories sends acyclic cofibrations
  to weak equivalences, then it preserves all weak equivalences.
  In particular, exact functors preserve weak equivalences.
\end{lemma}

\begin{proof}
  The proof of \cite{ho}*{Lemma 1.1.12} works for cofibration categories.
  (See also the proof of \cite{br}*{Lemma 4.1}
  where this result first appeared.)
\end{proof}

  \subsection{Homotopy theory of cofibration categories}

We are now ready to introduce the homotopy theory of cofibration categories.
For this it is sufficient to define a class of weak equivalences in the category
of cofibration categories which is what we will do next.
Later, we will proceed to define fibrations of cofibration categories
and prove that they satisfy the axioms of a fibration category
which will give us a solid grasp of the homotopy theory
of cofibration categories.

\begin{definition}
  An exact functor $F \from \cat{C} \to \cat{D}$ is a \emph{weak equivalence}
  if it induces an equivalence $\Ho \cat{C} \to \Ho \cat{D}$.
\end{definition}
\begin{anfxnote}{naming}
  is there a more descriptive name?
\end{anfxnote}

This notion is closely related to the \emph{Waldhausen approximation properties}
first formulated by Waldhausen as criteria for an exact functor to induce
an equivalence of the algebraic K-theory spaces \cite{wa}*{Section 1.6}.
Later, Cisinski showed that an exact functor satisfies (slightly reformulated)
Waldhausen approximation properties \iff{} it is a weak equivalence in the sense
of the definition above, see \Cref{App}.

It is far from obvious that weak equivalences preserve homotopy types
of homotopy mapping spaces.
This is indeed true by a theorem of Cisinski \cite{c-ik}*{Th\'eor\`eme 3.25}
which states that a weak equivalence induces an equivalence
of the \emph{hammock localizations} in the sense of Dwyer and Kan \cite{dk2}.
While this result will not be used in this paper, it justifies our choice
of weak equivalences of cofibration categories.
In fact, our main result implies that they correspond
to categorical equivalences of quasicategories and with some additional effort
this could be used to rederive Cisinski's theorem.

\begin{proposition}[\cite{c-cd}*{Th\'eor\`eme 3.19}]\label{App}
  An exact functor $F \from \cat{C} \to \cat{D}$ is a weak equivalence \iff{}
  it satisfies the following properties.
  \begin{itemize}
  \item[\emph{(App1)}] $F$ reflects weak equivalences.
  \item[\emph{(App2)}] Given a morphism $f \from F A \to Y$ in $\cat{D}$,
    there exists a morphism $i \from A \to B$ in $\cat{C}$
    and a commutative diagram
    \begin{ctikzpicture}
      \matrix[diagram]
      {
        |(A)| FA & |(Y)| Y \\
        |(B)| FB & |(Z)| Z \\
      };

      \draw[->] (A) to node[above] {$f$}  (Y);
      \draw[->] (A) to node[left]  {$Fi$} (B);

      \draw[->] (B) to node[below] {$\we$} (Z);
      \draw[->] (Y) to node[right] {$\we$} (Z);
    \end{ctikzpicture}
    in $\cat{D}$. \qed
  \end{itemize}
\end{proposition}

We are now ready to define fibrations of cofibration categories,
but before doing so we briefly explain the duality between cofibration
and fibration categories.
A \emph{fibration category} is a category $\cat{F}$ equipped with
subcategories of weak equivalences and \emph{fibrations} \st{}
$\cat{F}^\op$ is a cofibration category (where the fibrations of $\cat{F}$
become the cofibrations of $\cat{F}^\op$).
Similarly, an exact functor of fibration categories is a functor
that is exact as a functor of the corresponding cofibration categories.
As usual, all the results about cofibration categories readily dualize
to results about fibration categories.
We do not state them separately, but we point out that all the statements
in \cite{rb} are explicitly given in both versions.

\begin{definition}\label{dfn:fibration}
  Let $P \from \cat{E} \to \cat{D}$ be an exact functor
  of cofibration categories.
  \begin{enumerate}[label=\dfn]
  \item $P$ is an \emph{isofibration} if
    for every object $A \in \cat{E}$ and an isomorphism $g \from P A \to Y$
    there is an isomorphism $f \from A \to B$ \st{} $P f = g$.
  \item It is said to satisfy
    the \emph{lifting property for factorizations} if for any morphism
    $f \from A \to B$ of $\cat{E}$ and a factorization
    \begin{ctikzpicture}
      \matrix[diagram]
      {
        |(A)| PA &         & |(B)| P B \\
        & |(X)| X \\
      };

      \draw[->]  (A) to node[above]      {$Pf$} (B);
      \draw[cof] (A) to node[below left] {$j$}  (X);

      \draw[->] (X) to node[below right] {$t$} node[above left] {$\we$} (B);
    \end{ctikzpicture}
    there exists a factorization
    \begin{ctikzpicture}
      \matrix[diagram]
      {
        |(A)| A &         & |(B)| B \\
        & |(C)| C \\
      };

      \draw[->]  (A) to node[above]      {$f$} (B);
      \draw[cof] (A) to node[below left] {$i$} (C);

      \draw[->] (C) to node[below right] {$s$} node[above left] {$\we$} (B);
    \end{ctikzpicture}
    \st{} $Pi = j$ and $Ps = t$ (in particular, $PC = X$).
  \item It has
    the \emph{lifting property for pseudofactorizations} if for any morphism
    $f \from A \to B$ of $\cat{E}$ and a diagram
    \begin{ctikzpicture}
      \matrix[diagram]
      {
        |(A)| PA & |(B)| PB \\
        |(X)| X  & |(Y)| Y \\
      };

      \draw[->]  (A) to node[above] {$Pf$} (B);
      \draw[cof] (A) to node[left]  {$j$}  (X);

      \draw[->]  (X) to node[below] {$t$} node[above] {$\we$} (Y);
      \draw[cof] (B) to node[right] {$v$} node[left]  {$\we$} (Y);
    \end{ctikzpicture}
    there exists a diagram
    \begin{ctikzpicture}
      \matrix[diagram]
      {
        |(A)| A & |(B)| B \\
        |(C)| C & |(D)| D \\
      };

      \draw[->]  (A) to node[above] {$f$} (B);
      \draw[cof] (A) to node[left]  {$i$} (C);

      \draw[->]  (C) to node[below] {$s$} node[above] {$\we$} (D);
      \draw[cof] (B) to node[right] {$u$} node[left]  {$\we$} (D);
    \end{ctikzpicture}
    \st{} $Pi = j$, $Ps = t$ and $Pu = v$
    (in particular, $PC = X$ and $PD = Y$).
  \item We say that
    $P$ is a \emph{fibration} if it is an isofibration
    and satisfies the lifting properties
    for factorizations and pseudofactorizations.
  \end{enumerate}
\end{definition}

This definition can be restated in a more technical but convenient way.
We define a category $\eCofCat$ containing the category
of cofibration categories $\CofCat$ (whose morphisms are exact functors)
as a non-full subcategory.
Objects of $\eCofCat$ are small categories equipped with two subcategories:
the subcategory of weak equivalences and the subcategory of cofibrations \st{}
all identity morphisms are acyclic cofibrations.
Morphisms are functors that preserve both weak equivalences and cofibrations.
 
An exact functor between cofibration categories
is a fibration \iff{} it has the \rlp{}, as a morphism of $\eCofCat$,
\wrt{} the following functors.
\begin{itemize}
\item
  The inclusion of $[0]$ into $\wiso$ (the groupoid freely generated
  by an isomorphism $0 \to 1$).
\item The inclusion of $[1]$ (with only identities as weak equivalences
  or cofibrations) into
  \begin{ctikzpicture}
    \matrix[diagram]
    {
      |(0)| 0 &               & |(1)| 1 \text{.} \\
              & |(b)| \bullet \\
    };

    \draw[->]  (0) to (1);
    \draw[cof] (0) to (b);

    \draw[->] (b) to node[below right] {$\we$} (1);
  \end{ctikzpicture}
\item The inclusion of $[1] \times [0]$ (with only identities
  as weak equivalences or cofibrations) into
  \begin{ctikzpicture}
    \matrix[diagram]
    {
      |(00)| (0,0) & |(10)| (1,0) \\
      |(01)| (0,1) & |(11)| (1,1) \text{.} \\
    };

    \draw[->]  (00) to (10);
    \draw[cof] (00) to (01);

    \draw[->]  (01) to node[below] {$\we$} (11);
    \draw[cof] (10) to node[right] {$\we$} (11);
  \end{ctikzpicture}
\end{itemize}

In a few of the proofs in the remainder of this subsection we will refer forward
to \Cref{q-pullback-colimits,q-lifting-colimits}.
In \Cref{ch:qcats} they will be stated for quasicategories, but for now we will only use
their much simpler special cases for ordinary categories.

Let $\spn$ denote the poset of proper subsets of $\{ 0, 1 \}$.\newline

\begin{proposition}\label{pullback-of-cofcats}
  Let $F \from \cat{C} \to \cat{D}$ and $P \from \cat{E} \to \cat{D}$
  be exact functors between cofibration categories with $P$ a fibration.
  Then a pullback of $P$ along $F$ exists $\CofCat$.
\end{proposition}

\begin{proof}
  Form a pullback of $P$ along $F$ in the category of categories.
  \begin{ctikzpicture}
    \matrix[diagram]
    {
      |(P)| \cat{P} & |(E)| \cat{E} \\
      |(C)| \cat{C} & |(D)| \cat{D} \\
    };

    \draw[->] (P) to node[above] {$G$} (E);
    \draw[->] (P) to node[left]  {$Q$} (C);

    \draw[->] (C) to node[below] {$F$} (D);
    \draw[->] (E) to node[right] {$P$} (D);
  \end{ctikzpicture}
  Define a morphism $f$ of $\cat{P}$ to be a weak equivalence
  (respectively, a cofibration) if both $Gf$ and $Qf$ are weak equivalences
  (respectively, cofibrations). Then the above square becomes a pullback
  in $\eCofCat$.

  Now we check that $\cat{P}$ is a cofibration category.
  \begin{itemize}
  \item[(C0-1)] In $\cat{P}$ weak equivalences satisfy ``2 out of 6''
    and all isomorphisms are acyclic cofibrations
    since this holds in both $\cat{C}$ and $\cat{E}$.
  \item[(C2-3)] Let $\init_\cat{C}$ be an initial object of $\cat{C}$.
    By \Cref{q-lifting-colimits}
    there is an initial object $\init_\cat{E}$ of $\cat{E}$
    \st{} $P \init_\cat{E} = F \init_\cat{C}$.
    Then $(\init_\cat{C}, \init_\cat{E})$ is an initial object of $\cat{P}$
    by \Cref{q-pullback-colimits}.
    Moreover, every object of $\cat{P}$ is cofibrant since this holds
    in both $\cat{C}$ and $\cat{E}$.
  \item[(C4)] Let $X \from \spn \to \cat{P}$ be a span
    with $X_\emptyset \to X_0$ a cofibration.
    Let $S$ be a colimit of $QX$ in $\cat{C}$, then $FS$ is a colimit
    of $FQX = PGX$ in $\cat{D}$ since $F$ is exact.
    \Cref{q-lifting-colimits} implies that we can choose a colimit $T$ of $GX$
    in $\cat{E}$ so that $PT = FS$.
    Then it follows by \Cref{q-pullback-colimits} that $(S,T)$ is a colimit
    of $X = (QX, GX)$ in $\cat{P}$.
    Thus pushouts along cofibrations exist in $\cat{P}$ and both cofibrations
    and acyclic cofibrations are stable under pushouts since this holds
    in both $\cat{C}$ and $\cat{E}$.
  \item[(C5)] Let $f \from A \to B$ be a morphism of $\cat{P}$.
    Pick a factorization of $Qf$ as
    \begin{align*}
      QA \cto C \weto QB
    \end{align*}
    in $\cat{C}$.
    Then $FQf = PGf$ factors as
    \begin{align*}
      PGA = FQA \cto FC \weto FQB = PGB
    \end{align*}
    and we can lift this factorization to a factorization of $Gf$ as
    \begin{align*}
      GA \cto E \weto GB \text{.}
    \end{align*}
    It follows that
    \begin{align*}
      A = (QA, GA) \cto (C, E) \weto (QB, GB) = B
    \end{align*}
    is a factorization of $f$. This completes the verification that $\cat{P}$
    is a cofibration category.
  \end{itemize}

  Next, we need to verify that $Q$ and $G$ are exact.
  They preserve cofibrations and acyclic cofibrations by the definition
  of cofibrations and weak equivalences in $\cat{P}$.
  They also preserve initial objects and pushouts along cofibrations
  by the construction of these colimits in $\cat{P}$.

  It remains to see that the square we constructed is a pullback
  in the category of cofibration categories, i.e.\ that given a square
  \begin{ctikzpicture}
    \matrix[diagram]
    {
      |(F)| \cat{F} & |(E)| \cat{E} \\
      |(C)| \cat{C} & |(D)| \cat{D} \\
    };

    \draw[->] (F) to (E);
    \draw[->] (F) to (C);

    \draw[->] (C) to node[below] {$F$} (D);
    \draw[->] (E) to node[right] {$P$} (D);
  \end{ctikzpicture}
  of cofibration categories and exact functors, the induced functor
  $\cat{F} \to \cat{P}$ is also exact.
  Indeed, it was already observed that it preserves cofibrations
  and acyclic cofibrations.
  It also preserves initial objects and pushouts along cofibrations
  by \Cref{q-pullback-colimits}.
\end{proof}

The next proposition will imply the stability of acyclic fibrations
under pullbacks.
Moreover, in later sections it will serve as a useful criterion for verifying
that an exact functor is a weak equivalence.
Observe that the lifting property for pseudofactorizations is needed only here,
it was not used in the proof of the previous proposition.

\begin{proposition}\label{acyclic-fibrations}
  An exact functor $P \from \cat{C} \to \cat{D}$ is an acyclic fibration \iff{}
  it is a fibration, satisfies \emph{(App1)} and the \rlp{} (in $\eCofCat$)
  \wrt{} the inclusion of $[0]$ into
  \begin{ctikzpicture}
    \matrix[diagram]
    {
      |(0)| 0 & |(1)| 1 \text{.} \\
    };

    \draw[cof] (0) to (1);
  \end{ctikzpicture}
\end{proposition}

\begin{proof}
  First assume that $P$ satisfies the properties above.
  We need to check that it satisfies (App2).
  Let $f \from PA \to Z$ be a morphism of $\cat{D}$.
  Factor $f$ as a composite of $j \from PA \cto Y$ and $Y \weto Z$
  and apply the lifting property above to find a cofibration $i \from A \cto B$
  \st{} $Pi = j$.
  This yields a diagram
  \begin{ctikzpicture}
    \matrix[diagram]
    {
      |(A)| PA & |(Z0)| Z \\
      |(B)| PB & |(Z1)| Z \text{.} \\
    };

    \draw[->] (A) to node[above] {$f$}  (Z0);
    \draw[->] (A) to node[left]  {$Pi$} (B);

    \draw[->] (B)  to node[below] {$\we$}   (Z1);
    \draw[->] (Z0) to node[right] {$\id_Z$} (Z1);
  \end{ctikzpicture}

  Conversely, assume that $P$ is an acyclic fibration.
  We need to check that it satisfies the lifting property above.
  Consider a cofibration $j \from PA \cto Y$ and apply (App2) to it
  to get $f \from A \to B$ and a diagram
  \begin{ctikzpicture}
    \matrix[diagram]
    {
      |(A)| PA & |(Y)| Y \\
      |(B)| PB & |(Z)| Z \\
    };

    \draw[cof] (A) to node[above] {$j$}  (Y);
    \draw[->]  (A) to node[left]  {$Pf$} (B);

    \draw[->] (B) to node[below] {$t$} (Z);
    \draw[->] (Y) to node[right] {$s$} (Z);
  \end{ctikzpicture}
  with both $s$ and $t$ weak equivalences.
  We factor $[t, s] \from PB \push_{PA} Y \to Z$ as a composite
  of $[t', s'] \from PB \push_{PA} Y \cto W$ and $W \weto Z$.
  So we obtain the square on the right
  \begin{ctikzpicture}
    \matrix[diagram]
    {
      |(PA)| PA & |(PB)| PB & |(A)| A & |(B)| B \\
      |(Y)| Y   & |(W)|  W  & |(C)| C & |(D)| D \\
    };

    \draw[->]  (PA) to node[above] {$Pf$}  (PB);
    \draw[cof] (PA) to node[left]  {$j$}   (Y);

    \draw[->]  (Y)  to node[below] {$s'$} (W);
    \draw[cof] (PB) to node[right] {$t'$} (W);

    \draw[->]  (A) to node[above] {$f$} (B);
    \draw[cof] (A) to node[left]  {$i$} (C);

    \draw[->]  (C) to node[below] {$u$} (D);
    \draw[cof] (B) to node[right] {$v$} (D);
  \end{ctikzpicture}
  with both $s'$ and $t'$ weak equivalences.
  We can now apply the lifting property for pseudofactorizations to get
  the square on the left with $u$ and $v$ weak equivalences \st{} $Pu = s'$,
  $Pv = t'$ and (most importantly) $Pi = j$.
\end{proof}

Next, we proceed to the construction of factorizations.
This is the first of many situations where we need a way
of keeping track of certain homotopical properties of diagrams
in cofibration categories.
\emph{\HoRe{} categories} are very convenient for this purpose.

\begin{definition}
  A \emph{\hore{} category} is a category equipped with a subcategory
  whose morphisms are called \emph{weak equivalences}
  \st{} every identity morphism is a weak equivalence
  and the ``2 out of 6'' property holds.
\end{definition}

As discussed in the introduction, \hore{} categories are models
of homotopy theories in their own right, but we will use them merely as
a bookkeeping tool.
A functor $I \to J$ between \hore{} categories is \emph{\hore{}}
if it preserves weak equivalences.
In particular, for any cofibration category $\cat{C}$ and a \hore{} category $J$
the \hore{} functors $J \to \cat{C}$ will be called \emph{\hore{} diagrams}.
The notation $\cat{C}^J$ will always refer to the category
of all \hore{} diagrams $J \to \cat{C}$, it is itself a \hore{} category
with levelwise weak equivalences.
If $J$ is a plain category, then it will be considered as a \hore{} category
with the trivial \hore{} structure, i.e.\ with only isomorphisms
as weak equivalences.
On the other hand, $\hat{J}$ will denote $J$ equipped with
the largest \hore{} structure, i.e.\ the one where all morphisms
are weak equivalences.

Let $\cat{C}$ be a cofibration category and let $\Sd \hat{[1]}$ denote the poset
of non-empty subsets of $\{ 0, 1 \}$.
Make it into a \hore{} poset by declaring all morphisms
to be weak equivalences.
Call a diagram $X \from \Sd \hat{[1]} \to \cat{C}$ cofibrant
if both $X_0 \to X_{01}$ and $X_1 \to X_{01}$ are cofibrations in $\cat{C}$.
Let $P \cat{C}$ denote the category of all \hore{} cofibrant diagrams
$\Sd \hat{[1]} \to \cat{C}$ (i.e.\ $X$ \st{} both $X_0 \to X_{01}$
and $X_1 \to X_{01}$ are acyclic cofibrations).
Define weak equivalences in $P \cat{C}$ as levelwise weak equivalences
and define a morphism $A \to X$ to be a cofibration if all
\begin{align*}
  A_0 & \to X_0, \\
  A_1 & \to X_1, \\
  A_{01} \push_{A_0} X_0 & \to X_{01} \text{ and} \\
  A_{01} \push_{A_1} X_1 & \to X_{01}
\end{align*}
are cofibrations in $\cat{C}$.
(Note that this implies that $A_{01} \to X_{01}$ is a cofibration too.)

The notation $\Sd \hat{[1]}$ is a special case of the notation
that will be introduced later in \Cref{ch:qcats-of-frames},
but then we will always consider Reedy cofibrant diagrams
and not every cofibrant object in the sense above is Reedy cofibrant.
For a Reedy cofibrant object we would require $X_0 \coprod X_1 \to X_{01}$
to be a cofibration.
Similarly, cofibrations above are more general than Reedy cofibrations.
(See \Cref{direct-cat} for the definition.)
However, this notion reduces easily to the classical one,
i.e.\ a morphism $A \to X$ is a cofibration in $P \cat{C}$ \iff{}
its restrictions along the two non-trivial inclusions $[1] \ito \Sd \hat{[1]}$
are Reedy cofibrations.
The category $P \cat{C}$ will serve as a \emph{path object}
(i.e.\ a dual cylinder) in $\CofCat$.
The proof of the next proposition is merely an observation
that classical arguments about Reedy cofibrations are still valid
with this slightly more general definition.
Nonetheless, this modification is important since otherwise the diagonal functor
in the proof of \Cref{fibcat-of-cofcats} below would not be exact.

\begin{proposition}\label{path-object-cofcat}
  If $\cat{C}$ is a cofibration category, then so is $P \cat{C}$
  with the above weak equivalences and cofibrations.
\end{proposition}

\begin{proof}
  \enumhack
  \begin{itemize}
  \item[(C0)] Weak equivalences satisfy ``2 out of 6''
    since this holds in $\cat{C}$.
  \item[(C1)] A morphism $A \to X$ is an acyclic cofibration \iff{} all
    \begin{align*}
      A_0 & \to X_0, \\
      A_1 & \to X_1, \\
      A_{01} \push_{A_0} X_0 & \to X_{01} \text{ and} \\
      A_{01} \push_{A_1} X_1 & \to X_{01}
    \end{align*}
    are acyclic cofibrations in $\cat{C}$.
    Hence every isomorphism is an acyclic cofibration.
  \item[(C2-3)] The constant diagram of initial objects is cofibrant
    and initial in $P \cat{C}$.
    Moreover, the definition of a cofibrant object $X$ is equivalent to
    $\init \to X$ being a cofibration,
    thus all objects of $P \cat{C}$ are cofibrant.
  \item[(C4)] A cofibration in $P \cat{C}$ is in particular
    a levelwise cofibration and so pushouts along cofibrations in $P \cat{C}$
    exist and are constructed levelwise.
    Given a pushout square,
    \begin{ctikzpicture}
      \matrix[diagram]
      {
        |(A)| A & |(B)| B \\
        |(X)| X & |(Y)| Y \\
      };

      \draw[->]  (A) to (B);
      \draw[cof] (A) to (X);

      \draw[->] (B) to (Y);
      \draw[->] (X) to (Y);
    \end{ctikzpicture}
    in $P \cat{C}$ we observe that $B_0 \to Y_0$ and $B_1 \to Y_1$ are pushouts
    of $A_0 \to X_0$ and $A_1 \to X_1$ so they are cofibrations.
    The Pushout Lemma says that
    \begin{align*}
      B_{01} \push_{B_0} Y_0 \to Y_{01}
      \text{ and } B_{01} \push_{B_1} Y_1 \to Y_{01}
    \end{align*}
    are pushouts of
    \begin{align*}
      A_{01} \push_{A_0} X_0 \to X_{01}
      \text{ and } A_{01} \push_{A_1} X_1 \to X_{01}
    \end{align*}
    so they are cofibrations too.
    Consequently, $B \to Y$ is a cofibration in $P \cat{C}$.
    Stability of acyclic cofibrations under pushouts is obtained
    by combining this argument with the characterization
    of acyclic cofibrations given in (C1) above.
  \item[(C5)] Let $X \to Y$ be a morphism of $P \cat{C}$.
    For $i \in \{0, 1\}$ factor $X_i \to Y_i$ as $X_i \cto Z_i \weto Y_i$
    in $\cat{C}$ and form pushouts
    \begin{ctikzpicture}
      \matrix[diagram]
      {
        |(00)| X_i    & |(10)| Z_i \\
        |(01)| X_{01} & |(11)| W_i \text{.} \\
      };

      \draw[cof] (00) to (10);
      \draw[cof] (00) to (01);
      \draw[cof] (10) to (11);
      \draw[cof] (01) to (11);
    \end{ctikzpicture}
    Then we have the induced morphisms $W_i \to Y_{01}$ which make the square
    \begin{ctikzpicture}
      \matrix[diagram]
      {
        |(X)|  X_{01} & |(W0)| W_0 \\
        |(W1)| W_1    & |(Y)|  Y_{01} \text{.} \\
      };

      \draw[cof] (X) to (W0);
      \draw[cof] (X) to (W1);

      \draw[->] (W0) to (Y);
      \draw[->] (W1) to (Y);
    \end{ctikzpicture}
    commute and thus yield a morphism $W_0 \push_{X_{01}} W_1 \to Y_{01}$.
    We factor it in $\cat{C}$ as
    \begin{align*}
      W_0 \push_{X_{01}} W_1 \cto Z_{01} \weto Y_{01} \text{.}
    \end{align*}
    Then $Z$ becomes an object of $P \cat{C}$ and $X \cto Z \weto Y$
    is a factorization of the original morphism.
    \qedhere
  \end{itemize}
\end{proof}

We are ready to prove the main result of this section.

\begin{theorem}\label{fibcat-of-cofcats}
  The category $\CofCat$ with weak equivalences and fibrations as above
  is a fibration category.
\end{theorem}

In fact, $\CofCat$ is a homotopy complete category,
i.e.\ it has all small homotopy limits.
This will be explained in \Cref{sec:infinite-hocolims}.

\begin{proof}
  \enumhack
  \begin{itemize}
  \item[(C0)\textop] Weak equivalences satisfy ``2 out of 6''
    since they are created from equivalences of categories
    by $\Ho \from \CofCat \to \Cat$.
  \item[(C1)\textop] Isomorphisms are acyclic fibrations
    by \Cref{acyclic-fibrations}.
  \item[(C2-3)\textop] The category $[0]$ has a unique structure
    of a cofibration category and it is a terminal cofibration category.
    Moreover, every cofibration category is fibrant since every category
    is isofibrant while the lifting properties for factorizations
    and pseudofactorizations follow from the factorization axiom.
  \item[(C4)\textop] \Cref{pullback-of-cofcats} says that pullbacks
    along fibrations exist and by the construction they are also pullbacks
    in $\eCofCat$.
    Since fibrations are defined by the \rlp{} in this category they are stable
    under pullbacks.
    This argument also applies to acyclic fibrations
    by \Cref{acyclic-fibrations} since (App1) is equivalent to the \rlp{} \wrt{}
    the inclusion $[1] \ito \hat{[1]}$.
  \item[(C5)\textop] To verify the factorization axiom it suffices to construct
    a path object for every cofibration category $\cat{C}$
    by \cite{br}*{Factorization lemma, p. 421}.
    Let $\diag \from \cat{C} \to P \cat{C}$ be the diagonal functor.
    It preserves (acyclic) cofibrations since if $X \cto Y$
    is an (acyclic) cofibration in $\cat{C}$,
    then both $(\diag X)_0 \to (\diag Y)_0$ and $(\diag X)_1 \to (\diag Y)_1$
    coincide with $X \cto Y$ while
    \begin{align*}
      & (\diag X)_{01} \push_{(\diag X)_0} (\diag Y)_0 \to (\diag Y)_{01} \\
      \text{ and }
      & (\diag X)_{01} \push_{(\diag X)_1} (\diag Y)_1 \to (\diag Y)_{01}
    \end{align*}
    are isomorphisms.
    It also preserves the pushouts, sequential colimits and coproducts
    and hence is exact.
    The evaluation functor
    \begin{align*}
      \ev_{0,1} = (\ev_0, \ev_1) \from P \cat{C} \to \cat{C} \times \cat{C}
    \end{align*}
    is also exact.
    Together they form a factorization of the diagonal functor
    $\cat{C} \to \cat{C} \times \cat{C}$.
    We need to show that $\diag$ is a weak equivalence
    and that $\ev_{0,1}$ is a fibration.

    Consider the evaluation functor $\ev_{01} \from P \cat{C} \to \cat{C}$.
    It is a \hore{} functor \st{} $\ev_{01} \diag = \id_{\cat{C}}$ and there is
    a natural weak equivalence $\id_{P \cat{C}} \to \diag \ev_{01}$
    since all morphisms of $\Sd \hat{[1]}$ are weak equivalences.
    It follows that $\Ho \diag$ is an equivalence.

    It is easy to see that $\ev_{0,1}$ is an isofibration.
    The lifting property for factorizations is verified just like
    the factorization axiom in $P \cat{C}$ except that now the factorizations
    $X_i \cto Z_i \weto Y_i$ are given in advance.
    The lifting property for pseudofactorizations is handled similarly:
    let $X \to Y$ be a morphism in $P \cat{C}$ and let
    \begin{ctikzpicture}
      \matrix[diagram] {
        |(X)| X_i & |(Y)| Y_i \\
        |(W)| W_i & |(Z)| Z_i \\
      };

      \draw[->]  (X) to (Y);
      \draw[cof] (X) to (W);

      \draw[->]  (W) to node[below] {$\we$} (Z);
      \draw[cof] (Y) to node[right] {$\we$} (Z);
    \end{ctikzpicture}
    be pseudofactorizations of $X_i \to Y_i$ for $i \in \{ 0, 1 \}$.
    Form pushouts
    \begin{ctikzpicture}
      \matrix[diagram] {
        |(X)|  X_i    & |(W)| W_i & |(Y)| Y_i     & |(Z)| Z_i \\
        |(Xt)| X_{01} & |(U)| U_i & |(Yt)| Y_{01} & |(V)| V_i \text{.} \\
      };

      \draw[cof] (X) to (W);
      \draw[cof] (X) to (Xt);

      \draw[cof] (W)  to (U);
      \draw[cof] (Xt) to (U);

      \draw[cof] (Y) to node[above] {$\we$} (Z);
      \draw[cof] (Y) to (Yt);

      \draw[cof] (Z)  to (V);
      \draw[cof] (Yt) to node[below] {$\we$} (V);
    \end{ctikzpicture}
    There are induced morphisms $U_i \to V_i$ which fit into
    a commutative diagram
    \begin{ctikzpicture}
      \matrix[diagram] {
        |(U0)| U_0 & |(X)| X_{01} & |(U1)| U_1 \\
        |(V0)| V_0 & |(Y)| Y_{01} & |(V1)| V_1 \\
      };

      \draw[cof] (X) to (U0);
      \draw[cof] (X) to (U1);

      \draw[cof] (Y) to (V0);
      \draw[cof] (Y) to (V1);

      \draw[->] (U0) to (V0);
      \draw[->] (X)  to (Y);
      \draw[->] (U1) to (V1);
    \end{ctikzpicture}
    and thus induce a morphism
    $U_0 \push_{X_{01}} U_1 \to V_0 \push_{Y_{01}} V_1$
    which we pseudofactorize into
    \begin{ctikzpicture}
      \matrix [diagram] {
        |(X)| U_0 \push_{X_{01}} U_1 & |(Y)| V_0 \push_{Y_{01}} V_1 \\
        |(W)| W_{01}                 & |(Z)| Z_{01} \text{.} \\
      };

      \draw[->]  (X) to (Y);
      \draw[cof] (X) to (W);

      \draw[->]  (W) to node[below] {$\we$} (Z);
      \draw[cof] (Y) to node[right] {$\we$} (Z);
    \end{ctikzpicture}
    Then $W$ and $Z$ form objects of $P \cat{C}$ which fit into
    a pseudofactorization
    \begin{ctikzpicture}
      \matrix [diagram] {
        |(X)| X & |(Y)| Y \\
        |(W)| W & |(Z)| Z \text{.} \\
      };

      \draw[->]  (X) to (Y);
      \draw[cof] (X) to (W);

      \draw[->]  (W) to node[below] {$\we$} (Z);
      \draw[cof] (Y) to node[right] {$\we$} (Z);
    \end{ctikzpicture}
    as required. \qedhere
  \end{itemize}
\end{proof}

  \subsection{Cofibration categories of diagrams and homotopy colimits}
\label{sec:hocolims-and-diagrams}

As already suggested by the two proofs above, Reedy cofibrations play
an important role in the theory of cofibration categories.
The notion of a Reedy cofibrant diagram (but not really that
of a Reedy cofibration) will be essential in the proof of our main theorem.
We will not discuss the basic theory of Reedy cofibrations
since it is already well covered in the literature.
A good general reference is \cite{rv} which is written from the perspective
of Reedy categories and model categories.
The theory of diagrams over general Reedy categories
requires using both colimits and limits.
Thus in the case of cofibration categories we have to restrict attention
to a special class of Reedy categories called direct categories
where colimits suffice.
Specific results concerning Reedy cofibrations in cofibration categories
are explained in \cite{rb} from where we will cite a few most relevant
to the purpose of this paper.

\begin{definition}\label{direct-cat}
  \enumhack
  \begin{enumerate}[label=\dfn]
  \item A category $I$ is \emph{direct} if
    it admits a functor $\deg \from I \to \nat$ that reflects identities
    (here, we consider $\nat$ as a poset with its standard order).
  \item
    For a direct category $I$ and $i \in I$, the \emph{latching category} at $i$
    is the full subcategory of the slice $I \slice i$
    on all objects except for $\id_i$.
    It is denoted by $\bd (I \slice i)$.
  \item
    Let $X \from I \to \cat{C}$ be a diagram in some category and $i \in I$.
    The \emph{latching object} of $X$ at $i$ is the colimit
    of the composite diagram
    \begin{align*}
      \bd (I \slice i) \to I \to \cat{C}
    \end{align*}
    where $\bd (I \slice i) \to I$ is the forgetful functor sending a morphism
    of $I$ (i.e. an object of $\bd (I \slice i)$) to its source.
    The latching object (if it exists) is denoted by $L_i X$
    and comes with a canonical \emph{latching morphism} $L_i X \to X_i$
    induced by the inclusion $\bd (I \slice i) \to I \slice i$.
  \item
    Let $\cat{C}$ be a cofibration category.
    A diagram $X \from I \to \cat{C}$ is \emph{Reedy cofibrant}
    if for all $i \in I$ the latching object of $X$ at $i$ exists
    and the latching morphism $L_i X \to X_i$ is a cofibration.
  \item Let $f \from X \to Y$ be a morphism
    of Reedy cofibrant diagrams $I \to \cat{C}$.
    It is called a \emph{Reedy cofibration} if for all $i \in I$
    the induced morphism
    \begin{align*}
      X_i \push_{L_i X} L_i Y \to Y_i
    \end{align*}
    is a cofibration
    (observe that this pushout exists since $X$ is Reedy cofibrant).
  \end{enumerate}
\end{definition}

The main purpose of this subsection is to construct certain cofibration categories
of diagrams and establish some practical criteria for verifying
that particular functors between them are weak equivalences or fibrations.

\begin{proposition}\label{Reedy-level}
  Let $\cat{C}$ be a cofibration category and $J$ a \hore{} direct category
  with finite latching categories.
  \begin{enumerate}[label=\thm]
  \item
    The category $\cat{C}^J_\Reedy$ of \hore{} Reedy cofibrant diagrams
    with levelwise weak equivalences and Reedy cofibrations
    is a cofibration category.
  \item
    The category $\cat{C}^J$ of all \hore{} diagrams
    with levelwise weak equivalences and levelwise cofibrations
    is a cofibration category.
  \item The inclusion functor $\cat{C}^J_\Reedy \ito \cat{C}^J$
    is a weak equivalence.
  \end{enumerate}
\end{proposition}

\begin{proof}
  \enumhack
  \begin{enumerate}[label=\thm]
  \item \cite{rb}*{Theorem 9.3.8(1a)}
  \item \cite{rb}*{Theorem 9.3.8(1b)}
  \item The inclusion functor satisfies the approximation properties
    of \Cref{App} as follows from \Cref{relative-factorization}(1)
    (in fact, from its standard special case of $\cat{D} = [0]$
    and $I = \emptyset$). \qedhere
  \end{enumerate}
\end{proof}

The crucial step in the proof of the above proposition is the construction
of factorizations.
In \Cref{relative-factorization} we revisit that construction in order
to prove a more general version which will be a key technical tool
in many arguments of this paper.

A \hore{} functor $f \from I \to J$ is a \emph{homotopy equivalence}
if there is a \hore{} functor $g \from J \to I$
\st{} $g f$ is weakly equivalent to $\id_I$
and $f g$ is weakly equivalent to $\id_J$ (where ``weakly equivalent''
means ``connected by a zig-zag of natural weak equivalences'').

\begin{lemma}\label{level-htpy-eq}
  Let $\cat{C}$ be a cofibration category and $f \from I \to J$
  a \hore{} functor where $I$ and $J$ are \hore{} direct categories
  with finite latching categories.
  Then the induced functor $f^* \from \cat{C}^J \to \cat{C}^I$ is exact.
  Moreover, if $f$ is a homotopy equivalence,
  then $f^*$ is a weak equivalence of cofibration categories.
  Furthermore, if $f$ induces an exact functor
  $f^* \from \cat{C}^J_\Reedy \to \cat{C}^I_\Reedy$,
  then it is also a weak equivalence.
\end{lemma}

\begin{proof}
  The functor $f^*$ is clearly exact \wrt{} the levelwise structures
  and it is a homotopy equivalence when $f$ is.

  For the last statement, consider the commutative square of exact functors
  \begin{ctikzpicture}
    \matrix[diagram]
    {
      |(CJR)| \cat{C}^J_\Reedy & |(CIR)| \cat{C}^I_\Reedy \\
      |(CJ)|  \cat{C}^J        & |(CI)|  \cat{C}^I \\
    };

    \draw[->] (CJR) to node[above] {$f^*$} (CIR);
    \draw[->] (CJ)  to node[below] {$f^*$} (CI);

    \draw[inj] (CJR) to (CJ);
    \draw[inj] (CIR) to (CI);
  \end{ctikzpicture}
  the vertical maps are weak equivalences by \Cref{Reedy-level}
  so the conclusion follows by ``2 out of 3''.
\end{proof}

The utility of direct categories comes from the fact that it is easy
to construct diagrams and morphisms of diagrams inductively.
For our purposes it will be most convenient to state this in terms of sieves.
A functor $I \to J$ is called a \emph{sieve} if it is an inclusion of
a full downwards closed subcategory, i.e.\ if it is injective on objects,
fully faithful and if $i \to j$ is a morphism of $J$ \st{} $j \in I$,
then $i \in I$.

\begin{lemma}\label{latching-extension}
  Let $I \ito J$ be a sieve between direct categories
  and $j \in J \setminus I$ an object of a minimal degree.
  Let $X \from I \to \cat{C}$ be a Reedy cofibrant diagram.
  Then prolongations of  $X$ to a Reedy cofibrant diagram
  $I \union \{ j \} \to \cat{C}$ are naturally bijective with cofibrations
  $L_j X \cto X_j$ for varying $X_j \in \cat{C}$.
  ($L_j X$ exists by the minimality of $j$.)

  Similarly, if $X$ is a Reedy cofibrant diagram over $I \union \{ j \}$
  and $f \from X|I \to Y$ is a Reedy cofibration, then prolongations of $f$
  (and $Y$) to a Reedy cofibration over $I \union \{ j \}$
  correspond bijectively to cofibrations $L_j Y \push_{L_j X} X_j \cto Y_j$.
\end{lemma}

\begin{proof}
  The only (non-identity) morphisms of $I \union \{ j \}$ missing from $I$
  are those going from objects of degree less than $\deg j$ to $j$
  and they are encoded by the latching morphism.
  Similarly, if $f \from X \to Y$ is a morphism (cofibration) of diagrams
  over $I$ and $X$ is already defined over $j$,
  then extensions of $f$ over $j$ correspond to squares
  \begin{ctikzpicture}
    \matrix[diagram]
    {
      |(LX)| L_j X & |(X)| X_j \\
      |(LY)| L_j Y & |(Y)| Y_j \\
    };

    \draw[->] (LX) to (X);
    \draw[->] (LY) to (Y);

    \draw[->] (LX) to (LY);
    \draw[->] (X)  to (Y);
  \end{ctikzpicture}
  which in turn correspond to morphisms $L_j Y \push_{L_j X} X_j \to Y_j$
  and such an extension is a Reedy cofibration precisely when this morphism
  is a cofibration.
\end{proof}

The first part of the next lemma generalizes the standard construction
of factorizations into Reedy cofibrations followed by weak equivalences.
It says that given a morphism of diagrams $J \to \cat{C}$
and compatible factorizations of its restriction along a sieve $I \ito J$
and its image under a fibration $P \from \cat{C} \to \cat{D}$,
there is a factorization of the original morphism compatible with both of them.
The other two parts say the same for lifts for pseudofactorizations
and for cofibrations (when $P$ is an acyclic fibration
as in \Cref{acyclic-fibrations}).
\begin{anfxnote}{exposition, typesetting}
  the statement might be difficult to read, how to improve it?
\end{anfxnote}
\begin{lemma}\label{relative-factorization}
  Let $P \from \cat{C} \fto \cat{D}$ be a fibration between
  cofibration categories.
  Let $J$ be a \hore{} direct category with finite latching categories
  and $I \ito J$ a sieve.
  \begin{enumerate}[label=\thm]
  \item Let $f \from X \to Y$ be a morphism in $\cat{C}^J$.
    If $X$ is Reedy cofibrant,
    \begin{ctikzpicture}
      \matrix[diagram]
      {
        |(PX)| PX  & |(tYP)| \tilde{Y}_P & |(PY)| PY & \text{and} &
        |(XI)| X|I & |(tYI)| \tilde{Y}_I & |(YI)| Y|I \\
      };

      \draw[cof] (PX) to node[above] {$k_P$} (tYP);
      \draw[cof] (XI) to node[above] {$k_I$} (tYI);

      \draw[->] (tYP) to node[above] {$s_P$} node[below] {$\we$} (PY);
      \draw[->] (tYI) to node[above] {$s_I$} node[below] {$\we$} (YI);
    \end{ctikzpicture}
    are factorizations of $Pf$ and $f|I$ into Reedy cofibrations
    followed by weak equivalences \st{} $Pk_I = k_P|I$ and $Ps_I = s_P|I$
    (in particular, $P\tilde{Y}_I = \tilde{Y}_P|I$),
    then there is a factorization
    \begin{ctikzpicture}
      \matrix[diagram]
      {
        |(X)| X & |(tY)| \tilde{Y} & |(Y)| Y \\
      };

      \draw[cof] (X)  to node[above] {$k$} (tY);

      \draw[->] (tY) to node[above] {$s$} node[below] {$\we$} (Y);
    \end{ctikzpicture}
    of $f$ into a Reedy cofibration followed by a weak equivalence
    \st{} $Pk = k_P$, $k|I = k_I$, $Ps = s_P$ and $s|I = s_I$
    (in particular, $P\tilde{Y} = \tilde{Y}_P$ and $\tilde{Y}|I = \tilde{Y}_I$).
  \item Let $f \from X \to Y$ be a morphism in $\cat{C}^J$.
    If both $X$ and $Y$ are Reedy cofibrant,
    \begin{ctikzpicture}
      \matrix[diagram]
      {
        |(PX)| PX  & |(PY)| PY & \text{and} & |(XI)| X|I & |(YI)| Y|I \\
          |(tYP)| \tilde{Y}_P & |(hYP)| \hat{Y}_P &
        & |(tYI)| \tilde{Y}_I & |(hYI)| \hat{Y}_I \\
      };

      \draw[cof] (PX) to node[left] {$k_P$} (tYP);
      \draw[cof] (XI) to node[left] {$k_I$} (tYI);

      \draw[->] (tYP) to node[below] {$s_P$} node[above] {$\we$} (hYP);
      \draw[->] (tYI) to node[below] {$s_I$} node[above] {$\we$} (hYI);

      \draw[->] (PX) to node[above] {$Pf$}  (PY);
      \draw[->] (XI) to node[above] {$f|I$} (YI);

      \draw[cof] (PY) to node[right] {$l_P$} node[left] {$\we$} (hYP);
      \draw[cof] (YI) to node[right] {$l_I$} node[left] {$\we$} (hYI);
    \end{ctikzpicture}
    are pseudofactorizations of $Pf$ and $f|I$ \st{} $Pk = k_p$, $k|I = k_I$,
    $Pl = l_P$, $l|I = l_I$, $Ps = s_P$ and $s|I = s_I$
    (in particular, $P\tilde{Y}_I = \tilde{Y}_P|I$
    and $P\hat{Y}_I = \hat{Y}_P|I$), then there is a pseudofactorization
    \begin{ctikzpicture}
      \matrix[diagram]
      {
        |(X)|  X         & |(Y)|  Y  \\
        |(tY)| \tilde{Y} & |(hY)| \hat{Y} \\
      };

      \draw[cof] (X) to node[left] {$k$} (tY);

      \draw[->] (tY) to node[below] {$s$} node[above] {$\we$} (hY);

      \draw[->] (X) to node[above] {$f$} (Y);

      \draw[cof] (Y) to node[right] {$l$} node[left] {$\we$} (hY);
    \end{ctikzpicture}
    \st{} $Pk = k_P$, $k|I = k_I$, $Pl = l_P$, $l|I = l_I$, $Ps = s_P$
    and $s|I = s_I$ (in particular, $P\tilde{Y} = \tilde{Y}_P$,
    $\tilde{Y}|I = \tilde{Y}_I$, $P\hat{Y} = \hat{Y}_P$
    and $\hat{Y}|I = \hat{Y}_I$).
  \item If $P$ is acyclic, $X \in \cat{C}^J_\Reedy$ and
    \begin{ctikzpicture}
      \matrix[diagram]
      {
        |(PX)| PX & |(ZP)| Z_P & |(XI)| X|I & |(ZI)| Z_I \\
      };

      \draw[cof] (PX) to node[above] {$k_P$} (ZP);
      \draw[cof] (XI) to node[above] {$k_I$} (ZI);
    \end{ctikzpicture}
    are Reedy cofibrations \st{} $P k_I = k_P|I$, then there exists
    a Reedy cofibration
    \begin{ctikzpicture}
      \matrix[diagram]
      {
        |(X)| X & |(Z)| Z \\
      };

      \draw[cof] (X) to node[above] {$k$} (Z);
    \end{ctikzpicture}
    \st{} $Pk = k_P$ and $k|I = k_I$
    (in particular, $PZ = Z_P$ and $Z|I = Z_I$).
  \end{enumerate}
\end{lemma}

\begin{proof}
  The proofs of three parts are similar to each other
  so we only provide the first one.

    It suffices to extend the factorization $f|I = s_I k_I$ over an object
    $j \in J \setminus I$ of a minimal degree.
    Then the statement will follow by an induction over the degree.

    By the minimality of the degree of $j$, Reedy cofibrancy of $X$
    and since $I \ito J$ is a sieve the latching objects $L_j X$
    and $L_j \tilde{Y}_I$ exist.
    Moreover, the induced functor of latching categories
    $\bd (I \slice j) \to \bd (J \slice j)$ is an isomorphism.
    Thus $P$ sends the morphism $X_j \push_{L_j X} L_j \tilde{Y}_I \to Y_j$
    to the analogous morphism $PX_j \push_{L_j PX} P\tilde{Y}_I \to P Y_j$.
    The latter factors as
    \begin{align*}
      PX_j \push_{L_j PX} P\tilde{Y}_I \cto (\tilde{Y}_P)_j \weto P Y_j
    \end{align*}
    and since $P$ is a fibration we can lift this to a factorization
    of the former as
    \begin{align*}
      X_j \push_{L_j X} L_j \tilde{Y}_I \cto \tilde{Y}_j \weto Y_j \text{.}
    \end{align*}
    This extends the factorization $f|I = s_I k_I$ over $j$
    by \Cref{latching-extension}.
    The resulting diagram $\tilde Y$ is \hore{} since it is weakly equivalent to
    \hore{} $Y$.
\end{proof}

The most typical examples of fibrations are restrictions along sieves.

\begin{lemma}\label{sieve-exact}
  Let $\cat{C}$ be a cofibration category.
  If $I$ and $J$ are \hore{} direct categories with finite latching categories
  and $f \from I \to J$ a \hore{} functor \st{} for every $i \in I$
  the induced functor of the latching categories
  $\bd(I \slice i) \to \bd(J \slice fi)$ is an isomorphism,
  then the induced functor $f^* \from \cat{C}^J_\Reedy \to \cat{C}^I_\Reedy$
  is exact.

  Moreover, if $f$ is a sieve, then $f^*$ is a fibration.
\end{lemma}

\begin{proof}
  If $f$ induces isomorphisms of the latching categories,
  then $f^*$ preserves Reedy cofibrations
  (and, in particular, Reedy cofibrant diagrams).
  It also preserves weak equivalences and colimits
  that exist in $\cat{C}^J_\Reedy$ so it is exact.

  If $f$ is a sieve, then it satisfies the exactness criterion above.
  Moreover, $f^*$ is a fibration by parts (1) and (2)
  of \Cref{relative-factorization}.
\end{proof}

The next few lemmas establish some connections between sieves and fibrations
which are reminiscent of classical homotopical algebra if we think of
sieves as ``cofibrations'' and sieves $I \ito J$ inducing weak equivalences
$\cat{C}^J_\Reedy \to \cat{C}^I_\Reedy$ as ``acyclic cofibrations''.
This does not quite fit into the classical picture since such ``cofibrations''
do not really belong to the same category as the fibrations.

\begin{lemma}\label{dual-pushout-prod}
  Let $f \from I \ito J$ be a sieve between \hore{} direct categories
  with finite latching categories
  and $P \from \cat{C} \to \cat{D}$ a fibration of cofibration categories.
  Then the induced exact functor
  $(f^*, P) \from \cat{C}^J_\Reedy
    \to \cat{C}^I_\Reedy \times_{\cat{D}^I_\Reedy} \cat{D}^J_\Reedy$
  \begin{enumerate}[label=\thm]
  \item is a fibration,
  \item is an acyclic fibration provided that $P$ is acyclic,
  \item is an acyclic fibration provided that
    both $f^* \from \cat{C}^J_\Reedy \to \cat{C}^I_\Reedy$
    and $f^* \from \cat{D}^J_\Reedy \to \cat{D}^I_\Reedy$ are weak equivalences.
  \end{enumerate}
\end{lemma}

\begin{proof}
  First observe that the pullback in question exists since $f^*$ is a fibration
  by \Cref{sieve-exact}.
  \begin{enumerate}[label=\thm]
  \item This follows by parts (1) and (2) of \Cref{relative-factorization}.
  \item This follows by (1) above and part (3) of \Cref{relative-factorization}.
  \item This follows by (1) above and a diagram chase. \qedhere
  \end{enumerate}
\end{proof}

\begin{lemma}\label{sieves-pullback}
  If $\cat{C}$ is a cofibration category,
  \begin{ctikzpicture}
    \matrix[diagram]
    {
      |(I)| I & |(J)| J \\
      |(K)| K & |(L)| L \\
    };

    \draw[inj] (I) to (J);
    \draw[inj] (I) to (K);
    \draw[inj] (J) to (L);
    \draw[inj] (K) to (L);
  \end{ctikzpicture}
  is a pushout square of \hore{} direct categories
  with finite latching categories and both $I \ito J$ and $I \ito K$ are sieves,
  then the resulting square
  \begin{ctikzpicture}
    \matrix[diagram]
    {
      |(CL)| \cat{C}^L_\Reedy & |(CK)| \cat{C}^K_\Reedy \\
      |(CJ)| \cat{C}^J_\Reedy & |(CI)| \cat{C}^I_\Reedy \\
    };

    \draw[->] (CL) to (CK);
    \draw[->] (CL) to (CJ);
    \draw[->] (CJ) to (CI);
    \draw[->] (CK) to (CI);
  \end{ctikzpicture}
  is a pullback of cofibration categories.
\end{lemma}

\begin{proof}
  By the construction of pullbacks of cofibration categories it will suffice to
  verify that a morphism of diagrams over $L$ is a Reedy cofibration \iff{}
  it is one when restricted to both $J$ and $K$.
  For this it will be enough to observe that both $J \ito L$ and $K \ito L$
  are sieves and hence for an object $l \in L$ we have either $l \in J$
  and then $\bd(J \slice l) \to \bd(L \slice l)$ is an isomorphism
  or $l \in K$ and then $\bd(J \slice l) \to \bd(L \slice l)$ is an isomorphism.
\end{proof}

Let $f \from I \to J$ be a \hore{} functor of \hore{} direct categories
and $F \from \cat{C} \to \cat{D}$ an exact functor of cofibration categories.
We say that $f$ has the \Rllp{} \wrt{} $F$ (or $F$ has the \Rrlp{} \wrt{} $f$)
if every lifting problem
\begin{ctikzpicture}
  \matrix [diagram]
  {
    |(I)| I & |(C)| \cat{C} \\
    |(J)| J & |(D)| \cat{D} \\
  };

  \draw[->] (I) to node[above] {$X$} (C);
  \draw[->] (J) to node[below] {$Y$} (D);

  \draw[->] (I) to node[left]  {$f$} (J);
  \draw[->] (C) to node[right] {$F$} (D);
\end{ctikzpicture}
where $X$ and $Y$ are \hore{} Reedy cofibrant diagrams has a solution
that is also a \hore{} Reedy cofibrant diagram.
Such lifting properties will be heavily used in the latter two sections.

\begin{lemma}\label{Rlps}
  Let $f \from I \ito J$ and $g \from K \to L$ be sieves
  between \hore{} direct categories with finite latching categories
  and $F \from \cat{C} \to \cat{D}$ an exact functor of cofibration categories.
  Then there is a natural bijection between Reedy lifting problems
  (and their solutions) of the forms
  \begin{ctikzpicture}
    \matrix[diagram]
    {
      |(I)|  I &[-.5em] |(CL)| \cat{C}^L_\Reedy &[-2.7em]
      |(fg)| (I \times L) \push_{I \times K} (J \times K) &[-.5em]
        |(C)| \cat{C} &[-1.5em]
      |(K)|  K &[-.5em] |(CJ)| \cat{C}^J_\Reedy \\
      |(J)|  J & |(Fg)| \cat{C}^K_\Reedy \times_{\cat{D}^K_\Reedy}
                       \cat{D}^L_\Reedy &
      |(JL)| J \times L & |(D)| \cat{D} &
      |(L)|  L & |(Ff)| \cat{C}^I_\Reedy \times_{\cat{D}^I_\Reedy}
                       \cat{D}^J_\Reedy \text{.} \\
    };

    \draw[inj] (I)  to (J);
    \draw[inj] (K)  to (L);
    \draw[inj] (fg) to (JL);

    \draw[->] (CL) to (Fg);
    \draw[->] (C)  to (D);
    \draw[->] (CJ) to (Ff);

    \draw[->] (I)  to (CL);
    \draw[->] (fg) to (C);
    \draw[->] (K)  to (CJ);

    \draw[->] (J)  to (Fg);
    \draw[->] (JL) to (D);
    \draw[->] (L)  to (Ff);
  \end{ctikzpicture}
\end{lemma}

\begin{proof}
  This is proven with standard adjointness arguments,
  e.g. as in \cite{jo}*{Proposition D.1.18}, using the fact that
  a diagram $J \to \cat{C}^L_\Reedy$ is Reedy cofibrant \iff{}
  the corresponding diagram $J \times L \to \cat{C}$ is as follows from
  \cite{rv}*{Example 4.6}.
\end{proof}

\begin{lemma}\label{acyclic-fibration-Reedy}
  Let $P \from \cat{C} \to \cat{D}$ be a fibration of cofibration categories.
  \Tfae:
  \begin{enumerate}[label=\thm]
  \item $P$ is acyclic,
  \item $P$ has the \Rrlp{} \wrt{} all sieves between direct \hore{} categories
    with finite latching categories,
  \item $P$ has the \Rrlp{} \wrt{} $[0] \ito [1]$ and $[1] \ito \hat{[1]}$.
  \end{enumerate}
\end{lemma}

\begin{proof}
  If $P$ is acyclic, then it has the \Rrlp{} \wrt{} all sieves
  between \hore{} direct categories with finite latching categories
  by \Cref{relative-factorization}(3),
  in particular, \wrt{} $[0] \ito [1]$ and $[1] \ito \hat{[1]}$.

  Conversely, by \Cref{acyclic-fibrations} it suffices to see that
  if $P$ has the \Rrlp{} \wrt{} $[0] \ito [1]$ and $[1] \ito \hat{[1]}$,
  then it satisfies (App1) and has the \rlp{} in $\eCofCat$ \wrt{}
  the inclusion of $[0]$ into
  \begin{ctikzpicture}
    \matrix[diagram]
    {
      |(0)| 0 & |(1)| 1 \text{.} \\
    };

    \draw[cof] (0) to (1);
  \end{ctikzpicture}
  The latter is equivalent to the \Rrlp{} \wrt{} $[0] \ito [1]$.
  To see that the \Rrlp{} \wrt{} $[1] \ito \hat{[1]}$ implies (App1)
  take a morphism $f \from X \to Y$ in $\cat{C}$
  \st{} $Pf$ is a weak equivalence.
  Factor $f$ as
  \begin{ctikzpicture}
    \matrix[diagram]
    {
      |(X)| X & |(tY)| \tilde Y & |(Y)| Y \text{.} \\
    };

    \draw[cof] (X)  to node[above] {$j$}   (tY);
    \draw[->]  (tY) to node[above] {$\we$} (Y);
  \end{ctikzpicture}
  Then $Pj$ is a weak equivalence by ``2 out of 3'' and hence so is $j$
  by the \Rrlp{} \wrt{} $[1] \ito \hat{[1]}$.
  Thus $f$ is a weak equivalence, too.
\end{proof}

\begin{lemma}\label{acyclic-sieve}
  If a sieve $f \from I \to J$ between \hore{} direct categories has the \Rllp{}
  \wrt{} all fibrations of cofibration categories,
  then for every cofibration category $\cat{C}$
  the induced functor $f^* \from \cat{C}^J_\Reedy \to \cat{C}^I_\Reedy$
  is an acyclic fibration.
\end{lemma}

\begin{proof}
  Since $f$ is a sieve it will suffice to check that $f^*$ has the \Rrlp{}
  \wrt{} $[0] \ito [1]$ and $[1] \ito \hat{[1]}$
  by \Cref{acyclic-fibration-Reedy}.
  These are equivalent to the \Rrlp{}
  of $\cat{C}^{[1]}_\Reedy \to \cat{C}^{[0]}_\Reedy$
  and $\cat{C}^{\hat{[1]}}_\Reedy \to \cat{C}^{[1]}_\Reedy$ \wrt{} $I \ito J$
  by \Cref{dual-pushout-prod}.
\end{proof}

The following proposition says that in cofibration categories colimits
of Reedy cofibrant diagrams (over finite direct categories) exist
and are homotopy invariant.
In effect, this yields finite direct homotopy colimits
in cofibration categories.

\begin{proposition}\label{colims-in-cofcats}
  If $I$ is a finite direct category, then the colimit functor
  $\cat{C}^I_\Reedy \to \cat{C}$ exists and is exact.
\end{proposition}

\begin{proof}
  \cite{rb}*{Theorem 9.3.5(1)}
\end{proof}

It is perhaps worth pointing out that this construction does not directly apply
to non-cofibrant diagrams, but all direct diagrams can be replaced
by Reedy cofibrant ones.
This is not directly captured by the definition of a cofibration category
as given in the beginning of this section since we insisted
that all objects are cofibrant.
Instead, we can think of the homotopy colimit functor as a zig-zag
of exact functors
\begin{ctikzpicture}
  \matrix[diagram,column sep=3em]
  {
    |(CI)| \cat{C}^I & |(CIR)| \cat{C}^I_\Reedy &
    |(C)|  \cat{C} \text{.} \\
  };

  \draw[inj] (CIR) to node[below] {$\we$}      (CI);
  \draw[->]  (CIR) to node[above] {$\colim_I$} (C);
\end{ctikzpicture}
Here, the functor on the left is the one discussed in \Cref{Reedy-level}.

Cofibration categories admit all finite homotopy colimits,
but finiteness has to be understood in a rather strong sense.
Namely, a finite homotopy colimit is a homotopy colimit of a diagram indexed
over a category $I$ whose nerve is a finite simplicial set.
Such categories coincide with finite direct categories
and hence \Cref{colims-in-cofcats} implies existence of finite homotopy colimits
in cofibration categories.

Notice that e.g.\ homotopy colimits of diagrams indexed
by non-trivial finite groups are \emph{not} finite homotopy colimits
since the nerves of such groups are infinite and hence homotopy colimits
over them involve infinite amount of coherence data.

\Cref{colims-in-cofcats} implies that a pushout of two cofibrations
in a cofibration category is a homotopy pushout. In fact, a more general
and extremely useful statement is true: a pushout of any morphism
along a cofibration is a homotopy pushout. This is known as the Gluing Lemma.

\begin{lemma}[Gluing Lemma]
  Given a commutative cube
  \begin{ctikzpicture}
    \matrix[diagram,row sep=2em,column sep=2em]
    {
      |(A0)| A_0 & & |(B0)| B_0 \\
      & |(A1)| A_1 & & |(B1)| B_1 \\
      |(X0)| X_0 & & |(Y0)| Y_0 \\
      & |(X1)| X_1 & & |(Y1)| Y_1 \\
    };

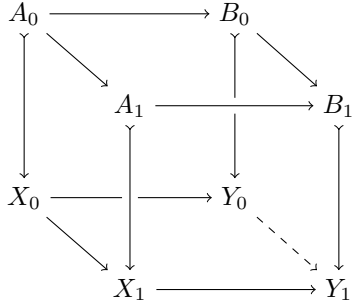
    \draw[->] (A0) to (B0);
    \draw[->] (X0) to (Y0);

    \draw[cof] (A0) to (X0);
    \draw[cof] (B0) to (Y0);

    \draw[->,cross line] (A1) to (B1);
    \draw[->] (X1) to (Y1);

    \draw[cof,cross line] (A1) to (X1);
    \draw[cof] (B1) to (Y1);

    \draw[->] (A0) to (A1);
    \draw[->] (B0) to (B1);
    \draw[->] (X0) to (X1);
    \draw[->,dashed] (Y0) to (Y1);
  \end{ctikzpicture}
  where the indicated morphisms are cofibrations and both front and back squares
  are pushouts, if the three solid arrows going from the back square
  to the front square are weak equivalences, then so is the dashed one.

  More generally, the conclusion holds provided that both front and back squares
  are homotopy pushouts, i.e.\ can be connected by zig-zags
  of natural weak equivalences to pushouts along cofibrations.
\end{lemma}

\begin{proof}
  \cite{rb}*{Lemma 1.4.1(1)}
\end{proof}

While the proof of the Gluing Lemma cited above does not state this explicitly,
the argument is basically an application of the K. Brown's Lemma.
Recall that $\spn$ is the poset of proper subsets of $\{ 0, 1 \}$.
It can be proven (similarly to \Cref{path-object-cofcat}) that
there is a cofibration category $\cat{C}^{\scalebox{.6}{\spn}}_{\mathrm{p}}$
of ``partially Reedy cofibrant diagrams'' $X \from \spn \to \cat{C}$,
i.e.\ \st{} $X_\emptyset \to X_0$ is a cofibration.
The weak equivalences are levelwise
and cofibrations are ``partial Reedy cofibrations'', i.e.\ levelwise cofibrations
that are Reedy cofibrations when restricted to $\emptyset \to 0$.
One way to motivate the pushout axiom (C4) is that this is what is required for
the pushout functor
$\colim_{\scalebox{.6}{\spn}} \from \cat{C}^{\scalebox{.6}{\spn}}_{\mathrm{p}}
  \to \cat{C}$ to be exact.
More precisely, stability of (acyclic) cofibrations under pushouts implies that
this functor preserves (acyclic) cofibrations.

We will often need to know that certain \hore{} functors
between \hore{} direct categories induce weak equivalences of homotopy colimits.
Such functors are called \emph{homotopy cofinal}.
For our purposes the following simple criterion is sufficient.

\begin{lemma}\label{equiv-cofinal}
  Let $f \from I \to J$ be a \hore{} functor
  between finite \hore{} direct categories and $\cat{C}$ a cofibration category.
  If $f$ induces a weak equivalence $\cat{C}^J_\Reedy \to \cat{C}^I_\Reedy$,
  then for every \hore{} Reedy cofibrant diagram $X \from J \to \cat{C}$
  the induced morphism $\colim_I f^* X \to \colim_J X$ is a weak equivalence.
\end{lemma}

\begin{proof}
  The left Kan extension functor
  $\Lan_f \from \cat{C}^I_\Reedy \to \cat{C}^J_\Reedy$ exists,
  is exact by \cite{rb}*{Theorem 9.4.3(1)} and is a left adjoint of $f^*$.
  Hence $\Lan_f$ is a weak equivalence since $f^*$ is.
  In particular, the counit $\Lan_f f^* X \to X$ is a weak equivalence
  and hence so is the resulting morphism $\colim_J \Lan_f f^* X \to \colim_J X$
  which coincides with the morphism $\colim_I f^* X \to \colim_J X$.
\end{proof}

  \subsection{Examples}
\label{sec:examples}

In order to better motivate cofibration and fibration categories we list a
number of interesting examples.
Neither of them is known to come from a model category
and some of them are not known to (and in some cases actually known not to)
be equivalent to model categories. 

\subsubsection*{$C^*$-algebras}

\begin{theorem}[\cite{sc}*{Section 1}]
  The category of $C^*$-algebras carries a structure of a pointed fibration
  category.
\end{theorem}

A streamlined proof of this theorem can be found in \cite{u}*{Theorem 2.19}
along with a few accompanying results in a similar spirit. Moreover, it is
proven \cite{u}*{Theorem A.1} that the \hore{} category of $C^*$-algebras
does not admit a model structure. This result (originally due to Andersen and
Grodal \cite{ag}*{Corollary 4.7}) can be phrased in an even stronger way: there
is no model category whose underlying fibration category is weakly equivalent to
the fibration category of the theorem above. This is because the loop functor
fails to have a left adjoint. It follows that not even a cofibration category
presenting the homotopy theory of $C^*$-algebras exists.

\subsection*{Proper homotopy theory}

\begin{theorem}[\cite{bq}*{Theorems 3.6 and 4.5}]
  The category of topological spaces with proper maps as morphisms carries a
  structure of a cofibration category.
\end{theorem}

The weak equivalences of this fibration category are
\emph{proper homotopy equivalences}.
A proper map $f \from X \to Y$ is a proper homotopy equivalence if it admits
a proper map $g \from Y \to X$ and homotopies $g f \htp \id_X$
and $f g \htp \id_Y$ through proper maps.
The cofibrations are \emph{proper (Hurewicz) cofibrations},
i.e. proper maps $A \to B$ with the proper homotopy extension property.
This means that we require that every proper homotopy defined on $A$ whose one
end extends over $B$ also extends over $B$ (to a proper homotopy).
This category does not carry a structure of a fibration category,
e.g.\ since it has no terminal object.

\subsection*{Homotopy type theory}

\begin{theorem}[\cite{akl}*{Theorem 2.2.5}]
  Every categorical model of homotopy type theory carries a canonical structure
  of a fibration category.
\end{theorem}

This category has certain distinguished class of maps
that are natural candidates for cofibrations (and would have to be cofibrations
if this fibration category was a part of a model category).
Unfortunately, it turns out that pushouts along these maps fail to exist
in general.

\subsection*{Topological spaces}

Most of the remaining examples discuss some well known \hore{} categories
which admit well know model structures, but in addition they also carry
less known structures of (co)fibration categories.
They typically have more (co)fibrations than the classical model structures
which means that they provide more point-set models of homotopy (co)limits.

We start with the category of topological spaces
which has two notable classes of weak equivalences: \emph{homotopy equivalences}
and \emph{weak homotopy equivalences}.
All of these examples seem to be folklore but we know almost no references.

A map of topological spaces $p \from X \to Y$ is a \emph{Dold fibration}
if it has the \emph{weak covering homotopy property},
i.e.\ for each square on the left
\begin{ctikzpicture}
  \matrix[diagram]
  {
    |(A)|  A          & |(X)| X & |(A1)| A & |(XI)| X^I \\
    |(AI)| A \times I & |(Y)| Y & |(B)|  B & |(X1)| X \\
  };

  \draw[inj] (A)  to node[left]  {$i_0$} (AI);
  \draw[->]  (X)  to node[right] {$p$}   (Y);
  \draw[->]  (A)  to node[above] {$u$}   (X);
  \draw[->]  (AI) to node[below] {$H$}   (Y);

  \draw[->] (XI) to node[right] {$p_0$} (X1);
  \draw[->] (A1) to node[left]  {$i$}   (B);
  \draw[->] (A1) to node[above] {$H$}   (XI);
  \draw[->] (B)  to node[below] {$v$}   (X1);
\end{ctikzpicture}
there exists a homotopy $G \from A \times I \to X$ \st{} $pG = H$
and $Gi_0$ is homotopic to $u$ fiberwise over $Y$.
Dually, a map $i \from A \to B$ is a \emph{Dold cofibration}
if for all squares on the right above
there exists a homotopy $G \from B \to X^I$ \st{} $Gi = H$
and $p_0 G$ is homotopic to $v$ relative to $A$.

\begin{theorem}\label{Dold-cof}
  \enumhack
  \begin{enumerate}[label=\thm]
  \item The category of topological spaces with homotopy equivalences
    and Dold fibrations is a fibration category.
  \item The category of topological spaces with homotopy equivalences
    and Dold cofibrations is a cofibration category.
  \end{enumerate}
\end{theorem}

Dold fibrations were introduced in \cite{d}
and both Dold fibrations and cofibrations are discussed in \cite{tdkp}.
There are more Dold (co)fibrations than classical Hurewicz (co)fibrations.

A \emph{Dold--Serre fibration} is a map satisfying
the weak covering homotopy property as above but only for $A = D^m$
for all $m \ge 0$.

\begin{theorem}
  The category of topological spaces with weak homotopy equivalences
  Dold--Serre fibrations is a fibration category.
\end{theorem}

Again, there are more Dold--Serre fibrations than classical Serre fibrations.

One could expect that there is a corresponding notion of
a ``Dold--Serre cofibration'', but this does not seem to be the case.
However, something even better is true.

\begin{theorem}
  The category of topological spaces with weak homotopy equivalences
  and Hurewicz cofibrations is a cofibration category.
\end{theorem}

At the first glance this may seem to come from a mixed model structure
in the sense of Cole \cite{co}, but it does not.
This is an attempt to mix in the ``wrong direction'' which succeeds
for delicate point-set reasons.
We know from \cite{rb}*{Lemma 1.4.3(1)} that is suffices to verify that
weak homotopy equivalences and Hurewicz cofibrations satisfy the Gluing Lemma
and this holds by \cite{bv}*{Appendix, Proposition 4.8(b)}.
In fact, by combining this observation with \Cref{Dold-cof}(2) one can show
that this is even true with Dold cofibrations
in the place of Hurewicz cofibrations.

\subsection*{Simplicial and categorical homotopy theory}

As we have already illustrated, one can often find classes of (co)fibrations
that are larger than ones coming from classical model structures.
In fact, it is not difficult to prove that if there is
at least one class of (co)fibrations compatible with a given \hore{} category,
then there is also the largest one.
One of the few examples where such class is well understood is the category
of simplicial sets.

A simplicial map $f$ is \emph{sharp} if every strict pullback along $f$
is a homotopy pullback.
With this definition it is routine to prove the following result.

\begin{theorem}
  The category of all simplicial sets with weak homotopy equivalences
  and sharp maps is a fibration category.
\end{theorem}

Sharp maps were introduced by Rezk \cite{r-sh}.
Clearly, a fibration in any fibration category of simplicial sets
(with weak homotopy equivalences) is sharp hence this is indeed
the largest class of fibrations.
Observe that in this fibration category every simplicial set is fibrant.

The next two examples exploit the notion of Dwyer maps to connect
category theory to homotopy theory.
A \emph{Dwyer map} is a functor $f$ of small categories that is a sieve
and factors as $f = gj$ where $g$ is a cosieve and $j$ admits
a deformation retraction.

While Thomason \cite{th} does not state this explicitly, a crucial step of
his construction of a model structure on small categories is contained
in the following theorem.

\begin{theorem}
  The category of small categories with weak homotopy equivalences
  (i.e. the ones created by the nerve functor from weak homotopy equivalences
  of simplicial sets) and Dwyer maps is a cofibration category.
\end{theorem}

Barwick and Kan \cites{bk1,bk2} in the construction of their model category
of relative categories (which was already discussed in the introduction)
used a similar approach.
They defined a suitable generalization of Dwyer maps and proved
(also implicitly) an analogous result.

\begin{theorem}
  The category of small relative categories with Dwyer--Kan equivalences
  and Dwyer maps is a cofibration category.
\end{theorem}

In both cases there are many more Dwyer maps than cofibrations
in their model categories.

  \section{Quasicategories}
  \label{ch:qcats}
  This section is devoted to a concise summary of the theory of quasicategories.
It is well covered in \cite{jo} and \cite{l} so we do not go into much detail.
Our main goal is to establish a fibration category
of finitely cocomplete quasicategories in \Cref{fibcat-of-fc-quasicats}.
We follow \cite{jo} to demonstrate that the fibration category
of all quasicategories can be obtained without constructing
the entire Joyal model structure (\Cref{fibcat-of-quasicats})
which makes the proof rather elementary.
(A more streamlined exposition of the same results can be found
in the appendices to \cite{ds}.)
Then we briefly introduce colimits in quasicategories
and state their basic properties used in the proof
of \Cref{fibcat-of-fc-quasicats} and later in \Cref{ch:qcats-of-frames}.

\subsection{Homotopy theory of quasicategories}

Recall that $\wiso$ is the groupoid freely generated
by an isomorphism $0 \to 1$.
Its nerve will be denoted by $\nwiso$.
Quasicategories are defined as certain special simplicial sets
and are to be thought of as models of $(\infty,1)$-categories
where vertices are objects, edges are morphisms
and higher simplices are higher morphisms (or higher homotopies).
Functors between quasicategories are just simplicial maps.
In particular, maps out of $\nwiso$ are equivalences in quasicategories
and $\nwiso$-homotopies are natural equivalences between functors.
The account of the homotopy theory of quasicategories below closely follows
the classical approach to simplicial homotopy theory
(see e.g. \cite{gj}*{Chapter I}) with Kan complexes replaced by quasicategories
and usual simplicial homotopies replaced by $\nwiso$-homotopies.

\begin{definition}
  \enumhack
  \begin{enumerate}[label=\dfn]
  \item
    Let $f, g \from K \to L$ be simplicial maps.
    An \emph{$\nwiso$-homotopy} from $f$ to $g$ is a simplicial map
    $K \times \nwiso \to L$ extending $[f, g] \from K \times \bdsimp{1} \to L$.
  \item Two simplicial maps $f, g \from K \to L$ are \emph{$\nwiso$-homotopic}
    if there exists a zig-zag of $\nwiso$-homotopies connecting $f$ to $g$.
    (It suffices to consider sequences instead of zig-zags since $\nwiso$
    has an automorphism that exchanges the vertices.)
  \item
    A simplicial map $f \from K \to L$ is
    an \emph{$\nwiso$-homotopy equivalence} if there is a simplicial map
    $g \from L \to K$ \st{} $f g$ is $\nwiso$-homotopic to $\id_L$
    and $g f$ is $\nwiso$-homotopic to $\id_K$.
  \end{enumerate}
\end{definition}

\begin{definition}
  \enumhack
  \begin{enumerate}[label=\dfn]
  \item
    A simplicial map is an \emph{inner fibration} if it has the \rlp{}
    \wrt{} the inner horn inclusions.
  \item
    A simplicial map is an \emph{inner isofibration}
    if it is an inner fibration and has the \rlp{}
    \wrt{} $\simp{0} \ito \nwiso$.
  \item
    A simplicial map is an \emph{acyclic Kan fibration} if it has the \rlp{}
    \wrt{} $\bdsimp{m} \ito \simp{m}$ for all $m$.
  \item
    A simplicial set $\qcat{C}$ is a \emph{quasicategory} if the unique map
    $\qcat{C} \to \simp{0}$ is an inner fibration.
  \end{enumerate}
\end{definition}

We will refer to $\nwiso$-equivalences between quasicategories as
\emph{categorical equivalences} and use them to introduce the homotopy theory
of quasicategories.
(It is also possible to extend this notion to maps of general simplicial sets,
but we have no need to do it.)
If $K$ is any simplicial set and $\qcat{C}$ is a quasicategory,
then the relation of ``being connected by a single $\nwiso$-homotopy''
is already an equivalence relation on the set
of simplicial maps $K \to \qcat{C}$ by \cite{ds}*{Proposition 2.3}.
This simplifies the definition of categorical equivalences since
it is always sufficient to consider one-step $\nwiso$-homotopies.

\begin{theorem}\label{fibcat-of-quasicats}
  The category of small quasicategories with simplicial maps as morphisms,
  categorical equivalences as weak equivalences
  and inner isofibrations as fibrations is a fibration category.
\end{theorem}

In fact, this fibration category is homotopy complete,
i.e.\ it admits all small homotopy limits as will be discussed
in \Cref{sec:infinite-hocolims}.

\begin{proof}
  Only two of the axioms require non-trivial proofs:
  stability of acyclic fibrations under pullbacks which follows from the fact that
  acyclic (inner iso-) fibrations coincide with acyclic Kan fibrations
  by \cite{jo}*{Theorem 5.15} and the factorization axiom which is verified
  in \cite{jo}*{Proposition 5.16}.
\end{proof}

This fibration category is a part of the Joyal model structure on simplicial sets
established in \cite{jo}*{Theorem 6.12}.
Indeed, the theorem above is an intermediate step in the construction
of this model category.

Quasicategories are models for homotopy theories and as such
they have homotopy categories.
Two morphisms $f, g \from x \to y$ of a quasicategory $\qcat{D}$
are \emph{homotopic} if there exists a simplex $H \from \simp{2} \to \qcat{D}$
\st{} $H \face_0 = y \dgn_0$, $H \face_1 = g$ and $H \face_2 = f$.
The \emph{homotopy category} of $\qcat{D}$ is the category $\Ho \qcat{D}$
with the same objects as $\qcat{D}$, homotopy classes of morphisms of $\qcat{D}$
as morphisms and the composition induced by filling horns.

If $f$ is a morphism of a quasicategory $\qcat{C}$, then we say that $f$
is an \emph{equivalence} if the simplicial map $f \from \simp{1} \to \qcat{C}$
extends to $\nwiso \to \qcat{C}$.
(By \cite{jo}*{Proposition 4.22} a morphism is an equivalence \iff{}
it becomes an isomorphism in the homotopy category.)
Two objects of $\qcat{C}$ are \emph{equivalent} if they are connected
by an equivalence.

We conclude this subsection by a technical lemma
saying that in quasicategories certain outer horns can be filled.
Let $\qcat{C}$ be a quasicategory.
A map $X \from \horn{m,i} \to \qcat{C}$ is called
a \emph{special outer horn in $\qcat{C}$} if $i = 0$ and $X | \Delta \{ 0, 1\}$
is an equivalence or $i = m$ and $X | \Delta \{ m - 1, m\}$ is an equivalence.

\begin{lemma}\label{special-horns}
  If $X \from \horn{m,i} \to \qcat{C}$ is a special outer horn
  and $p \from \qcat{C} \to \qcat{D}$ is an inner isofibration
  between quasicategories, then the diagram
  \begin{ctikzpicture}
    \matrix[diagram]
    {
      |(h)| \horn{m,i} & |(C)| \qcat{C} \\
      |(s)| \simp{m}   & |(D)| \qcat{D} \\
    };

    \draw[->] (h) to node[above] {$X$} (C);

    \draw[->]  (s) to (D);
    \draw[inj] (h) to (s);
    \draw[fib] (C) to (D);
  \end{ctikzpicture}
  admits a lift.
\end{lemma}

\begin{proof}
  \cite{jo}*{Theorem 4.13} or \cite{ds}*{Proposition B.11}
\end{proof}

  \subsection{Colimits}
\label{sec:colimits}

We proceed to the discussion of colimits in quasicategories.
Such colimits are homotopy invariant by design
and they serve as models for homotopy colimits.
However, in quasicategories there is no corresponding notion
of a ``strict'' colimit and thus it is customary
to refer to ``homotopy colimits'' in quasicategories simply as colimits.
The general theory of colimits is explored in depth in \cite{l}*{Chapter 4},
here we only discuss its most basic aspects.

The quasicategorical notion of colimit is defined using the join construction
for simplicial sets.
In order to define joins efficiently we briefly introduce
augmented simplicial sets.
The category $\Delta_\aug$ is defined as the category
of finite totally ordered sets of the form $[m]$ for $m \ge -1$
(where $[-1] = \emptyset$).
\label{augmented simplicial set}
The category of \emph{augmented simplicial sets} is the category of presheaves
on $\Delta_\aug$ and is denoted by $\asSet$.
The standard category $\Delta$ is a full subcategory of $\Delta_\aug$,
we denote the inclusion functor by $i \from \Delta \ito \Delta_\aug$.
Precomposition with $i$ is the forgetful functor $i^* \from \asSet \to \sSet$
and it has a right adjoint, the right Kan extension along $i$
denoted by $\Ran_i \from \sSet \to \asSet$.
Explicitly, $\Ran_i$ prolongs a simplicial set to an augmented simplicial set
by setting the value at $[-1]$ to a singleton.

The category $\Delta_\aug$ carries a (non-symmetric) strict monoidal structure
given by concatenation $[m], [n] \mapsto [m] \join [n] \iso [m + 1 + n]$
with $[-1]$ as the monoidal unit.
On morphisms it is also defined by concatenation:
$\phi \join \psi \from [k] \join [l] \to [m] \join [n]$ acts via $\phi$
on the first $k+1$ elements and via $\psi$ on the last $l+1$ ones.

\begin{proposition}
  \mbox{}
  \begin{enumerate}[label=\thm]
    \item The category of augmented simplicial sets carries
      a closed monoidal structure with the monoidal product,
      the \emph{join} $\join \from \asSet \times \asSet \to \asSet$
      uniquely characterized by its action on representables
      \begin{align*}
        \Delta_\aug[m], \Delta_\aug[n] \mapsto \Delta_\aug([m] \join
        [n]) \iso \Delta_\aug[m + 1 + n] \text{.}
      \end{align*}
      The unit is $\Delta_\aug[-1]$.
    \item The category of simplicial sets carries
      a monoidal structure with the product,
      again called the \emph{join}, given by
      $K \join L = i^*(\Ran_i K \join \Ran_i L)$.
      The unit is the empty simplicial set.
  \end{enumerate}
\end{proposition}

\begin{proof}
  The first statement follows from the classical theorem
  of Day \cite{day}*{Theorem 3.3}.
  The second one can be proven by observing that $\Ran_i$ embeds $\sSet$
  fully and faithfully into $\asSet$ with the essential image consisting
  of augmented simplicial sets $X$ with $X_{-1}$ a singleton.
  Under this identification the join of augmented simplicial sets restricts to
  the join of simplicial sets.
\end{proof}

The category of small categories embeds as a full category of $\sSet$
via the nerve functor and the join product restricts to the category
of small categories.
Explicitly, given small categories $I$ and $J$ the join $I \join J$
is defined as follows.
The set of objects of $I \join J$
is the coproduct of the sets of objects of $I$ and $J$ and
\begin{align*}
  (I \join J)(x, y) =
  \begin{cases}
    I(x, y)   & \text{if } x, y \in I \text{,} \\
    J(x, y)   & \text{if } x, y \in J \text{,} \\
    *         & \text{if } x \in I, y \in J \text{,} \\
    \emptyset & \text{if } x \in J, y \in I \text{.}
  \end{cases}
\end{align*}
The composition of $I \join J$ is the unique composition
that restricts to the compositions of $I$ and $J$.

For example $[0] \join J$ is formed by adjoining an initial object to $J$
(a new one if $J$ already had one).
If $J$ is discrete, then colimits over $[0] \join J$
are called \emph{wide pushouts}.
(They reduce to classical pushouts when $J$ has exactly two objects.)

The join monoidal structure on simplicial set is not closed
and the join doesn't preserve all colimits in either of its variables.
However, a slightly weaker statement holds.
First, we need to observe that for any simplicial set $K$
the functor $K \join \uvar \from \sSet \to \sSet$
lifts to a functor $\sSet \to K \slice \sSet$ (also denoted by $K \join \uvar$.)
Such a lift is defined by the following composite
\begin{ctikzpicture}
  \matrix [diagram]
  {
    |(S)| \sSet &[-1.5em] |(ES)| \emptyset \slice \sSet
    & |(AS)| \Delta_\aug[-1] \slice \asSet &[1em] |(RS)| \Ran_i K \slice \asSet
    &[-1.5em] |(KS)| K \slice \sSet \text{.} \\
  };

  \draw[->] (S)  to (ES);

  \draw[->] (ES) to node[above] {$\Ran_i$}               (AS);
  \draw[->] (AS) to node[above] {$\Ran_i K \join \uvar$} (RS);
  \draw[->] (RS) to node[above] {$i^*$}                  (KS);
\end{ctikzpicture}

\begin{proposition}\label{join-colimits}
  For each simplicial set $K$,
  the functor $K \join \uvar \from \sSet \to K \slice \sSet$ preserves colimits.
  In particular, the functor $K \join \uvar \from \sSet \to \sSet$
  preserves pushouts and sequential colimits
  and carries coproducts to wide pushouts under $K$.
  (The same statement holds for $\uvar \join K$.)
\end{proposition}

\begin{proof}
  For any cocomplete category $\cat{C}$ and $X \in \cat{C}$ colimits over $J$
  in $X \slice \cat{C}$ are computed as colimits over $[0] \join J$
  in $\cat{C}$.
  Thus a colimit preserving functor $F \from \cat{C} \to \cat{D}$ induces
  a colimit preserving functor $X \slice \cat{C} \to FX \slice \cat{D}$.

  It follows that in the composite above all the functors preserve colimits.
  (Note that $\Ran_i$ doesn't preserve all colimits
  as a functor $\sSet \to \asSet$ but it does
  as a functor $\emptyset \slice \sSet \to \Delta_\aug[-1] \slice \asSet$.)

  The final statement holds since the inclusion $J \ito [0] \join J$
  is cofinal whenever $J$ is connected
  and $[0] \join J$ is the indexing category
  for wide pushouts if $J$ is discrete.
\end{proof}

\begin{corollary}\label{join-slice}
  For each simplicial set $K$ the functor
  $K \join \uvar \from \sSet \to K \slice \sSet$ has a right adjoint denoted by
  $(X \from K \to M) \mapsto X \uslice M$.
  ($X \slice M$ is called the \emph{slice of $M$ under $X$}.)
\end{corollary}

\begin{proof}
  Since $K \join \uvar$ is a colimit preserving functor on a category
  of presheaves its right adjoint is given by an explicit formula
  $(X \uslice M)_m = K \slice \sSet(K \join \simp{m}, M)$.
\end{proof}

\begin{lemma}\label{inner-isofibration-slice}
  Let $P \from \qcat{C} \fto \qcat{D}$ be a inner isofibration
  of quasicategories and $X \from K \to \qcat{C}$ a~diagram.
  Then the induced map $X \uslice \qcat{C} \to PX \uslice \qcat{D}$
  is an inner isofibration.
  In particular, $X \uslice \qcat{C}$ is a quasicategory.
\end{lemma}

\begin{proof}
  This follows from \cite{jo}*{Theorem 3.19(i) and Proposition 4.10}.
\end{proof}

For any simplicial set $K$ we define the \emph{under-cone} on $K$
as $K^\ucone = K \join \simp{0}$.

\begin{definition}
  Let $\qcat{C}$ be a quasicategory and let $X \from K \to \qcat{C}$
  be any simplicial map
  (which we consider as a $K$-indexed diagram in $\qcat{C}$).
  \begin{enumerate}[label=\dfn]
  \item A \emph{cone under $X$} is a diagram
    $S \from K^\ucone \to \qcat{C}$ \st{} $S|K = X$.
  \item
    A cone $S$ under $X$ is \emph{universal} or a \emph{colimit of $X$}
    if for any $m > 0$ and any diagram of solid arrows
    \begin{ctikzpicture}
      \matrix [diagram]
      {
        |(Kb)| K \join \bdsimp{m} & |(C)| \qcat{C} \\
        |(Ks)| K \join \simp{m} \\
      };

      \draw[->] (Kb) to node[above] {$U$} (C);
      \draw[->] (Kb) to (Ks);

      \draw[->,dashed] (Ks) to (C);
    \end{ctikzpicture}
    where $U|K^\ucone = S$ there exists a dashed arrow
    making the diagram commute.
  \item An \emph{initial object} of $\qcat{C}$ is a colimit of
    the unique empty diagram in $\qcat{C}$.
  \item A simplicial map $f \from K \to L$
    is \emph{cofinal} if for every quasicategory $\qcat{C}$
    and every universal cone $S \from L^\ucone \to \qcat{C}$ the induced cone
    $S f^\ucone$ is also universal.
  \item
    The quasicategory $\qcat{C}$ is \emph{finitely cocomplete}
    if for every finite simplicial set $K$
    every diagram $K \to \qcat{C}$ has a colimit.
  \item
    A functor $F \from \qcat{C} \to \qcat{D}$ between
    finitely cocomplete quasicategories is \emph{exact}
    (or \emph{preserves finite colimits}) if for every finite simplicial set $K$
    and every universal cone $S \from K^\ucone \to \qcat{C}$
    the cone $FS$ is also universal.
  \end{enumerate}
\end{definition}

For any quasicategory $\qcat{C}$ and objects $x, y \in \qcat{C}$
it is possible to construct the \emph{mapping space} $\qcat{C}(x,y)$,
though there is no preferred such construction.
A variety of (equivalent) possibilities is discussed in \cite{ds}.
Then an object $x$ is initial \iff{} for every $y$
the mapping space $\qcat{C}(x,y)$ is contractible
(see \cite{l}*{Proposition 1.2.12.4}) and the next lemma allows us to translate
this observation to general colimits.
However, it turns out that the definition given above is more convenient.

\begin{lemma}\label{colim-init}
  A cone $S$ under $X$ is universal \iff{} it is an initial object
  of $X \uslice \qcat{C}$.
\end{lemma}

\begin{proof}
  This follows directly from \Cref{join-slice}.
\end{proof}

In the remainder of this subsection we discuss the counterparts of
classical statements of category theory saying that
colimits are essentially unique and invariant under equivalences.
For a quasicategory $\qcat{C}$ and a diagram $X \from K \to \qcat{C}$
we let $(X \uslice \qcat{C})^\univ$ denote the simplicial subset
of $X \uslice \qcat{C}$ consisting of these simplices
whose all vertices are universal.

\begin{lemma}
  The simplicial set $(X \uslice \qcat{C})^\univ$ is empty
  or a contractible Kan complex.
\end{lemma}

\begin{proof}
  A simplicial set is empty or a contractible Kan complex \iff{}
  it has the \rlp{} \wrt{} the boundary inclusions $\bdsimp{m} \ito \simp{m}$
  for all $m > 0$.
  For $(X \uslice \qcat{C})^\univ$ such lifting problems
  are equivalent to the lifting problems
  \begin{ctikzpicture}
    \matrix[diagram]
    {
      |(b)| K \join \bdsimp{m} & |(C)| \qcat{C} \\
      |(s)| K \join \simp{m} \\
    };

    \draw[->] (b) to node[above] {$U$} (C);
    \draw[->] (b) to (s);

    \draw[->,dashed] (s) to (C);
  \end{ctikzpicture}
  with $U|(K \join \{ i \})$ universal for each $i \in [m]$
  which have solutions by the definition of universal cones.
\end{proof}

\begin{corollary}\label{colim-equivalent}
  If $X \from K \to \qcat{C}$ is a diagram in a quasicategory
  and $S$ and $T$ are two universal cones under $X$,
  then they equivalent under $X$, i.e.\ as objects of $X \uslice \qcat{C}$.
\end{corollary}

\begin{proof}
  The simplicial set $(X \uslice \qcat{C})^\univ$ is non-empty
  and thus a contractible Kan complex by the previous lemma.
  Hence it has the \rlp{} \wrt{} the inclusion $\bdsimp{1} \ito \nwiso$
  which translates to the lifting property
  \begin{ctikzpicture}
    \matrix[diagram]
    {
      |(b)| K \join \bdsimp{1} & |(C)| \qcat{C} \\
      |(E)| K \join \nwiso \\
    };

    \draw[->] (b) to node[above] {$[S, T]$} (C);
    \draw[->] (b) to (E);

    \draw[->,dashed] (E) to (C);
  \end{ctikzpicture}
  which yields an equivalence of $S$ and $T$.
\end{proof}

\begin{lemma}\label{init-equivalent}
  If $\qcat{C}$ is a quasicategory and $X$ and $Y$ are equivalent objects
  of $\qcat{C}$, then $X$ is initial \iff{} $Y$ is.
\end{lemma}

\begin{proof}
  Assume that $X$ is initial and let $U \from \bdsimp{m} \to \qcat{C}$
  be \st{} $U|\simp{0} = Y$.
  We can consider an equivalence from $X$ to $Y$ as a diagram
  $f \from \simp{0} \join \simp{0} \to \qcat{C}$.
  Then by the universal property of $X$ there is a diagram
  $\simp{0} \join \bdsimp{m}$ extending both $f$ and $U$.
  (We can iteratively choose extensions over $\simp{0} \join \simp{k}$
  for all faces $\simp{k} \ito \bdsimp{m}$.)
  This diagram is a special outer horn
  (under the isomorphism $\simp{0} \join \bdsimp{m} \iso \horn{m+1,0}$)
  and thus has a filler by \Cref{special-horns}.
  Therefore $U$ extends over $\simp{m}$ and hence $Y$ is initial.
\end{proof}

  \subsection{Homotopy theory of cocomplete quasicategories}

Our goal is to compare cofibration categories to quasicategories,
but we expect cofibration categories to correspond
to finitely cocomplete quasicategories, not to arbitrary ones.
In this subsection we will restrict the fibration structure
of \Cref{fibcat-of-quasicats} to the subcategory
of finitely cocomplete quasicategories and exact functors.

First, we need two lemmas about lifting colimits along inner isofibrations.

\begin{lemma}\label{q-pullback-colimits}
  Let
  \begin{ctikzpicture}
    \matrix [diagram]
    {
      |(P)| \qcat{P} & |(E)| \qcat{E} \\
      |(C)| \qcat{C} & |(D)| \qcat{D} \\
    };

    \draw[->] (P) to node[above] {$G$} (E);
    \draw[->] (P) to node[left]  {$Q$} (C);

    \draw[->]  (C) to node[below] {$F$} (D);
    \draw[fib] (E) to node[right] {$P$} (D);
  \end{ctikzpicture}
  be a pullback square of quasicategories where $P$ is an inner isofibration.
  Let $S \from K^\ucone \to \qcat{P}$ be a cone.
  If all $GS$, $QS$ and $PGS = FQS$ are universal, then so is $S$.
\end{lemma}

\begin{proof}
  Under these assumptions the square
  \begin{ctikzpicture}
    \matrix [diagram]
    {
      |(P)| X \uslice \qcat{P} & |(E)| GX \uslice \qcat{E} \\
      |(C)| QX \uslice\qcat{C} & |(D)| PGX \uslice \qcat{D} \\
    };

    \draw[->] (P) to node[above] {$G$} (E);
    \draw[->] (P) to node[left]  {$Q$} (C);

    \draw[->]  (C) to node[below] {$F$} (D);
    \draw[fib] (E) to node[right] {$P$} (D);
  \end{ctikzpicture}
  (where $X = S|K$) is also a pullback along an inner isofibration
  by \Cref{inner-isofibration-slice}.
  Hence it suffices to verify the conclusion for initial objects.

  Thus assume that $K = \emptyset$ and let $m > 0$
  and $U \from \bdsimp{m} \to \qcat{P}$ be \st{} $U|\simp{0} = S$.
  Then we have
  \begin{align*}
    GU|\simp{0} = GS \text{ and } QU|\simp{0} = QS
  \end{align*}
  and since both $GS$ and $QS$ are initial we can find
  $V_{\qcat{E}} \in \qcat{E}_m$ and $V_{\qcat{C}} \in \qcat{C}_m$
  \st{} $V_{\qcat{E}}|\bdsimp{m} = GU$ and $V_{\qcat{C}}|\bdsimp{m} = QU$.
  Next, define $\tilde V \from \bdsimp{m+1} \to \qcat{D}$
  by replacing the $1$st face of $PV_{\qcat{E}}\dgn_1|\bdsimp{m+1}$
  with $FV_{\qcat{C}}$ and $\tilde W \from \horn{m+1,1} \to \qcat{E}$
  by setting it to $V_{\qcat{E}}\dgn_1|\horn{m+1,1}$.

  By the assumption $PGS$ is initial and $\tilde V|\simp{0} = PGS$
  so $\tilde V$ extends to $V \in \qcat{D}_{m+1}$.
  Then we have a commutative square
  \begin{ctikzpicture}
    \matrix [diagram]
    {
      |(h)| \horn{m+1,1} & |(E)| \qcat{E} \\
      |(s)| \simp{m+1}   & |(D)| \qcat{D} \\
    };

    \draw[->] (h) to node[above] {$\tilde W$} (E);

    \draw[->] (h) to (s);

    \draw[->]  (s) to node[below] {$V$} (D);
    \draw[fib] (E) to node[right] {$P$} (D);
  \end{ctikzpicture}
  which admits a lift $W$ since $P$ is an inner isofibration and $0 < 1 < m+1$.
  We have $FV_{\qcat{C}} = PW\face_1$ and thus $(V_{\qcat{C}}, W\face_1)$
  is an $m$-simplex of $\qcat{P}$ whose boundary is $U$.
  Hence $S$ is initial.
\end{proof}

\begin{lemma}\label{q-lifting-colimits}
  Let $P \from \qcat{C} \fto \qcat{D}$ be an inner isofibration,
  $X \from K \to \qcat{C}$ a diagram and $T \from K^\ucone \to \qcat{D}$
  a colimit of $PX$.
  If $X$ has a colimit in $\qcat{C}$ which is preserved by $P$,
  then there exists a colimit $S \from K^\ucone \to \cat{C}$ of $X$
  \st{} $PS = T$.
\end{lemma}

\begin{proof}
  Let $\tilde S \from K^\ucone \to \qcat{C}$ be some colimit of $X$.
  Since both $T$ and $P \tilde S$ are universal, we have a simplicial map
  $U \from K \join \nwiso \to \qcat{D}$ \st{}
  $U|(K \join \bdsimp{1}) = [T, P \tilde S]$ by \Cref{colim-equivalent}.
  The conclusion now follows
  from \Cref{inner-isofibration-slice,init-equivalent}.
\end{proof}

The homotopical content of the next proposition is the same as
that of \cite{l}*{Lemma 5.4.5.5}.
However, we need a stricter point-set level statement.

\begin{proposition}\label{pullback-of-fc-qcats}
  Let $F \from \qcat{C} \to \qcat{D}$ and $P \from \qcat{E} \fto \qcat{D}$
  be exact functors between finitely cocomplete quasicategories
  with $P$ an inner isofibration.
  Then a pullback of $P$ along $F$ exists in the category of finitely cocomplete
  quasicategories and exact functors.
\end{proposition}

\begin{proof}
  Form a pullback of $P$ along $F$ in the category of quasicategories.
  \begin{ctikzpicture}
    \matrix [diagram]
    {
      |(P)| \qcat{P} & |(E)| \qcat{E} \\
      |(C)| \qcat{C} & |(D)| \qcat{D} \\
    };

    \draw[->] (P) to node[above] {$G$} (E);
    \draw[->] (P) to node[left]  {$Q$} (C);

    \draw[->]  (C) to node[below] {$F$} (D);
    \draw[fib] (E) to node[right] {$P$} (D);
  \end{ctikzpicture}
  We will check that this square is also a pullback
  in the category of finitely cocomplete quasicategories and exact functors.

  First, we verify that $\qcat{P}$ has finite colimits.
  Let $X \from K \to \qcat{P}$ be a diagram with $K$ finite.
  Let $S \from K^\ucone \to \qcat{C}$ be a colimit of $QX$,
  then $FS$ is a colimit of $FQX = PGX$ in $\qcat{D}$.
  \Cref{q-lifting-colimits} implies that we can choose a colimit $T$ of $GX$
  in $\qcat{E}$ so that $PT = FS$.
  Then it follows by \Cref{q-pullback-colimits} that $(S, T)$
  is a colimit of $X = (QX, GX)$ in $\qcat{P}$.

  It remains to see that given a square
  \begin{ctikzpicture}
    \matrix [diagram]
    {
      |(F)| \qcat{F} & |(E)| \qcat{E} \\
      |(C)| \qcat{C} & |(D)| \qcat{D} \\
    };

    \draw[->] (F) to (E);
    \draw[->] (F) to (C);

    \draw[->] (C) to node[below] {$F$} (D);
    \draw[->] (E) to node[right] {$P$} (D);
  \end{ctikzpicture}
  of finitely cocomplete quasicategories and exact functors,
  the induced functor $\qcat{F} \to \qcat{P}$ preserves finite colimits.
  Indeed, this follows directly from \Cref{q-pullback-colimits}.
\end{proof}

\begin{theorem}\label{fibcat-of-fc-quasicats}
  The category of small finitely cocomplete quasicategories with exact functors
  as morphisms, categorical equivalences as weak equivalences
  and (exact) inner isofibrations as fibrations is a fibration category.
\end{theorem}

In fact, this fibration category is homotopy complete,
i.e.\ it has all small homotopy limits.
This will be explained in \Cref{sec:infinite-hocolims}.

\begin{proof}
  By \Cref{fibcat-of-quasicats} it suffices to observe that
  \begin{enumerate}[label=\thm]
  \item a terminal quasicategory
    is also a terminal finitely cocomplete quasicategory (which is clear),
  \item a pullback (in the category of all quasicategories)
    of finitely cocomplete quasicategories and exact functors
    one of which is an inner isofibration is also a pullback in the category
    of finitely cocomplete quasicategories which follows by (the proof of)
    \Cref{pullback-of-fc-qcats},
  \item for a finitely cocomplete quasicategory $\qcat{C}$ the functor
    $\qcat{C}^\nwiso \to \qcat{C} \times \qcat{C}$ is an exact functor
    between finitely cocomplete quasicategories.
    Indeed, $\qcat{C}^\nwiso$ is finitely cocomplete
    since it is categorically equivalent to $\qcat{C}$
    (by \Cref{colim-init,init-equivalent})
    and $\qcat{C} \times \qcat{C}$ is finitely cocomplete by (2).
    Finally, $\qcat{C}^\nwiso \to \qcat{C} \times \qcat{C}$
    preserves finite colimits by (2)
    since both projections $\qcat{C}^\nwiso \to \qcat{C}$ do. \qedhere
  \end{enumerate}
\end{proof}

  \section{Quasicategories of frames\\in cofibration categories}
  \label{ch:qcats-of-frames}
  In this section we will associate to every cofibration category $\cat{C}$
a corresponding quasicategory called
the \emph{quasicategory of frames in $\cat{C}$} obtaining an exact functor
between the fibration categories established in \Cref{ch:cof-cats,ch:qcats}.
Later, we will prove that this functor is a weak equivalence
between these fibration categories.

\subsection{Definitions and basic properties}
\label{sec:qcats-of-frames}

Before introducing quasicategories of frames we need to explain
a preliminary construction which will play an essential role in the remainder
of this paper.

Let $\pDelta$ denote the subcategory of injective maps in $\Delta$
and let $J$ be a \hore{} category.
We construct a direct \hore{} category $DJ$ and a \hore{} functor
$p_J \from DJ \to J$ as follows.
The underlying category of $DJ$ is the slice $\pDelta \slice J$,
i.e.\ objects are all functors $[m] \to J$ for all $m$ and a morphism
from $x \from [m] \to J$ to $y \from [n] \to J$
is an injective order preserving map $\phi \from [m] \ito [n]$
\st{} $x = y \phi$.
The functor $p_J \from \pDelta \slice J \to J$ is defined by evaluating
$[m] \to J$ at $m$ and the weak equivalences in $DJ$ are created by $p_J$.
Then $DJ$ is \hore{} category, $p_J$ is a \hore{} functor
and $DJ$ is also direct (by setting the degree of $[m] \to J$ to $m$).
We can think of $DJ$ as a \emph{direct approximation} to $J$.
Observe that $D$ is a functor from \hore{} categories to \hore{} categories
and that $DJ$ has a non-trivial \hore{} structure
even if $J$ has the trivial one (unless $J$ is empty).
This construction has multiple motivations which will be given
right after the definition of quasicategories of frames below.

First, we need to verify that Reedy cofibrant diagrams over $DJ$
are well behaved \wrt{} \hore{} functors $I \to J$.
If $f$ is such a functor we will abbreviate the induced functor
$(Df)^* \from \cat{C}^{DJ}_\Reedy \to \cat{C}^{DI}_\Reedy$ to $f^*$
to simplify the notation.
Recall that $\cat{C}^{DJ}_\Reedy$ refers to the cofibration category
of \hore{} Reedy cofibrant diagrams $DJ \to \cat{C}$
with levelwise weak equivalences and Reedy cofibrations which exists
by \Cref{colims-in-cofcats} since $DJ$ has finite latching categories.

\begin{lemma}\label{D-exact}
  Let $\cat{C}$ be a cofibration category.
  If $f \from I \to J$ is a \hore{} functor of small \hore{} categories,
  then the induced functor
  $f^* \from \cat{C}^{DJ}_\Reedy \to \cat{C}^{DI}_\Reedy$ is exact.
  If $f$ is injective on objects and faithful, then $f^*$ is a fibration.
\end{lemma}

\begin{proof}
  Both statements follow from \Cref{sieve-exact} since
  if $f$ is injective on objects and faithful, then $Df$ is a sieve.
\end{proof}

For a cofibration category $\cat{C}$ we define the
\emph{quasicategory of frames in $\cat{C}$} as a simplicial set denoted by
$\nf \cat{C}$ where $(\nf \cat{C})_m$ is the set
of all \hore{} Reedy cofibrant diagrams $D[m] \to \cat{C}$
($[m]$ is a \hore{} category with only identities as weak equivalences).
The simplicial structure is given by functoriality of $D$ (using \Cref{D-exact}
to see that simplicial operators preserve Reedy cofibrancy).
Since exact functors of cofibration categories
preserve Reedy cofibrant diagrams, $\nf$ is a functor from the category
of cofibration categories to the category of simplicial sets.

\begin{remark}
  As a side note, we point out that this construction can be enhanced
  as follows.
  If $\hat{[n]}$ denotes the \hore{} poset $[n]$ with all morphisms
  as weak equivalences, then the bisimplicial set
  \begin{align*}
    [m], [n] \mapsto \{ \text{\hore{} Reedy cofibrant diagrams }
    D([m] \times \hat{[n]}) \to \cat{C} \}
  \end{align*}
  is a complete Segal space with $\nf \cat{C}$ as its $0$th row.
\end{remark}

This definition has a threefold motivation.
First, the objects of $\nf \cat{C}$ are called \emph{frames} in $\cat{C}$.
They are counterparts to frames in a model category $\cat{M}$,
i.e.\ homotopically constant Reedy cofibrant diagrams $\Delta \to \cat{M}$
which can be used to enrich the homotopy category $\Ho \cat{M}$
in the homotopy category of simplicial sets as explained
in \cite{ho}*{Chapter 5}.
In cofibration categories we are forced to replace $\Delta$ by $\pDelta$
and then homotopically constant diagrams over $\pDelta$ are precisely
the \hore{} diagrams over $D[0]$.
Again, one can prove using such frames that the homotopy category $\Ho \cat{C}$
is enriched in the category of homotopy types,
see \cite{s}*{Theorems 3.10 and 3.17}.\footnote{
This result differs from its counterpart for model categories
since it uses presimplicial sets (a.k.a. $\Delta$-sets or semisimplicial sets)
as models of homotopy types.
Presimplicial sets are less well-behaved than simplicial sets,
but their homotopy theory is equivalent to that of simplicial sets.}
Our construction can be seen as an alternative way of using frames
to enrich $\Ho \cat{C}$ in homotopy types, namely, by using the mapping spaces
of the quasicategory $\nf \cat{C}$.

\label{fractions-motivation}
The second motivation is that $\nf \cat{C}$ can be seen as an enhancement
of the calculus of fractions.
Let $\Sd[m]$ denote the poset of non-empty subsets of $m$.
It can be seen as the full subcategory of $D[m]$ spanned
by the non-degenerate simplices of $[m]$ as explained in more detail
on p.\ \pageref{subdiv}.
\HoRe{} Reedy cofibrant diagrams over $D[m]$ can be seen as resolutions
of their restrictions to $\Sd[m]$.
Therefore an object of $\nf \cat{C}$ is a resolution of an object of $\cat{C}$
and its morphism is a resolution of a diagram of the form
\begin{ctikzpicture}
  \matrix[diagram]
  {
    |(0)| X_0 & |(01)| X_{01} & |(1)| X_1 \text{,} \\
  };

  \draw[->] (0) to (01);

  \draw[->] (1) to node[above] {$\we$} (01);
\end{ctikzpicture}
i.e.\ a left fraction from $X_0$ to $X_1$.
Similarly, a $2$-simplex of $\nf \cat{C}$ is a resolution of a diagram
of the form
\begin{ctikzpicture}
  \matrix[diagram]
  {
    & & |(b)| X_1 \\
    \\ \\
    |(a)| X_0 & & & & |(c)| X_2 \\
  };

  \node (ab)  at (barycentric cs:a=1,b=1,c=0) {$X_{01}$};
  \node (bc)  at (barycentric cs:a=0,b=1,c=1) {$X_{12}$};
  \node (ac)  at (barycentric cs:a=1,b=0,c=1) {$X_{02}$};
  \node (abc) at (barycentric cs:a=1,b=1,c=1) {$X_{012}$};

  \draw[->] (a) to (ab);
  \draw[->] (b) to node[above left] {$\we$} (ab);
  \draw[->] (b) to (bc);
  \draw[->] (c) to node[above right] {$\we$} (bc);
  \draw[->] (a) to (ac);
  \draw[->] (c) to node[below] {$\we$} (ac);

  \draw[->] (ab) to (abc);
  \draw[->] (bc) to node[above left] {$\we$} (abc);
  \draw[->] (ac) to node[right] {$\we$} (abc); 
\end{ctikzpicture}
which consists of two fractions going from $X_0$ to $X_1$
and from $X_1$ to $X_2$ along with a composite fraction
going directly from $X_0$ to $X_2$.
Such diagrams simultaneously encode the composition of left fractions
and the notion of equivalence of fractions which is made precise
in the proof of \Cref{dgrm-relative}.
Higher simplices encode the higher homotopy of the mapping spaces of $\cat{C}$
in a similar manner.

It might be tempting to simplify the definition of $\nf \cat{C}$ by replacing
$D[m]$ with $\Sd[m]$. This would not work since functors $\Sd[m] \to \Sd[n]$
induced by degeneracy operators $[m] \sto [n]$ do not respect
Reedy cofibrant diagrams and thus this modification would not even yield
a simplicial set.

Finally, the quasicategory of frames can be motivated by the discussion
in \Cref{sec:infinite-hocolims} which suggests
that \hore{} Reedy cofibrant diagrams $DJ \to \cat{C}$ contain the information
about all homotopy colimits in $\cat{C}$.
In fact, this information can be reduced just to
\hore{} Reedy cofibrant diagrams $D[m] \to \cat{C}$
as implied by \Cref{nf-representation}.
These observations will be formalized in the next theorem that says,
among other things, that the functor $\nf$ converts homotopy colimits
in the sense of homotopical algebra to colimits in quasicategories.
(A more precise statement to this effect is \Cref{nerve-cones}.)

\begin{theorem}\label{nerve-exact}
  The functor $\nf$ takes values in finitely cocomplete quasicategories
  and is an exact functor from the fibration category
  of \Cref{fibcat-of-cofcats} to the fibration category
  of \Cref{fibcat-of-fc-quasicats}.
\end{theorem}

One part of the proof is quite easy.

\begin{proposition}\label{nerve-limits}
  The functor $\nf$ preserves a terminal object and pullbacks along fibrations.
\end{proposition}

\begin{proof}
  The preservation of a terminal object is clear.
  In order to see that pullbacks are also preserved it suffices to verify
  that given a pullback square
  \begin{ctikzpicture}
    \matrix [diagram]
    {
      |(P)| \cat{P} & |(E)| \cat{E} \\
      |(C)| \cat{C} & |(D)| \cat{D} \text{.} \\
    };

    \draw[->] (P) to node[above] {$G$} (E);
    \draw[->] (P) to node[left]  {$Q$} (C);

    \draw[->]  (C) to node[below] {$F$} (D);
    \draw[fib] (E) to node[right] {$P$} (D);
  \end{ctikzpicture}
  of cofibration categories and exact functors
  a functor $X \from D[m] \to \cat{P}$ is a \hore{} Reedy cofibrant diagram
  \iff{} both $QX$ and $GX$ are.
  This follows since latching objects in $\cat{P}$ are computed pointwise
  in $\cat{C}$ and $\cat{E}$ by \Cref{q-pullback-colimits}.
\end{proof}

We will commit the next subsection to the verification that $\nf$ preserves
(acyclic) fibrations. Before that we need to establish some basic properties
of this functor.

First, we will give another version of the $D$ construction.
For a simplicial set $K$ we define a \hore{} direct category $DK$ as follows.
The underlying category of $DK$ is the category of elements of $K$
but \emph{only with face operators} as morphisms,
i.e.\ objects of $DK$ are all simplices of $K$ and a morphism from $x \in K_m$
to $y \in K_n$ is an injective order preserving map $\phi \from [m] \ito [n]$
\st{} $x = y \phi$.

Such a morphism is a \emph{generating} weak equivalence if $y \nu$
is a degenerate edge of $K$ where $\nu \from [1] \to [n]$ is defined by
$\nu(0) = \phi(m)$ and $\nu(1) = n$.
The generating weak equivalences do not necessarily satisfy
the ``2 out of 6'' property
(they are not even closed under composition in general).
Thus we define the subcategory of weak equivalences as the smallest subcategory
containing the generating weak equivalences and satisfying
the ``2 out of 6'' property.
Of course, in order to verify that a functor from $DK$ to a \hore{} category
is \hore{} it suffices to check that it sends the generating weak equivalences
to weak equivalences.

This construction is functorial in $K$.
Moreover, the next lemma says that if $K$ is the nerve of a category $J$,
then $DK$ coincides with $DJ$ in the sense of the previous definition.

\begin{lemma}
  \index{2 out of 6}
  Let $J$ be a category with the trivial \hore{} structure.
  Then the \hore{} categories $DJ$ and $DNJ$ coincide.
\end{lemma}

\begin{proof}
  The underlying categories of $DJ$ and $DNJ$ are the same by definition.
  The generating weak equivalences of $DNJ$ are mapped to identities
  by $p_J \from DJ \to J$ and hence it suffices to see
  that every weak equivalence created by $p_J$ can be obtained from
  the generating ones by applying the ``2 out of 6'' property.
  Let $\phi, \psi \in DJ$ and consider a morphism $\phi \to \psi$
  mapped by $p_J$ to an isomorphism $f \from x \to y$ of $J$.
  Then we have a diagram
  \begin{ctikzpicture}
    \matrix[diagram]
    {
      |(x)|  x    & |(xy)| xy   & |(xyx)| xyx & |(xyxy)| xyxy \\
      |(ph)| \phi & |(ps)| \psi \\
    };

    \draw[->] (x)   to (xy);
    \draw[->] (xy)  to (xyx);
    \draw[->] (xyx) to (xyxy);

    \draw[->] (x)  to[bend left]  node[above] {$\we$} (xyx);
    \draw[->] (xy) to[bend right] node[below] {$\we$} (xyxy);

    \draw[->] (x)  to node[left] {$\we$} (ph);
    \draw[->] (xy) to node[left] {$\we$} (ps);

    \draw[->] (ph) to (ps);
  \end{ctikzpicture}
  in $DJ$ where $xyxy$ denotes the sequence
  \begin{ctikzpicture}
    \matrix[diagram]
    {
      |(x0)| x & |(y0)| y & |(x1)| x & |(y1)| y \\
    };

    \draw[->] (x0) to node[above] {$f$} (y0);
    \draw[->] (x1) to node[above] {$f$} (y1);

    \draw[->] (y0) to node[above] {$f^{-1}$} (x1);
  \end{ctikzpicture}
  and the remaining objects in the first row are its initial segments.
  The indicated morphisms are generating weak equivalences
  and hence by ``2 out of 6'' $\phi \to \psi$
  is also a weak equivalence of $DNJ$.
\end{proof}

\begin{lemma}\label{D-preserves-colimits}
  The functor $D \from \sSet \to \Cat$ (i.e.\ when we disregard
  the \hore{} structures of $DK$s) preserves colimits.
\end{lemma}

\begin{proof}
  Since $N \from \Cat \to \sSet$ is fully faithful it reflects colimits
  (see \cite{bo}*{Proposition 2.2.9}).
  Thus it will suffice to verify that the composite functor $K \mapsto NDK$
  preserves colimits.
  This follows from the fact that
  \begin{align*}
    (NDK)_m = \bigcoprod_{[j_0] \ito [j_1] \ito \ldots \ito [j_m]} K_{j_m}
    & \text{.} \qedhere
  \end{align*}
\end{proof}

Let $X \from DK \to \cat{C}$ be a \hore{} Reedy cofibrant diagram.
For each simplex $x \from \simp{m} \to K$
consider the restriction $x^* X \from D[m] \to \cat{C}$
which is an $m$-simplex of $\nf \cat{C}$.
(Recall that $x^*$ is an abbreviation of $(Dx)^*$.)
These simplices fit together to form a simplicial map $K \to \nf \cat{C}$.

\begin{proposition}\label{nf-representation}
  Let $\cat{C}$ be a cofibration category and $K$ a simplicial set.
  The map described above is a natural bijection between
  \begin{itemize}
  \item the set of \hore{} Reedy cofibrant diagrams $DK \to \cat{C}$
  \item and the set of simplicial maps $K \to \nf \cat{C}$.
  \end{itemize}
\end{proposition}

\begin{proof}
  Denote the former set by $R(DK, \cat{C})$
  and observe that $R(D\uvar, \cat{C})$ is a contravariant functor
  from simplicial sets to sets.
  The statement says that this functor is representable
  and the representing object is $\nf \cat{C}$.
  This will follow
  if we can verify
  that if we consider any simplicial set $K$ as a colimit of its simplices,
  then this colimit is preserved (i.e.\ carried to a limit)
  by $R(D\uvar, \cat{C})$.

  First, note that by \Cref{D-preserves-colimits} the functor
  $\Cat(D\uvar, \cat{C})$ carries colimits to limits.
  Since $R(D\uvar, \cat{C})$ is a subfunctor of $\Cat(D\uvar, \cat{C})$
  it will suffice to see that a diagram $X \from DK \to \cat{C}$ is \hore{}
  and Reedy cofibrant \iff{} for all $x \in K_m$ the induced diagram $x^* X$
  is \hore{} and Reedy cofibrant.
  The cofibrancy statement follows by (the argument of) \Cref{D-exact}.

  It is clear that if $X$ is \hore{} then so are all $x^* X$.
  In order to prove the converse it suffices to consider
  the generating weak equivalences of $DK$.
  Let $x \in K_m$, $y \in K_n$ and $\phi \from [m] \ito [n]$
  be \st{} $x = y \phi$ and $y \nu$ is a degenerate edge
  where $\nu \from [1] \to [n]$ is defined by $\nu(0) = \phi(m)$
  and $\nu(1) = n$.
  We need to prove that $X \phi$ is a weak equivalence in $\cat{C}$.
  First, let's assume that $\phi(m) = n$, then $\phi$ is a weak equivalence
  when seen as a morphism $\phi \to \id_{[n]}$ in $D[n]$.
  Therefore $X \phi = (y^* X) \phi$ is a weak equivalence since $y^* X$
  is a \hore{} diagram.
  Next, assume that $\phi(m) < n$, then $\nu$ is injective and can be seen
  as a morphism $y \nu \to y$ in $DK$ and we have a commutative diagram
  on the left in $\pDelta$ which can be reinterpreted as a diagram in the middle
  in $DK$ which in turn yields the diagram on the right in $\cat{C}$
  (here $\epsilon_i \from [0] \to [k]$ is the morphism with image $i$).
  \begin{ctikzpicture}
    \matrix[diagram]
    {
      |(0)| [0] & |(m)| [m] & |(ye)| y \epsilon_m & |(yf)| y \phi &
      |(Xe)| X(y \epsilon_m) & |(Xf)| X(y \phi) \\
      |(1)| [1] & |(n)| [n] & |(yn)| y \nu        & |(y)| y &
      |(Xn)| X(y \nu)        & |(X)| X y \\
    };

    \draw[->] (0) to node[above] {$\epsilon_m$} (m);
    \draw[->] (1) to node[below] {$\nu$}        (n);

    \draw[->] (0) to node[left]  {$\epsilon_0$} (1);
    \draw[->] (m) to node[right] {$\phi$}       (n);

    \draw[->] (ye) to node[above] {$\epsilon_m$} (yf);
    \draw[->] (yn) to node[below] {$\nu$}        (y);

    \draw[->] (ye) to node[left]  {$\epsilon_0$} (yn);
    \draw[->] (yf) to node[right] {$\phi$}       (y);

    \draw[->] (Xe) to node[above] {$X \epsilon_m$} (Xf);
    \draw[->] (Xn) to node[below] {$X \nu$}        (X);

    \draw[->] (Xe) to node[left]  {$X \epsilon_0$} (Xn);
    \draw[->] (Xf) to node[right] {$X \phi$}       (X);
  \end{ctikzpicture}
  Now, $\epsilon_m$ and $\nu$ are weak equivalences when seen as morphisms
  of $D[m]$ and $D[n]$ respectively.
  Thus $X \epsilon_m$ and $X \nu$ are weak equivalences.
  The edge $y \nu$ is degenerate, i.e.\ $y \nu = y \epsilon_n \sigma_0$,
  so the diagram $(y \nu)^* X \from D[1] \to \cat{C}$ factors through
  $(y \epsilon_n)^* X \from D[0] \to \cat{C}$.
  Since all morphisms of $D[0]$ are weak equivalences it follows that
  $(y \nu)^* X$ sends all morphisms, including $\epsilon_0$ above,
  to weak equivalences thus $X \epsilon_0$ is a weak equivalence
  and hence so is $X \phi$.
\end{proof}

This immediately implies the following.

\begin{corollary}\label{simplicial-Reedy}
  Let $i \from K \to L$ be a simplicial map and $F \from \cat{C} \to \cat{D}$
  an exact functor between cofibration categories.
  Then $\nf F$ has the \rlp{} \wrt{} $i$ \iff{} $F$ has the \Rrlp{} \wrt{} $Di$.
  \qed
\end{corollary}

Our goal is to find some general procedure of solving such lifting problems.

  \subsection{Reedy lifting properties}
\label{sec:Rlps}

The results of \Cref{sec:hocolims-and-diagrams} give criteria
for verifying Reedy lifting properties.
In this subsection we verify these criteria for the inner horn inclusions
$D\horn{m,i} \ito D[m]$ and for $D[0] \to D\nwiso$.

The case of inner horn inclusions will be handled by comparing both $D[m]$
and $D\horn{m,i}$ to $[m]$ and various ``generalized inner horns''.

\begin{lemma}\label{Dm-m-he}
  For every $m \ge 0$ the functor $p_{[m]} \from D[m] \to [m]$ is a homotopy
  equivalence of \hore{} categories.
\end{lemma}

\begin{proof}
  Let $f \from [m] \to D[m]$ be the functor that sends $i \in [m]$
  to the standard inclusion $[i] \ito [m]$.
  This is a \hore{} functor and we have $p_{[m]} f = \id_{[m]}$.
  We will verify that $f p_{[m]}$ is weakly equivalent to $\id_{D[m]}$
  which will finish the proof.

  To this end define $s \from D[m] \to D[m]$ as follows.
  Represent an object $x \in D[m]$ as a non-empty finite non-decreasing sequence
  of elements of $[m]$.
  Then $s(x)$ is obtained by inserting one extra occurrence
  of each of the elements $0, 1, \ldots, p_{[m]}(x)$ into $x$.
  Every such element $i$ is added ``at the end'' of the (possibly empty) block
  of $i$s already present in $x$.
  This explains the functoriality of $s$.
  Namely, given $\phi \from x \to y$ and $i \le p_{[m]}(x)$,
  the map $s(\phi)$ acts on the ``old'' occurrences of $i$ as $\phi$ does
  and sends the ``new'' occurrences to the ``new'' occurrences.
  Thus the functor $s$ is \hore{} and admits natural weak equivalences
  \begin{ctikzpicture}
    \matrix[diagram]
    {
      |(id)| \id & |(s)| s & |(fp)| f p_{[m]} \\
    };

    \draw[->] (id) to node[above] {$\we$} (s);
    \draw[->] (fp) to node[above] {$\we$} (s);
  \end{ctikzpicture}
  where the map on the left inserts $x$ onto the ``old'' occurrences in $s(x)$
  and the right one inserts $f p_{[m]}(x)$ onto the ``new'' ones.
\end{proof}

Let $A \subseteq [m]$, we define the \emph{generalized horn} $\horn{m,A}$
as the simplicial subset of $\simp{m}$ generated by its codimension $1$ faces
lying opposite of vertices not in $A$.
Observe that $\horn{m,\{i\}} = \horn{m,i}$.

\begin{lemma}\label{skinny-horn}
  The inclusion functor $D\horn{m,{\{1,\ldots,m-1\}}} \ito D[m]$
  induces a weak equivalence
  $\cat{C}^{D[m]}_\Reedy \to \cat{C}^{D\horn{m,{\{1,\ldots,m-1\}}}}_\Reedy$
  for every cofibration category $\cat{C}$ and each $m \ge 2$.
\end{lemma}

\begin{proof}
  By \Cref{level-htpy-eq} it suffices to verify the statement
  for the levelwise structures and hence it will be enough to show that
  the composite $D\horn{m,{\{1,\ldots,m-1\}}} \ito D[m] \to [m]$
  induces a weak equivalence \wrt{} the levelwise structures.

  In the diagram
  \begin{ctikzpicture}
    \matrix[diagram]
    {
      |(D2)| D[m-2] & & |(D1u)| D[m-1] \\
      & |(2)| [m-2] & & |(1u)| [m-1] \\
      |(D1l)| D[m-1] & & |(Dh)| D\horn{m,{\{1,\ldots,m-1\}}} \\
      & |(1l)| [m-1] & & |(0)| [m] \\
    };

    \draw[inj] (D2) to node[above left] {$\face_{m-1}$} (D1u);
    \draw[inj] (D2) to node[above left] {$\face_0$}     (D1l);

    \draw[inj] (D1l) to (Dh);
    \draw[inj] (D1u) to (Dh);

    \draw[inj,cross line] (2) to node[above left] {$\face_{m-1}$} (1u);
    \draw[->,cross line]  (2) to node[above left] {$\face_0$}     (1l);

    \draw[inj] (1l) to (0);
    \draw[->]  (1u) to (0);

    \draw[->] (D2)  to (2);
    \draw[->] (D1u) to (1u);
    \draw[->] (D1l) to (1l);
    \draw[->] (Dh)  to (0);
  \end{ctikzpicture}
  the back square is a pushout of two sieves hence it induces
  a homotopy pullback of the associated categories of Reedy cofibrant diagrams
  by \Cref{sieves-pullback}.
  The front square is a pushout along a sieve,
  but the vertical map is not a sieve.
  Nonetheless, the conclusion of \Cref{sieves-pullback} holds
  because of a particularly simple form of the latching categories
  in totally ordered sets so that a map of diagrams $[m-1] \to \cat{C}$
  is a Reedy cofibration \iff{} it is one when restricted along both $\face_0$
  and $\face_{m-1}$.
  Hence both squares induce homotopy pullbacks on levelwise categories
  of diagrams by \Cref{level-htpy-eq} and then the assumptions
  of the Gluing Lemma are satisfied by \Cref{Dm-m-he} which finishes the proof.
\end{proof}

An \emph{interval} is a subset of $[m]$ of the form
$\set{x \in [m]}{ i \le x \le j}$ for some $i \le j \in [m]$.
In the next lemma we will consider generalized horns $\horn{m,A}$
with $A \subseteq [m]$ \st{} $[m] \setminus A$ is not an interval
(e.g.\ $A = \{ 1, \ldots, m-1 \}$).
Such horns are called \emph{generalized inner horns}.

\begin{lemma}\label{generalized-inner-horns}
  Let $A \subseteq B$ be subsets of $[m]$ whose complements are not intervals.
  Then the inclusion $\horn{m,B} \ito \horn{m,A}$ is a composite
  of pushouts of inner horn inclusions in dimensions at most $m - |A|$.
  Moreover, all these horns are attached along injective maps.
\end{lemma}

\begin{proof}
  This follows by the proof of \cite{jo}*{Proposition 2.12 (iv)}.
  (The proposition itself is less specific, but the inductive step in its proof
  amounts exactly to the statement above.)
\end{proof}

\begin{proposition}\label{nf-inner}
  The functor $\nf$ carries fibrations of cofibration categories
  to inner fibrations.
\end{proposition}

\begin{proof}
  By \Cref{Rlps,dual-pushout-prod} it suffices to check that
  $D\horn{m,i} \ito D[m]$ induces a weak equivalence
  $\cat{C}^{D[m]}_\Reedy \to \cat{C}^{D\horn{m,i}}_\Reedy$
  for every cofibration category $\cat{C}$ and $0 < i < m$.
  By \Cref{skinny-horn} it will be enough to check this for
  $D\horn{m,{\{1,\ldots,m-1\}}} \ito D\horn{m,i}$.

  That follows by an induction \wrt{} $m$ since this inclusion is built out of
  pushouts of horn inclusions in dimensions below $m$
  by \Cref{generalized-inner-horns}. Since these are pushouts
  along injective maps \Cref{sieves-pullback} says that they induce pullbacks
  of cofibration categories of Reedy diagrams.
\end{proof}

Next, we move to $[0] \ito D\nwiso$ which will be dealt with by constructing
an explicit contraction of $D\nwiso = D\wiso$.

\begin{lemma}\label{E1-contractible}
  The functor $f \from [0] \to D\wiso$ given by the sequence $0 \in D\wiso$
  is a homotopy equivalence of \hore{} categories.
\end{lemma}

\begin{proof}
  The proof is similar to that of \Cref{Dm-m-he}.
  This time objects of $D\wiso$ are represented
  as arbitrary finite non-empty binary sequences.
  Let $p \from D\wiso \to [0]$ be the unique functor to $[0]$
  and let $s \from D\wiso \to D\wiso$ append a new $0$ to every sequence.
  (As before, $s(\phi)$ acts on ``old'' elements as $\phi$
  and sends the ``new'' $0$ to the ``new'' $0$.)
  Every morphism of $E(1)$ is an isomorphism so the \hore{} structure on
  $D\wiso$ is the maximal one.
  Hence the functor $s$ is \hore{} and admits natural weak equivalences
  \begin{ctikzpicture}
    \matrix[diagram]
    {
      |(id)| \id & |(s)| s & |(fp)| f p \\
    };

    \draw[->] (id) to node[above] {$\we$} (s);
    \draw[->] (fp) to node[above] {$\we$} (s);
  \end{ctikzpicture}
  where the map on the left inserts $x$ onto the ``old'' occurrences in $s(x)$
  and the right one inserts $f p(x)$ onto the ``new'' $0$.
\end{proof}

Before completing the main result of this subsection we record a corollary
which will considerably simplify constructions of $\nwiso$-homotopies
in the final section.

\begin{corollary}\label{eqs-in-nf}
  For a cofibration category $\cat{C}$ a \hore{} Reedy cofibrant diagram
  $X \from D[1] \to \cat{C}$ is an equivalence when seen as a morphism of
  $\nf \cat{C}$ \iff{} it is \hore{} \wrt{} $D\hat{[1]}$.
\end{corollary}

\begin{proof}
  If $X$ is an equivalence, then it extends to $D\nwiso$.
  Hence it is \hore{} \wrt{} $D\hat{[1]}$.

  Conversely, consider a diagram
  \begin{ctikzpicture}
    \matrix[diagram]
    {
      |(D0)| D[0] & |(D1)| D\hat{[1]} & |(DE)| D\nwiso \\
      |(0)|  [0]  & |(1)|  \hat{[1]}  \\
    };

    \draw[->] (D0) to node[left]  {$\htp$} (0);
    \draw[->] (D1) to node[right] {$\htp$} (1);

    \draw[->] (D0) to (D1);
    \draw[->] (D1) to (DE);

    \draw[->] (0) to node[below] {$\htp$} (1);

    \draw[->] (D0) to[bend left] node[above] {$\htp$} (DE);
  \end{ctikzpicture}
  where the indicated maps are homotopy equivalences, the vertical ones
  by (the proof of) \Cref{Dm-m-he}, the top one by \Cref{E1-contractible}
  and the bottom one by direct inspection.
  Hence so is the map $D\hat{[1]} \to D\nwiso$ which is also a sieve so that
  the induced restriction functor
  $\cat{C}^{D\nwiso}_\Reedy \to \cat{C}^{D\hat{[1]}}_\Reedy$
  is an acyclic fibration and thus every \hore{} Reedy cofibrant diagram
  on $D\hat{[1]}$ extends to one on $D\nwiso$.
\end{proof}

\begin{proposition}\label{nf-iso}
  The functor $\nf$ carries fibrations of cofibration categories
  to isofibrations.
\end{proposition}

\begin{proof}
  By \Cref{Rlps} it suffices to check that $D[0] \ito \wiso$ induces
  a weak equivalence $\cat{C}^{D\wiso}_\Reedy \to \cat{C}^{D[0]}_\Reedy$
  for every cofibration category $\cat{C}$.
  \Cref{E1-contractible} asserts that this is the case for the composite
  \begin{ctikzpicture}
    \matrix[diagram]
    {
      |(0)| [0] & |(D0)| D[0] & |(DE1)| D\wiso \\
    };

    \draw[inj] (0)   to (D0);
    \draw[inj] (D0)  to (DE1);
  \end{ctikzpicture}
  while \Cref{Dm-m-he} says the same for the first functor.
  Thus the conclusion follows by 2 out of 3.
\end{proof}

\begin{proposition}
  The functor $\nf$ carries acyclic fibrations of cofibration categories
  to acyclic Kan fibrations.
\end{proposition}

\begin{proof}
  This follows from \Cref{Rlps,acyclic-fibration-Reedy}
  and the fact that $D\bdsimp{m} \ito D[m]$ is a sieve for all $m$.
\end{proof}

This concludes the verification of all lifting properties necessary
for the exactness of $\nf$.
In the remainder of this subsection we will derive some further lifting properties
which will be useful later.

Occasionally, it will be convenient to consider
\emph{marked simplicial complexes} instead of simplicial sets.
Recall from the classical simplicial homotopy theory that
an \emph{ordered simplicial complex} is a poset $P$ equipped with a family
of finite, non-empty totally ordered subsets of $P$ (called \emph{simplices})
\st{}
\begin{itemize}
\item a non-empty subset of a simplex is a simplex,
\item for each $x \in P$ the singleton $\{ x \}$ is a simplex.
\end{itemize}
Simplicial complexes with an underlying poset $P$ can be identified with
simplicial subsets of $NP$ (containing all vertices of $NP$).
This is the point of view that we will adopt to define a marked version
of this notion.

\begin{definition}
  A \emph{marked simplicial complex} is a simplicial set $K$
  equipped with an embedding $K \ito NP$ where $P$ is a \hore{} poset.
\end{definition}

Marked simplicial complexes can be seen as certain special
\emph{marked simplicial sets} which are sometimes used to provide some extra
flexibility to the theory of quasicategories.

We extend the definition of $DK$ to a marked simplicial complex $K$ as follows.
The underlying category of $DK$ is the same as previously,
but the \hore{} structure is created by the inclusion $DK \ito DP$.
This agrees with the old definition when $P$ has the trivial \hore{} structure.

\label{subdiv}
Moreover, for a marked simplicial complex $K$ we define a \hore{} poset $\Sd K$
as the full subcategory of $DK$ spanned by the non-degenerate simplices of $K$
and with the \hore{} structure inherited from $DP$.
The category $\Sd K$ is known as the \emph{barycentric subdivision} of $K$
hence the notation.
(By analogy we may think of $DK$ as the \emph{fat barycentric subdivision}
of $K$.)
It is indeed a poset since its objects can be identified with
finite non-empty totally ordered subsets of $P$ that correspond
to non-degenerate simplices of $K$ (just as in the classical definition
of an ordered simplicial complex above) and morphisms with inclusions
of such subsets.
With this interpretation an inclusion $A \subseteq B$ is a weak equivalence
\iff{} $\max A \to \max B$ is a weak equivalence of $P$.
(Of course, if $P$ has the trivial \hore{} structure,
then this condition reduces to $\max A = \max B$.)
In the case when $K = NP$ we will usually write $\Sd P$ in place of $\Sd K$.

The next two lemmas will allow us to reduce constructions of diagrams over $DK$
to constructions of diagrams over $\Sd K$.

\begin{lemma}\label{D-Sd-he-msc}
  For any marked simplicial complex $K$ the inclusion $f \from \Sd K \to DK$
  is a homotopy equivalence.
\end{lemma}

\begin{proof}
  The construction is a minor modification of the one used in \Cref{Dm-m-he}.
  Let $P$ denote the underlying \hore{} poset of $K$.
  We define $q_K \from DK \to \Sd K$ by sending each simplex of $K$ seen as
  a map $[k] \to P$ to its image and $s \from DK \to DK$
  by inserting one extra occurrence of each $p \in P$ that is already present
  in a given $x \in DK$.
  Just as in \Cref{Dm-m-he} a new occurrence is inserted at the end of the block
  of the old occurrences which yields analogous weak equivalences
  \begin{ctikzpicture}
    \matrix[diagram]
    {
      |(id)| \id & |(s)| s & |(fq)| f q_K \text{.} \\
    };

    \draw[->] (id) to node[above] {$\we$} (s);
    \draw[->] (fq) to node[above] {$\we$} (s);
  \end{ctikzpicture}
  Moreover, $q_K f = \id_{\Sd K}$ which finishes the proof.  
\end{proof}

\begin{lemma}\label{relative-D-Sd}
  Let $K \ito L$ be an injective map of finite marked simplicial complexes
  (which means that it covers an injective \hore{} map
  of the underlying \hore{} posets).
  Then for every cofibration category $\cat{C}$
  the inclusion $DK \union \Sd L \ito DL$ induces an acyclic fibration
  $\cat{C}^{DL}_\Reedy \to \cat{C}^{DK \union \Sd L}_\Reedy$.
\end{lemma}

\begin{proof}
  We have the following pushout square of sieves
  between \hore{} direct categories on the left
  and hence a pullback square of cofibration categories on the right
  by \Cref{sieves-pullback}.
  \begin{ctikzpicture}
    \matrix[diagram]
    {
      |(SdK)| \Sd K & |(SdL)| \Sd L &
        |(CP)| \cat{C}^{DK \union \Sd L}_\Reedy & |(CDK)| \cat{C}^{DK}_\Reedy \\
      |(DK)| DK & |(P)| DK \union \Sd L &
      |(CSdL)| \cat{C}^{\Sd L}_\Reedy & |(CSdK)| \cat{C}^{\Sd K}_\Reedy \\
    };

    \draw[inj] (SdK) to (SdL);
    \draw[inj] (DK)  to (P);
    \draw[inj] (SdK) to (DK);
    \draw[inj] (SdL) to (P);

    \draw[fib] (CSdL) to (CSdK);
    \draw[fib] (CP)   to (CDK);
    \draw[fib] (CDK)  to (CSdK);
    \draw[fib] (CP)   to (CSdL);
  \end{ctikzpicture}
  The fibration $\cat{C}^{DK}_\Reedy \fto \cat{C}^{\Sd K}_\Reedy$ is acyclic
  by \Cref{D-Sd-he-msc} and therefore so is
  $\cat{C}^{DK \union \Sd L}_\Reedy \fto \cat{C}^{\Sd L}_\Reedy$.
  Moreover, we have a triangle of fibrations
  \begin{ctikzpicture}
    \matrix[diagram]
    {
      |(DL)| \cat{C}^{DL}_\Reedy & & |(SdL)| \cat{C}^{\Sd L}_\Reedy \\
      & |(P)| \cat{C}^{DK \union \Sd L}_\Reedy \\
    };

    \draw[fib] (DL) to (P);
    \draw[fib] (P)  to (SdL);
    \draw[fib] (DL) to (SdL);
  \end{ctikzpicture}
  where $\cat{C}^{DL}_\Reedy \fto \cat{C}^{\Sd L}_\Reedy$ is acyclic again
  by \Cref{D-Sd-he-msc} and thus so is
  $\cat{C}^{DL}_\Reedy \fto \cat{C}^{DK \union \Sd L}_\Reedy$.
\end{proof}

For future reference we will reinterpret lifting properties
for special outer horns in terms of certain \hore{} structures on categories
$D\horn{m,0}$ and $D\horn{m,m}$.

For each $m > 1$ let $\totl{m}$ denote the \hore{} poset with the underlying poset
$[m]$ and $0 \weto 1$ as the only non-identity weak equivalence.
Similarly, let $\tott{m}$ denote the \hore{} poset with the underlying poset
$[m]$ and $m-1 \weto m$ as the only non-identity weak equivalence.
Let $\hornl{m}$ and $\hornt{m}$ denote the outer horns seen as
marked simplicial complexes with the underlying \hore{} posets $\totl{m}$
and $\tott{m}$.

\begin{lemma}\label{special-horns-nf}
  For every cofibration category $\cat{C}$ the inclusion
  $D\hornl{m} \ito D\totl{m}$ induces a weak equivalence
  $\cat{C}^{D\totl{m}}_\Reedy \to \cat{C}^{D\hornl{m}}_\Reedy$.

  The same holds for $D\hornt{m} \ito D\tott{m}$.
\end{lemma}

\begin{proof}
  By \Cref{acyclic-sieve} it will suffice to see that the inclusion
  $D\hornl{m} \ito D\totl{m}$ has the \Rllp{} \wrt{} all fibrations
  of cofibration categories.

  By \Cref{nf-representation} every Reedy lifting problem of
  $D\hornl{m} \ito D\totl{m}$ against a fibration of cofibration categories
  $P \from \cat{C} \to \cat{D}$ is equivalent to a problem of lifting
  $\hornl{m} \ito \totl{m}$ against $\nf P$
  where the latter is an inner isofibration by \Cref{nf-inner,nf-iso}
  and the horn is special by \Cref{eqs-in-nf}.
  Hence it has a solution by \Cref{special-horns}.

  The same argument works for $D\hornt{m} \ito D\tott{m}$
  since \Cref{special-horns} applies to both types of special horns.
\end{proof}

Let $[k+\tilde{1}+m]$ denote a \hore{} category with underlying category
$[k+1+m]$ and $k \weto k+1$ as the only non-identity weak equivalence.
Let $\horn{k+\tilde{1}+m,[k]}$ denote the generalized horn $\horn{k+1+m,[k]}$
seen as a marked simplicial complex with the underlying \hore{} poset
$[k+\tilde{1}+m]$.
The next lemma is a generalization of the previous one.

\begin{lemma}\label{special-generalized-horns}
  The inclusion $D\horn{k+\tilde{1}+m,[k]} \ito D[k+\tilde{1}+m]$
  has the \Rllp{} \wrt{} all fibrations of cofibration categories.
  Hence for any cofibration category $\cat{C}$ it induces a weak equivalence
  $\cat{C}^{D[k+\tilde{1}+m]}_\Reedy
    \to \cat{C}^{D\horn{k+\tilde{1}+m,[k]}}_\Reedy$.
\end{lemma}

\begin{proof}
  The case of $k = 0$ is just the previous lemma (with $m$ replaced by $1+m$).
  The case of $k > 0$ can be reduced to the case of $k = 0$ as follows.
  We have $[k+1+m] \iso [k] \join [m]$
  and $\horn{k+1+m,[k]} \iso \simp{k} \join \bdsimp{m}$ and hence
    it will suffice
  to solve every lifting problem
  \begin{ctikzpicture}
    \matrix[diagram]
    {
      |(b)| \simp{k} \join \bdsimp{m} & |(C)| \qcat{C} \\
      |(s)| \simp{k} \join \simp{m}   & |(D)| \qcat{D} \\
    };

    \draw[->]  (b) to node[above] {$X$} (C);
    \draw[->]  (s) to node[below] {$Y$} (D);
    \draw[fib] (C) to node[right] {$P$} (D);

    \draw[->] (b) to (s);
  \end{ctikzpicture}
  where $X$ and $Y$ send the edge $k \to k+1$ to an equivalence
  and $P$ is an inner isofibration (by \Cref{nf-representation}).
  This problem is equivalent to
  \begin{ctikzpicture}
    \matrix[diagram] {
      |(b)| \{ k \} \join \bdsimp{m} & |(C)| X' \uslice \qcat{C} \\
      |(s)| \{ k \} \join \simp{m}   & |(D)| Y' \uslice \qcat{D} \\
    };

    \draw[->] (b) to (C); \draw[->] (s) to (D);

    \draw[fib] (C) to (D);

    \draw[->] (b) to (s);
  \end{ctikzpicture}
  where $X'$ and $Y'$ are the restrictions of $X$ and $Y$ to $\simp{k-1}$
  so that the resulting horn is special (under identifications
  $\{ k \} \join \simp{m} \iso \simp{1+m}$
  and $\{ k \} \join \bdsimp{m} \iso \horn{1+m,0}$).
  It has a solution by the case of $k = 0$.
\end{proof}
  \subsection{Infinite homotopy colimits}
\label{sec:infinite-hocolims}

The next step is to verify that $\nf \cat{C}$ is finitely cocomplete.
This proof is rather involved, but it turns out that this is largely due to
certain technicalities which disappear if we assume that $\cat{C}$ has some
infinite homotopy colimits.

In this subsection we explain how infinite homotopy colimits can be introduced
to cofibration categories and how the results discussed so far can be extended
to this context.

\subsection*{Infinite homotopy colimits in cofibration categories}

We will consider the following axioms in addition to axioms (C0-5)
of \Cref{ch:cof-cats}.

\begin{itemize}
\item[(C6)] Cofibrations are stable under sequential colimits,
  i.e. given a sequence of cofibrations
  \begin{ctikzpicture}
    \matrix[diagram]
    {
      |(A0)| A_0 & |(A1)| A_1 & |(A2)| A_2 & |(ld)| \ldots \\
    };

    \draw[cof] (A0) to (A1);
    \draw[cof] (A1) to (A2);
    \draw[cof] (A2) to (ld);
  \end{ctikzpicture}
  its colimit $A_\infty$ exists and the induced morphism $A_0 \to A_\infty$
  is a cofibration.
  Acyclic cofibrations are stable under sequential colimits.
\item[(C7-$\kappa$)]
  Coproducts of $\kappa$-small families of objects exist.
  Cofibrations and acyclic cofibrations
  are stable under $\kappa$-small coproducts.
\end{itemize}

Axiom (C7) is parametrized by a regular cardinal number $\kappa$.
(And if we write (C7) we will take it to refer to all small coproducts.)
A set is \emph{$\kappa$-small} if its cardinality
is strictly less than $\kappa$.
In particular, $\aleph_0$-small sets are precisely finite sets
and $\aleph_1$-small sets are precisely countable sets.
We say that a cofibration category is
\begin{itemize}
\item
  \emph{$\kappa$-cocomplete} for $\kappa > \aleph_0$
  if it satisfies (C6) and (C7-$\kappa$),
\item \emph{cocomplete}
  if it satisfies (C6) and (C7).
\end{itemize}
Again, the words ``$\kappa$-cocomplete'' and ``cocomplete'' are really
shorthands for ``homotopy $\kappa$-cocomplete'' and ``homotopy cocomplete''.
We will justify below that $\kappa$-cocomplete cofibration categories indeed
have all $\kappa$-small homotopy colimits.
Axioms (C0-5) imply (C7-$\aleph_0$) and we will sometimes refer to
finitely cocomplete cofibration categories
as \mbox{$\aleph_0$-cocomplete} cofibration categories.
Similarly, the axioms (C0-6) imply (C7-$\aleph_1$)
which is therefore redundant in the definition of
a homotopy $\aleph_1$-cocomplete cofibration category.
This name will be abbreviated to
a \emph{countably cocomplete} cofibration category.

Next, we introduce \emph{$\kappa$-cocontinuous} functors which
(according to the definition and K. Brown's Lemma)
are essentially homotopy invariant functors
that preserve certain basic $\kappa$-small homotopy colimits.
It will be explained later in this subsection
that they actually preserve all $\kappa$-small homotopy colimits.

\begin{definition}
  \enumhack
  \begin{enumerate}[label=\dfn]
  \item
    For $\kappa > \aleph_0$ a functor $F \from \cat{C} \to \cat{D}$
    between $\kappa$-cocomplete cofibration categories
    is \emph{$\kappa$-cocontinuous} if it preserves cofibrations,
    acyclic cofibrations, pushouts along cofibrations,
    colimits of sequences of cofibrations and $\kappa$-small coproducts.
  \item
    A functor $F \from \cat{C} \to \cat{D}$
    between cocomplete cofibration categories is \emph{cocontinuous}
    if it preserves cofibrations, acyclic cofibrations,
    pushouts along cofibrations, colimits of sequences of cofibrations
    and small coproducts.
  \end{enumerate}
\end{definition}

Just as in the case of countably cocomplete cofibration categories,
preservation of countable coproducts follows from preservation
of colimits of sequences of cofibrations and thus it is redundant
in the definition of an $\aleph_1$-cocontinuous functor.
(But then preservation of an initial object has to be assumed explicitly.)

By extension, exact functors in the sense of \Cref{ch:cof-cats}
will be sometimes referred to as as \emph{$\aleph_0$-cocontinuous}.

The notions of ($\kappa$-)cocomplete cofibration categories
and ($\kappa$-)cocontinuous functors dualize to the notions of
\emph{($\kappa$-)complete} fibration categories
and \emph{($\kappa$-)continuous} functors.

Let $\CofCat_\kappa$ denote the category of small $\kappa$-cocomplete
cofibration categories and $\kappa$-cocontinuous functors.

All the results about cofibration categories proven
or cited in \Cref{ch:cof-cats,sec:Rlps} readily generalize
to $\kappa$-cocomplete cofibration categories.
The correct statements can be obtained by replacing phrases
\begin{itemize}
\item ``cofibration category'' with
   ``$\kappa$-cocomplete cofibration category'',
\item ``exact functor'' with ``$\kappa$-cocontinuous functor'',
\item ``finite direct category'' with ``$\kappa$-small direct category''.
\end{itemize}
The proofs will occasionally require extra arguments, but they are all routine
and completely analogous to the ones already given
for finitely cocomplete cofibration categories.
For example, an updated version of \Cref{pullback-of-cofcats} says that
in the category $\CofCat_\kappa$ pullbacks along fibrations exist.
The main modification is that we need to verify that
the resulting pullback $\cat{P}$ satisfies axioms (C6) and (C7-$\kappa$).
The proofs are essentially the same as the proof of (C4)
using an obvious analogues of \Cref{q-pullback-colimits} for limits of towers
and products.

We have restricted attention to $\CofCat_\kappa$ only for convenience.
If we want to consider cocomplete cofibration categories,
we cannot assume that they are small.
However, all the results of \Cref{ch:cof-cats} apply to this case
in the sense that cocomplete cofibration categories
form a fibration category in a higher Grothendieck universe as explained
in the introduction.

The updated \Cref{colims-in-cofcats} says that
$\kappa$-cocomplete cofibration categories
have $\kappa$-small direct homotopy colimits.
This can be used to motivate axioms (C6) and (C7) just like the Gluing Lemma
motivated (C4).
Namely, (C6) is used to show that the colimit functor
$\colim_\nat \from \cat{C}^\nat_\Reedy \to \cat{C}$ is exact.
More precisely, stability of (acyclic) cofibrations under sequential colimits
implies that $\colim_\nat$ preserves (acyclic) cofibrations,
see \cite{rb}*{Lemma 9.3.4(1)}.
Similarly, (C7) implies that $\colim_J \from \cat{C}^J_\Reedy \to \cat{C}$
is exact for $J$ discrete.
The case of all the other direct categories is reduced to these two
and the Gluing Lemma as in the proof of \cite{rb}*{Theorem 9.3.5(1)}.

This handles the case of \emph{direct} homotopy colimits and,
as was pointed out before, for $\kappa = \aleph_0$
restricting to direct categories was essential.
However, for $\kappa > \aleph_0$ $\kappa$-cocomplete cofibration categories
have all $\kappa$-small homotopy colimits, i.e.\ the ones indexed by
arbitrary $\kappa$-small categories.
Their construction is more complicated and uses categories of the form $DJ$
introduced in \Cref{sec:qcats-of-frames}.
In fact, one of the main reasons for introducing this construction is that
the problem of computing homotopy colimits over $J$ can be reduced to
the problem of computing homotopy colimits over $DJ$ which is direct.

The way this works is that a \hore{} diagram $X \from J \to \cat{C}$
contains the same homotopical information as $p_J^* X \from DJ \to \cat{C}$.
In fact, \hore{} diagrams over $DJ$ are these that are
(weakly equivalent to the ones) pulled back along $p_J$
from \hore{} diagrams over $J$.
This is made precise as follows.
The category $\cat{C}^J$ of all \hore{} diagrams $J \to \cat{C}$
has a structure of a cofibration category with levelwise weak equivalences
and cofibrations by \cite{rb}*{Theorem 9.5.5(1)}.\footnote{
Note that this means that $\cat{C}^J$ can be made into a cofibration category
for an arbitrary ($\kappa$-small) $J$ which is not known for model categories.}
Moreover, $p_J^* \from \cat{C}^J \to \cat{C}^{DJ}$ is a weak equivalence
of cofibration categories by \cite{rb}*{Theorem 9.5.8(1)}.

As a result, just as in the case of direct homotopy colimits,
the homotopy colimit functor can be thought of as a zig-zag of exact functors
\begin{ctikzpicture}
  \matrix[diagram,column sep=3.5em]
  {
    |(CJ)| \cat{C}^J & |(CDJ)| \cat{C}^{DJ} &
    |(CDJR)| \cat{C}^{DJ}_{\mathrm{R}} & |(C)| \cat{C} \text{.} \\
  };

  \draw[->] (CJ) to node[above] {$\we$} node[below] {$p_J^*$} (CDJ);

  \draw[inj] (CDJR) to node[above] {$\we$} (CDJ);

  \draw[->] (CDJR) to node[above] {$\colim_{DJ}$} (C);
\end{ctikzpicture}
These results were used by Cisinski to prove that every cofibration category
has an associated derivator \cite{c-cd}*{Corollaire 6.21},
see also \cite{rb}*{Theorem 10.3.2}.

\subsection*{Infinite colimits in quasicategories}

The results of \Cref{ch:qcats} also generalize to
$\kappa$-cocomplete quasicategories, in fact,
in an even more straightforward manner since the notion of a colimit of
a diagram $K \to \qcat{C}$ is completely uniform in $K$ and there is no need
to distinguish between cases depending on the cardinality of $K$.

A quasicategory $\qcat{C}$ to be \emph{$\kappa$-cocomplete}
if it has colimits indexed over all $\kappa$-small simplicial sets.
Similarly, a functor between $\kappa$-cocomplete quasicategories
is \emph{$\kappa$-cocontinuous} if it carries universal cones under
all $\kappa$-small diagrams to universal cones.

All the results of \Cref{ch:qcats} remain correct when we replace phrases
``finitely cocomplete quasicategory'' and ``exact functor'' with
``$\kappa$-cocomplete quasicategory'' and ``$\kappa$-cocontinuous functor''
respectively.
This time proofs require no modifications.

\subsection*{Completeness of fibration categories
  of cofibration categories and quasicategories}

The discussion in the two previous paragraphs implies that $\CofCat_\kappa$
and $\QCat_\kappa$ are fibration categories for all regular cardinals $\kappa$.
In fact, they are both complete, i.e.\ satisfy axioms (C6)\textop{}\ 
and (C7)\textop{}.
We state the upgraded theorems explicitly for future reference.

\begin{theorem}\label{fibcat-of-cofcats-complete}
  The category $\CofCat_\kappa$ of
  small $\kappa$-cocomplete cofibration categories
  and $\kappa$-cocontinuous functors with weak equivalences and fibrations
  defined as in \Cref{ch:cof-cats} is a complete fibration category.
\end{theorem}

\begin{theorem}\label{fibcat-of-quasicats-complete}
  The category $\QCat$ of small quasicategories with simplicial maps as morphisms,
  categorical equivalences as weak equivalences
  and inner isofibrations as fibrations
  is a complete fibration category.
\end{theorem}

\begin{theorem}\label{fibcat-of-kappa-quasicats-complete}
  The category $\QCat_\kappa$ of small $\kappa$-cocomplete quasicategories
  with \mbox{$\kappa$-cocon}-tinuous functors as morphisms,
  categorical equivalences as weak equivalences
  and inner isofibrations as fibrations
  is a complete fibration category.
\end{theorem}

Proofs of these theorems are routine modifications of the proofs
of their counterparts discussed in \Cref{ch:cof-cats,ch:qcats}.

Finally, we state an updated version of \Cref{nerve-exact}.

\begin{theorem}\label{nerve-continuous}
  The functor $\nf \from \CofCat_\kappa \to \QCat_\kappa$
  is a continuous functor of complete fibration categories.
  In particular, it takes values in $\kappa$-cocomplete quasicategories
  and $\kappa$-cocontinuous functors.
\end{theorem}

This theorem clearly generalizes \Cref{nerve-exact}.
In the rest of this section we will proceed with the proof of the general statement.

  \subsection{Cocompleteness: the infinite case}
\label{sec:cocompleteness-infinite}

In order to complete the proof of \Cref{nerve-continuous} it remains to verify
that $\nf$ takes values in $\kappa$-cocomplete quasicategories
and $\kappa$-cocontinuous functors.
From this point on the cases of finitely cocomplete cofibration categories
and $\kappa$-cocomplete cofibration categories for $\kappa > \aleph_0$
will diverge.
The general approaches to both cases are still analogous, but they differ
in technical details and there seems to be no way of presenting them
in a completely uniform manner.
The presence of infinite homotopy colimits allows us
to use simpler constructions so we will consider the case of $\kappa > \aleph_0$
first.
The remaining case of $\kappa = \aleph_0$ will be covered in the next subsection.

First, we need a few preliminary lemmas.
Recall that if $I$ is a discrete category, then colimits over $[0] \join I$
are called wide pushouts.
A wide pushout of a diagram $X \from [0] \join I \to \cat{C}$ will be denoted by
\begin{align*}
  \bigpush{X_0}{i \in I} X_i \text{.}
\end{align*}
The inclusion of the $m$th vertex $\simp{0} \to K \join \simp{m}$ is cofinal
which suggests that colimits over $D(K \join \simp{m})$ should be given by
evaluating diagrams at any simplex containing that vertex.

\begin{lemma}\label{colim-simplex-join}
  Let $\cat{C}$ be a $\kappa$-cocomplete cofibration category and $K$
  a $\kappa$-small simplicial set.
  If $X \from D(K \join \simp{m}) \to \cat{C}$
  is a \hore{} Reedy cofibrant diagram, then the induced morphism
  \begin{align*}
    X_{[m]} \to \colim_{D(K \join \simp{m})} X
  \end{align*}
  is a weak equivalence.
\end{lemma}

\begin{proof}
  The morphism in question factors as
  \begin{align*}
    X_{[m]} \to \colim_{D[m]} X \to \colim_{D(K \join \simp{m})} X
  \end{align*}
  where the first morphism is a weak equivalence
  by \Cref{equiv-cofinal,Dm-m-he}.
  Thus it will be enough to check that the second one is.

  It will suffice to verify that this statement holds when $K$ is a simplex
  and that it is preserved under coproducts, pushouts along monomorphisms
  and colimits of sequences of monomorphisms.

  Let $K = \simp{k}$ and let $\iota$ be the composite
  $[m] \ito [k] \join [m] \iso [k+1+m]$.
  Then we have a commutative square
  \begin{ctikzpicture}
    \matrix[diagram]
    {
      |(m)|  X_\iota           & |(cm)|  \colim_{D[m]} X \\
      |(mk)| X_{\id_{[k+1+m]}} & |(cmk)| \colim_{D[k+1+m]} X \\
    };

    \draw[->] (m)  to (cm);
    \draw[->] (mk) to (cmk);

    \draw[->] (m)  to (mk);
    \draw[->] (cm) to (cmk);
  \end{ctikzpicture}
  where the left morphism is a weak equivalence since $X$ is \hore{}
  and so are the horizontal ones by the argument above.
  Thus the right morphism is also a weak equivalence.

  Next, consider a pushout square
  \begin{ctikzpicture}
    \matrix [diagram]
    {
      |(A)| A & |(K)| K \\
      |(B)| B & |(L)| L \\
    };

    \draw[->] (A) to (K);
    \draw[->] (B) to (L);

    \draw[inj] (A) to (B);
    \draw[inj] (K) to (L);
  \end{ctikzpicture}
  \st{} the statement holds for $A$, $B$ and $K$.
  The functor $\uvar \join \simp{m}$ preserves pushouts by \Cref{join-colimits}
  and so does $D$ by \Cref{D-preserves-colimits}.
  Thus in the cube
  \begin{ctikzpicture}
    \matrix[diagram,row sep=2em,column sep=1.2em]
    {
      |(A)| D A & & |(K)| D K \\
      & |(Am)| D(A \join \simp{m}) & & |(Km)| D(K \join \simp{m}) \\
      |(B)| D B & & |(L)| D L \\
      & |(Bm)| D(B \join \simp{m}) & & |(Lm)| D(L \join \simp{m}) \\
    };

    \draw[->] (A) to (K);
    \draw[->] (B) to (L);

    \draw[inj] (A) to (B);
    \draw[inj] (K) to (L);

    \draw[->,cross line] (Am) to (Km);
    \draw[->] (Bm) to (Lm);

    \draw[inj,cross line] (Am) to (Bm);
    \draw[inj] (Km) to (Lm);

    \draw[->] (A) to (Am);
    \draw[->] (K) to (Km);
    \draw[->] (B) to (Bm);
    \draw[->] (L) to (Lm);
  \end{ctikzpicture}
  both the front and the back faces are pushouts along sieves
  and the conclusion follows by \cite{rb}*{Theorem 9.4.1 (1a)}
  and the Gluing Lemma.

  The case of colimits of sequences of monomorphisms is similar and we omit it.

  The case of coproducts is also similar, but there is a difference in the fact
  that $\uvar \join \simp{m}$ doesn't preserve coproducts.
  Instead, it sends coproducts to wide pushouts under $\simp{m}$.
  Thus if we have a $\kappa$-small family $\{ K_i \mid i \in I \}$
  of $\kappa$-small simplicial sets and a diagram
  $X \from D((\bigcoprod_i K_i) \join \simp{m}) \to \cat{C}$,
  then there is a canonical isomorphism
  \begin{align*}
    \bigpush{\colim_{D[m]} X}{i \in I} (\colim_{D(K_i \join \simp{m})} X)
      \iso \colim_{D((\bigcoprod_{i \in I} K_i) \join \simp{m})} X \text{.}
  \end{align*}
  The conclusion follows by the fact that in a cofibration category
  all the structure morphisms of a wide pushout of acyclic cofibrations
  are again acyclic cofibrations.
  (By \Cref{equiv-cofinal} since $\hat{[0] \join I}$ is contractible
  to its cone object as a \hore{} category.)
\end{proof}

Note that for any simplicial set $K$ there is a unique functor
$p_K \from D(K^\ucone) \to (DK)^\ucone$ that restricts
to the identity of $DK$ and sends all the objects not in $DK$ to the cone point
of $(DK)^\ucone$.
This functor is \hore{}.
In the next lemma we use it to compare colimits over $DK$ and $D(K^\ucone)$.

\begin{lemma}\label{cone-factorization}
  Let $\cat{C}$ be a $\kappa$-cocomplete cofibration category,
  $K$ a $\kappa$-small simplicial set and $X \from DK \to \cat{C}$
  a \hore{} Reedy cofibrant diagram.
  Consider a morphism $f \from \colim_{DK} X \to Y$
  and the corresponding cone $\tilde T \from (DK)^\ucone \to \cat{C}$.
  If $T$ is any Reedy cofibrant replacement of $p_K^* \tilde T$ relative to $DK$
  (which exists by \Cref{relative-factorization}), then $f$ factors as
  \begin{align*}
    \colim_{DK} X \to \colim_{D(K^\ucone)} T \weto Y \text{.}
  \end{align*}
\end{lemma}

\begin{proof}
  To verify that the above composite agrees with $f$ it suffices to check that
  it agrees upon precomposition with $X_x \to \colim_{DK} X$ for all $x \in DK$.
  That's indeed the case since $T|DK = X$.

  It remains to check that the latter morphism
  is a weak equivalence.
  In the diagram
  \begin{ctikzpicture}
    \matrix[diagram]
    {
      |(KT)| \colim_{D(K^\ucone)} T & |(Y)| Y \\
      |(T0)| T_0 \\
    };

    \draw[->] (KT) to (Y);
    \draw[->] (T0) to (KT);

    \draw[->] (T0) to (Y);
  \end{ctikzpicture}
  the left morphism is a weak equivalence by \Cref{colim-simplex-join}
  and so is the diagonal one since $T$
  is a cofibrant replacement of $p_K^* \tilde T$.
  Therefore the top morphism is also a weak equivalence.
\end{proof}

We will need an augmented version of the $D$ construction.
In fact, we will only need to apply it to $[m]$ and $\bdsimp{m}$ so we define it
only in these cases.

We will denote by $D_\aug[m]$ the category of all order preserving maps
$[k] \to [m]$ including the one with $[k] = [-1] = \emptyset$.
A morphism from $x \from [k] \to [m]$ to $y \from [l] \to [m]$
is an injective order preserving map $\phi \from [k] \ito [l]$
\st{} $x = y \phi$.
In other words, $D_\aug[m]$ is obtained from $D[m]$ by adjoining an initial object.
The \hore{} structure on $D_\aug[m]$ is an extension of the one on $D[m]$
where $[-1] \to [m]$ is not weakly equivalent to any other object.
We will also consider a slightly richer \hore{} structure $\tilde D_\aug[m]$
where $[-1] \to [m]$ is weakly equivalent to all the constant maps
with the value $0$.

The \hore{} categories $D_\aug\bdsimp{m}$ and $\tilde D_\aug\bdsimp{m}$ are the full
\hore{} subcategories of $D_\aug[m]$ and $\tilde D_\aug[m]$ spanned
by the non-surjective maps $[k] \to [m]$ (i.e. by the simplices of $\bdsimp{m}$
including the ``$(-1)$-dimensional'' one).

Similarly, the \hore{} posets $\Sd_\aug[m]$, $\tilde\Sd_\aug[m]$, $\Sd_\aug\bdsimp{m}$
and $\tilde\Sd_\aug\bdsimp{m}$ are the full \hore{} subcategories of $D_\aug[m]$,
$\tilde D_\aug[m]$, $D_\aug\bdsimp{m}$ and $\tilde D_\aug\bdsimp{m}$ respectively spanned
by their objects that are injective as maps $[k] \to [m]$.

\begin{lemma}\label{D-Sd-he}
  The restriction functors
  \begin{align*}
    \cat{C}^{D_\aug[m]}_\Reedy & \to \cat{C}^{\Sd_\aug[m]}_\Reedy &
    \cat{C}^{D_\aug\bdsimp{m}}_\Reedy & \to \cat{C}^{\Sd_\aug\bdsimp{m}}_\Reedy \\
    \cat{C}^{\tilde D_\aug[m]}_\Reedy & \to \cat{C}^{\tilde \Sd_\aug[m]}_\Reedy &
    \cat{C}^{\tilde D_\aug\bdsimp{m}}_\Reedy &
      \to \cat{C}^{\tilde \Sd_\aug\bdsimp{m}}_\Reedy
  \end{align*}
  are all acyclic fibrations.
\end{lemma}

\begin{proof}
  All these functors are induced by sieves so they are fibrations.
  We will construct a homotopy inverse to
  $f \from \tilde\Sd_\aug[m] \ito \tilde D_\aug[m]$ which will restrict to
  homotopy inverses of all the other sieves in question.
  The construction is a minor modification of the one used in \Cref{Dm-m-he}
  (and essentially the same as in \Cref{D-Sd-he-msc}).
  Namely, we define $q \from \tilde D_\aug[m] \to \tilde \Sd_\aug[m]$ by sending each
  $[k] \to [m]$ to its image and $s \from \tilde D_\aug[m] \to \tilde D_\aug[m]$
  by inserting one extra occurrence of each $i \in [m]$ that is already present
  in a given $x \in \tilde D_\aug[m]$.
  Just as in \Cref{Dm-m-he} a new occurrence is inserted at the end of the block
  of the old occurrences which yields analogous weak equivalences
  \begin{ctikzpicture}
    \matrix[diagram]
    {
      |(id)| \id & |(s)| s & |(fq)| f q \text{.} \\
    };

    \draw[->] (id) to node[above] {$\we$} (s);
    \draw[->] (fq) to node[above] {$\we$} (s);
  \end{ctikzpicture}
  Moreover, $q f = \id_{\tilde \Sd_\aug[m]}$ which finishes the proof.
\end{proof}

\HoRe{} Reedy cofibrant diagrams on $D_\aug[1]$ will be used to encode cones
on diagrams in $\nf \cat{C}$ and the ones which are \hore{} \wrt{}
$\tilde D_\aug[1]$ will correspond to the universal cones.
The following lemma (and, more directly, \Cref{universal-cones} below)
will translate between the universality of such cones in $\nf \cat{C}$
and strict colimits of the corresponding diagrams in $\cat{C}$.

\begin{lemma}\label{augmented-bd}
  The two functors
  \begin{enumerate}[label=\thm]
  \item $\cat{C}^{\tilde\Sd_\aug[m]}_\Reedy
    \to \cat{C}^{\tilde\Sd_\aug\bdsimp{m}}_\Reedy$ and
  \item $\cat{C}^{\tilde D_\aug[m]}_\Reedy
    \to \cat{C}^{\tilde D_\aug\bdsimp{m}}_\Reedy$
  \end{enumerate}
  induced by the inclusion $\bdsimp{m} \ito \simp{m}$ are acyclic fibrations.
\end{lemma}

\begin{proof}
  Both inclusions $\tilde\Sd_\aug\bdsimp{m} \ito \tilde\Sd_\aug[m]$
  and $\tilde D_\aug\bdsimp{m} \ito \tilde D_\aug[m]$ are sieves
  hence it will be enough to prove that they are homotopy equivalences.
  \begin{enumerate}[label=\thm]
  \item Consider two \hore{} functors
    $i_0, i_1 \from \Sd_\aug[m-1] \to \tilde\Sd_\aug[m]$ defined as $i_0 A = A + 1$
    and $i_1 A = i_0 A \union \{ 0 \}$ for any $A \subseteq [m-1]$.
    We have $i_0 A \subseteq i_1 A$ and the resulting natural transformation
    induces an isomorphism of \hore{} categories
    $\Sd_\aug[m-1] \times \hat{[1]} \to \tilde\Sd_\aug[m]$.
    It follows that $i_0$ is a homotopy equivalence
    since $[0] \ito \hat{[1]}$ is.
    This homotopy equivalence also restricts to a homotopy equivalence
    $\Sd_\aug[m-1] \ito \tilde\Sd_\aug\bdsimp{m}$ and thus the conclusion follows
    by the triangle
    \begin{ctikzpicture}
      \matrix[diagram]
      {
        & |(m-1)| \Sd_\aug[m-1] \\
        |(bd)| \tilde\Sd_\aug\bdsimp{m} & & |(m)| \tilde\Sd_\aug[m] \text{.} \\
      };

      \draw[inj] (m-1) to node[above right] {$i_0$} (m);

      \draw[inj] (m-1) to (bd);
      \draw[inj] (bd)  to (m);
    \end{ctikzpicture}
  \item We have a square
    \begin{ctikzpicture}
      \matrix[diagram]
      {
        |(sb)| \tilde\Sd_\aug\bdsimp{m} & |(ss)| \tilde\Sd_\aug[m] \\
        |(db)| \tilde D_\aug\bdsimp{m}  & |(ds)| \tilde D_\aug[m] \\
      };

      \draw[inj] (sb) to (ss);
      \draw[inj] (db) to (ds);

      \draw[inj] (sb) to (db);
      \draw[inj] (ss) to (ds);
    \end{ctikzpicture}
    where the top functor is a homotopy equivalence by the first part
    of the lemma and so are the horizontal ones by \Cref{D-Sd-he}.
    Therefore so is the bottom one. \qedhere
  \end{enumerate}
\end{proof}

For every $m > 0$ each object of $D(K \join \simp{m})$ can be uniquely written
as $x \join \phi$ with $x \in D_\aug K$ and $\phi \in D_\aug[m]$.
This yields a functor $r_K \from D(K \join \simp{m}) \to D_\aug[m]$
sending $x \join \phi$ to $\phi$ to which we associate the left Kan extension
\begin{align*}
  \Lan_{r_K} \from \cat{C}^{D(K \join \simp{m})}_\Reedy
    \to \cat{C}^{D_\aug[m]}_\Reedy
\end{align*}
which can be constructed as 
\begin{align*}
  (\Lan_{r_K} X)_\phi = \colim_{D[k]} \phi^* X
\end{align*}
where $\phi \from [k] \to [m]$.
Analogously, we have a functor
$s_K \from D(K \join \bdsimp{m}) \to D_\aug\bdsimp{m}$
and the associated left Kan extension
\begin{align*}
  \Lan_{s_K} \from \cat{C}^{D(K \join \bdsimp{m})}_\Reedy
    \to \cat{C}^{D_\aug\bdsimp{m}}_\Reedy \text{.}
\end{align*}

We form pullbacks (the front and back squares of the cube)
\begin{ctikzpicture}
  \matrix[diagram]
  {
    |(ju)| \cat{C}^{\tilde D(K \join \simp{m})}_\Reedy & &
    |(au)| \cat{C}^{\tilde D_\aug[m]}_\Reedy \\
    & |(bju)| \cat{C}^{\tilde D(K \join \bdsimp{m})}_\Reedy & &
    |(bau)| \cat{C}^{\tilde D_\aug\bdsimp{m}}_\Reedy \\
    |(j)| \cat{C}^{D(K \join \simp{m})}_\Reedy & &
    |(a)| \cat{C}^{D_\aug[m]}_\Reedy \\
    & |(bj)| \cat{C}^{D(K \join \bdsimp{m})}_\Reedy & &
    |(ba)| \cat{C}^{D_\aug\bdsimp{m}}_\Reedy \text{.} \\
  };

  \draw[->] (ju) to (au);
  \draw[->] (j)  to node[below right] {$\Lan_{r_K}$} (a);

  \draw[fib] (ju) to (j);
  \draw[fib] (au) to (a);

  \draw[->,cross line]   (bju) to (bau);
  \draw[fib, cross line] (bju) to (bj);

  \draw[->]  (bj)  to node[below right] {$\Lan_{s_K}$} (ba);
  \draw[fib] (bau) to (ba);

  \draw[->,shorten >=6pt] (ju) to node[above right] {$P_K$} (bju);

  \draw[->,shorten >=3pt] (au) to (bau);
  \draw[->,shorten >=6pt] (j)  to (bj);
  \draw[->,shorten >=4pt] (a)  to (ba);
\end{ctikzpicture}
Observe that $\cat{C}^{\tilde D(K \join \simp{m})}$
and $\cat{C}^{\tilde D(K \join \bdsimp{m})}$ are just atomic notations
for the pullbacks above, i.e.\ $\tilde D(K \join \simp{m})$
and $\tilde D(K \join \bdsimp{m})$ are \emph{not} \hore{} categories
for general $K$, although they will be interpreted as such
when $K$ is a simplex.

\begin{lemma}\label{universal-cones}
  The induced functor
  $P_K \from \cat{C}^{\tilde D(K \join \simp{m})}_\Reedy \to
    \cat{C}^{\tilde D(K \join \bdsimp{m})}_\Reedy$
  is an acyclic fibration for every $\kappa$-small simplicial set $K$.
\end{lemma}

\begin{proof}
  First, we verify that $P_K$ is a fibration.
  The categories $\cat{C}^{\tilde D(K \join \simp{m})}_\Reedy$
  and $\cat{C}^{\tilde D(K \join \bdsimp{m})}_\Reedy$ are full subcategories of
  $\cat{C}^{D(K \join \simp{m})}_\Reedy$
  and $\cat{C}^{D(K \join \bdsimp{m})}_\Reedy$ respectively.
  They are both closed under taking weakly equivalent objects.
  Hence the lifting properties of the fibration
  $\cat{C}^{D(K \join \simp{m})}_\Reedy
    \fto \cat{C}^{D(K \join \bdsimp{m})}_\Reedy$ are inherited by $P_K$

  For the rest of the argument it will suffice to check that $P_K$
  is a weak equivalence when $K$ is empty or a simplex and that this property
  is preserved under coproducts, pushouts along monomorphisms
  and colimits of sequences of monomorphisms.

  When $K$ is empty then the top square of the cube above happens to be
  a pullback and hence $P_\emptyset$ is an acyclic fibration
  by \Cref{augmented-bd}.

  For $K = \simp{k}$ we will check that $P_{\simp{k}}$ coincides with
  \begin{align*}
    \cat{C}^{D[k+\tilde{1}+m]}_\Reedy
      \to \cat{C}^{D\horn{k+\tilde{1}+m,[k]}}_\Reedy
  \end{align*}
  and the conclusion will follow from \Cref{special-generalized-horns}.
  It is enough to verify that a \hore{} Reedy cofibrant diagram
  $X \from D[k+1+m] \to \cat{C}$ is \hore{} \wrt{} $D[k+\tilde{1}+m]$ \iff{}
  the induced morphism
  \begin{align*}
    \colim X|D[k] \to \colim X|D[k+1]
  \end{align*}
  is a weak equivalence.
  This follows from \Cref{colim-simplex-join}.
  The same argument works with $\horn{k+1+m,[k]}$ in place of $[k+1+m]$,
  since $\simp{k+1}$ is contained in $\horn{k+1+m,[k]}$ for $m > 0$.

  If
  \begin{ctikzpicture}
    \matrix [diagram]
    {
      |(A)| A & |(K)| K \\
      |(B)| B & |(L)| L \\
    };

    \draw[->] (A) to (K);
    \draw[->] (B) to (L);

    \draw[inj] (A) to (B);
    \draw[inj] (K) to (L);
  \end{ctikzpicture}
  is a pushout square of simplicial sets \st{} the conclusion holds
  for $A$, $B$ and $K$, then there is a pullback square
  of cofibration categories
  \begin{ctikzpicture}
    \matrix [diagram]
    {
      |(L)| \cat{C}^{\tilde D(L \join \simp{m})}_\Reedy &
      |(B)| \cat{C}^{\tilde D(B \join \simp{m})}_\Reedy \\
      |(K)| \cat{C}^{\tilde D(K \join \simp{m})}_\Reedy &
      |(A)| \cat{C}^{\tilde D(A \join \simp{m})}_\Reedy \\
    };

    \draw[->] (K) to (A);
    \draw[->] (L) to (B);

    \draw[fib] (B) to (A);
    \draw[fib] (L) to (K);
  \end{ctikzpicture}
  and a similar one with $\bdsimp{m}$ in place of $\simp{m}$.
  Hence the conclusion for $L$ follows from the Gluing Lemma.

  If $K$ is a colimit of a sequence of monomorphisms
  $K_0 \ito K_1 \ito K_2 \ito \ldots$,
  then $\cat{C}^{\tilde D(K \join \simp{m})}_\Reedy$
  is the limit of the tower of fibrations
  \begin{align*}
    \ldots \fto \cat{C}^{\tilde D(K_2 \join \simp{m})}_\Reedy
      \fto \cat{C}^{\tilde D(K_1 \join \simp{m})}_\Reedy
      \fto \cat{C}^{\tilde D(K_0 \join \simp{m})}_\Reedy
  \end{align*}
  and analogously for $\cat{C}^{\tilde D(K \join \bdsimp{m})}_\Reedy$.
  Therefore,
  if $P_{K_i}$ is a weak equivalence for all $i$, then so is $P_K$.

  The case of coproducts is handled similarly except that $\uvar \join \simp{m}$
  doesn't preserve coproducts but carries them to wide pushouts.
  Hence $\cat{C}^{\tilde D((\coprod_i K_i) \join \simp{m})}_\Reedy$
  is the wide pullback
  \begin{align*}
    \bigpull{\cat{C}^{D[m]}_\Reedy}{i}
      \cat{C}^{\tilde D(K_i \join \simp{m})}_\Reedy \text{.}
  \end{align*}
  The conclusion follows since the wide pullback functor is an exact functor
  of fibration categories.
\end{proof}

We are ready to characterize colimits in $\nf \cat{C}$ in terms of
homotopy colimits in $\cat{C}$.

\begin{proposition}\label{nerve-cones}
  Let $\cat{C}$ be a $\kappa$-cocomplete cofibration category,
  $K$ a $\kappa$-small simplicial set and $S \from K^\ucone \to \nf \cat{C}$.
  Then $S$ is universal as a cone under $S|K$ \iff{} the induced morphism
  \begin{align*}
    \colim_{DK} S \to \colim_{D(K^\ucone)} S
  \end{align*}
  is a weak equivalence (with $S$ seen as a \hore{} Reedy cofibrant diagram
  $D(K^\ucone) \to \cat{C}$ by \Cref{nf-representation}).
  Such a cone exists under every diagram $K \to \nf \cat{C}$.
\end{proposition}

\begin{proof}
  If the morphism above is a weak equivalence
  let $U \from K \join \bdsimp{m} \to \nf \cat{C}$ extend $S$.
  The functor $\cat{C}^{\tilde D(K \join \simp{m})}_\Reedy
  \to \cat{C}^{\tilde D(K \join \bdsimp{m})}_\Reedy$
  is an acyclic fibration by \Cref{universal-cones}
  and thus the corresponding \hore{} Reedy cofibrant diagram
  $\tilde D(K \join \bdsimp{m}) \to \cat{C}$ prolongs to
  $\tilde D(K \join \simp{m}) \to \cat{C}$.
  Hence $S$ is universal.

  Conversely, let $S$ be universal.
  Define $T \from D(K^\ucone) \to \cat{C}$ as in \Cref{cone-factorization}
  where we take $f$ to be the identity of $\colim_{D K} S$.
  Then the induced morphism $\colim_{DK} T \to \colim_{D(K^\ucone)} T$
  is a weak equivalence and so $T$ is universal by the argument above
  (which proves the existence statement).
  Therefore by \Cref{colim-equivalent} there exists
  a \hore{} Reedy cofibrant diagram $W \from D(K \join \nwiso) \to \cat{C}$
  which restricts to $[S, T]$ on $D(K \join \bdsimp{1})$.
  In the diagram
  \begin{ctikzpicture}
    \matrix[diagram]
    {
      |(KX)| \colim_{DK} S & |(K0X)| \colim_{D(K^\ucone)} S & |(X0)| S_0 \\
      |(K0Z)| \colim_{D(K^\ucone)} T & |(K1W)| \colim_{D(K \join \simp{1})} W
      & |(W01)| W_{01} \\
    };

    \draw[->] (KX) to node[left] {$\we$} (K0Z);
    \draw[->] (K0X) to (K1W);
    \draw[->] (X0)  to (W01);

    \draw[->] (KX)  to (K0X);
    \draw[->] (K0Z) to (K1W);

    \draw[->] (X0)  to (K0X);
    \draw[->] (W01) to (K1W);
  \end{ctikzpicture}
  both bottom horizontal morphisms and the top right one are weak equivalences
  by \Cref{colim-simplex-join} and so is the right vertical one since
  the \hore{} structure of $D\nwiso$ is the maximal one.
  It follows that $\colim_{DK} S \to \colim_{D(K^\ucone)} S$
  is also a weak equivalence.
\end{proof}

Before completing the proof of \Cref{nerve-continuous} we will point out that
in certain special cases the above criterion for recognizing universal cones
can be simplified considerably.

\begin{example}\label{ex:init}
  A \hore{} Reedy cofibrant diagram $X \from D[0] \to \cat{C}$ is initial
  as an object of $\nf \cat{C}$ \iff{} the canonical morphism $0 \to X_0$
  is a weak equivalence (where $0$ is an initial object of $\cat{C}$).
  This is because the induced morphism $X_0 \to \colim X$ is a weak equivalence
  by \Cref{equiv-cofinal,Dm-m-he}.
\end{example}

\begin{example}\label{ex:push}
  For a \hore{} Reedy cofibrant diagram $X \from D([1] \times [1]) \to \cat{C}$
  consider its restriction to $\Sd([1] \times [1])$.
  \begin{ctikzpicture}
    \matrix[diagram,column sep=10em,row sep=10em]
    {
      |(a-a)| X_{0,0} & |(b-a)| X_{1,0} \\
      |(a-b)| X_{0,1} & |(b-b)| X_{1,1} \\
    };

    \node (aa-ab) at (barycentric cs:a-a=1,a-b=1,b-b=0) {$X_{00,01}$};
    \node (ab-bb) at (barycentric cs:a-a=0,a-b=1,b-b=1) {$X_{01,11}$};
    \node (ab-ab) at (barycentric cs:a-a=1,a-b=0,b-b=1) {$X_{01,01}$};

    \node (ab-aa) at (barycentric cs:a-a=1,b-a=1,b-b=0) {$X_{01,00}$};
    \node (bb-ab) at (barycentric cs:a-a=0,b-a=1,b-b=1) {$X_{11,01}$};

    \node (aab-abb) at (barycentric cs:a-a=2,a-b=3,b-b=2) {$X_{001,011}$};
    \node (abb-aab) at (barycentric cs:a-a=2,b-a=3,b-b=2) {$X_{011,001}$};

    \draw[->] (a-a) to (aa-ab);
    \draw[->] (a-b) to node[left] {$\we$} (aa-ab);

    \draw[->] (a-b) to (ab-bb);
    \draw[->] (b-b) to node[below] {$\we$} (ab-bb);

    \draw[->] (a-a) to (ab-aa);
    \draw[->] (b-a) to node[above] {$\we$} (ab-aa);

    \draw[->] (b-a) to (bb-ab);
    \draw[->] (b-b) to node[right] {$\we$} (bb-ab);

    \draw[->,dashed] (a-a) to (ab-ab);
    \draw[->] (b-b) to node[above right] {$\we$} (ab-ab);

    \draw[->,dashed] (aa-ab) to (aab-abb);
    \draw[->] (ab-bb) to node[above right] {$\we$} (aab-abb);
    \draw[->] (ab-ab) to node[above left]  {$\we$} (aab-abb);

    \draw[->,dashed] (ab-aa) to (abb-aab);
    \draw[->] (bb-ab) to node[below left] {$\we$} (abb-aab);
    \draw[->] (ab-ab) to node[above left] {$\we$} (abb-aab);
  \end{ctikzpicture}
  The corresponding square $\simp{1} \times \simp{1} \to \nf \cat{C}$
  is a pushout (observe that $\spn^\ucone \iso [1] \times [1]$) \iff{}
  the morphism
  \begin{align*}
    X_{00,01} \push_{X_{0,0}} X_{01,00}
      \to X_{001,011} \push_{X_{01,01}} X_{011,001}
  \end{align*}
  induced by the three dashed arrows above is a weak equivalence.
  This can be justified by observing that in the square
  \begin{ctikzpicture}
    \matrix[diagram]
    {
      |(s)| X_{00,01} \push_{X_{0,0}} X_{01,00} &
        |(S)| \colim_{D\spn} X \\
      |(p)| X_{001,011} \push_{X_{01,01}} X_{011,001} &
        |(P)| \colim X \\
    };

    \draw[->] (s) to node[above] {$\we$} (S);
    \draw[->] (p) to node[below] {$\we$} (P);

    \draw[->] (s) to (p);
    \draw[->] (S) to (P);
  \end{ctikzpicture}
  both horizontal morphism are weak equivalences.
  Indeed, they are induced by the composite functors
  \begin{align*}
    \{ (00,01), (0,0), (01,00) \} & \ito \Sd\spn \ito D\spn \\
    \{ (001,011), (01,01), (011,001) \} & \ito \Sd([1] \times [1])
      \ito D([1] \times [1]) \\
  \end{align*}
  where in both cases the latter functor is a homotopy equivalence
  by \Cref{D-Sd-he-msc} while the former functor is a homotopy equivalence
  in the first case and cofinal in the second one.
  The conclusion follows by \Cref{equiv-cofinal}.
\end{example}

\begin{proof}[Proof of \Cref{nerve-continuous}]
  Since we have already verified \Cref{nerve-limits,nf-inner,nf-iso}
  (\Cref{nerve-limits} was generalized to infinite limits
  in the end of \Cref{sec:infinite-hocolims})
  it remains to check that $\nf$ takes values
  in $\kappa$-cocomplete quasicategories and $\kappa$-cocontinuous functors.

  It takes values in quasicategories by \Cref{nf-inner}
  that are $\kappa$-cocomplete by \Cref{nerve-cones}.

  Similarly, colimits in quasicategories of frames were characterized
  in \Cref{nerve-cones} by certain morphisms being weak equivalences
  and weak equivalences are preserved by exact functors by \Cref{Ken-Brown}.
\end{proof}

In the next subsection we will adapt the arguments above to
the case of $\kappa = \aleph_0$.
The proof of the main theorem continues in \Cref{ch:cofcats-of-dgrms}.

  \subsection{Cocompleteness: the finite case}
\label{sec:cocompleteness-finite}

In this subsection we will prove that $\nf \cat{C}$ is finitely cocomplete
for any cofibration category.
The arguments of the previous subsection do not directly apply to this case since
they heavily use the existence of colimits of Reedy cofibrant diagrams
over categories of the form $DK$.
Unfortunately, $DK$ is infinite even when $K$
is a finite (non-empty) simplicial set.
In order to address this problem, we will filter the category $DK$
by finite subcategories
\begin{align*}
  \filt{0,K} \ito \filt{1,K} \ito \filt{2,K} \ito \ldots
\end{align*}
and instead of using a colimit of a Reedy cofibrant diagram
$X \from DK \to \cat{C}$ we will consider the resulting sequence
of finite colimits
\begin{align*}
  \colim_{\filt{0,K}} X \cto \colim_{\filt{1,K}} X \cto \colim_{\filt{2,K}} X
    \cto \ldots 
\end{align*}
If $X$ is \hore{} this sequence stabilizes in the sense that
from some point on (depending on $K$) all morphisms are weak equivalences
and this stable value is a homotopy colimit of $X$.
However, there is no universal bound on when such a sequence stabilizes
when $K$ varies and hence we are forced to think of that entire sequence
as a homotopy colimit of $X$.
It turns out that the proofs of the previous subsection will work
if we carefully substitute such sequences
for actual colimits over categories $DK$.
The difficult part is constructing such filtrations
with all the desired naturality and homotopy invariance
which is the main purpose of this subsection.

Let $J$ be a \hore{} category and $A$ a set of objects of $DJ$,
we denote the sieve generated by $A$ in $DJ$ by $\pfilt{A,J}$.
Moreover, when $J = [m]$ (possibly with some non-trivial \hore{} structure)
we will write objects of $D[m]$ as non-decreasing sequences of elements of $[m]$
often using abbreviations like $0^k 1^l$ to denote the sequence of $k$ $0$s
followed by $l$ $1$s.

The category $D[0]$ can be seen as the category of non-degenerate simplices of
a simplicial set $S$ with exactly one non-degenerate simplex in each dimension.
As it turns out, the skeleton $\Sk^k S$ is contractible for $k$ even
but weakly equivalent to the sphere $\simp{k} / \bdsimp{k}$ for $k$ odd.
This suggests that the filtration of $D[0]$ by sieves generated by
even-dimensional simplices of $S$ should be well-behaved homotopically.
We verify that this is the case in the next two lemmas and later generalize it
to $DK$ for arbitrary finite simplicial sets $K$.

\begin{lemma}\label{cone-filt-contractible}
  For each $k$ the functor $t \from \pfilt{0^k 1,\hat{[1]}} \to [0]$
  is a homotopy equivalence of \hore{} categories.
\end{lemma}

\begin{proof}
  Represent objects of $\pfilt{0^k 1,\hat{[1]}}$ as binary sequences
  and let $j \from [0] \to \pfilt{0^k 1,\hat{[1]}}$ classify the object $1$.
  Next, define $s \from \pfilt{0^k 1,\hat{[1]}} \to \pfilt{0^k 1,\hat{[1]}}$
  by appending a trailing $1$ to each sequence that doesn't have one.
  Then there are natural weak equivalences
  \begin{ctikzpicture}
    \matrix[diagram]
    {
      |(id)| \id_{\pfilt{0^k 1,\hat{[1]}}} & |(s)| s & |(jt)| j t \text{.} \\
    };

    \draw[->] (id) to node[above] {$\we$} (s);
    \draw[->] (jt) to node[above] {$\we$} (s);
  \end{ctikzpicture}
  Moreover, we have $t j = \id_{[0]}$ which finishes the proof.
\end{proof}

The images of the composite functors
\begin{align*}
  \Sd[k] \ito D[k] \to D[0] \text{ and }
  \Sd \bdsimp{k+1} \ito D \bdsimp{k+1} \to D[0]
\end{align*}
are both $\pfilt{0^{k+1},[0]}$.
In the next lemma we consider the resulting functors
\begin{align*}
  t \from \Sd[k] \to \pfilt{0^{k+1},[0]} \text{ and }
  t \from \Sd \bdsimp{k+1} \to \pfilt{0^{k+1},[0]} \text{.}
\end{align*}

\begin{lemma}\label{filt-even-odd}
  Let $k \ge 0$ and let $\cat{C}$ be a cofibration category.
  If $X \from \pfilt{0^{k+1},[0]} \to \cat{C}$
  is a \hore{} Reedy cofibrant diagram, then
  \begin{enumerate}[label=\thm]
  \item the induced morphism
    \begin{align*}
      \colim_{\Sd \simp{k}} t^* X \to \colim_{\pfilt{0^{k+1},[0]}} X
    \end{align*}
    is a weak equivalence when $k$ is even,
  \item the induced morphism
    \begin{align*}
      \colim_{\Sd \bdsimp{k+1}} t^* X \to \colim_{\pfilt{0^{k+1},[0]}} X
    \end{align*}
    is a weak equivalence when $k$ is odd.
  \end{enumerate}
\end{lemma}

\begin{proof}
  We prove both statements by an alternating induction \wrt{} $k$.

  The functor $\Sd[0] \to \pfilt{0,[0]}$ is an isomorphism,
  so condition (1) holds for $k = 0$.

  Next, we assume that condition (2) holds for a given odd $k$
  and prove that condition (1) holds for $k+1$.
  The category $\Sd \bdsimp{k+1}$ is nothing but the latching category
  of $\pfilt{0^{k+2},[0]}$ at $0^{k+2}$ and hence the inductive construction
  of the colimit of $X$ yields a pushout square
  \begin{ctikzpicture}
    \matrix[diagram]
    {
      |(Lk1)| \colim_{\Sd \bdsimp{k+1}} t^* X &
        |(ck)| \colim_{\pfilt{0^{k+1},[0]}} X \\
      |(k1)| \colim_{\Sd \simp{k+1}} t^* X &
        |(ck1)| \colim_{\pfilt{0^{k+2},[0]}} X \\
    };

    \draw[cof] (Lk1) to (k1);
    \draw[->]  (ck)  to (ck1);
    \draw[->]  (Lk1) to (ck);
    \draw[->]  (k1)  to (ck1);
  \end{ctikzpicture}
  where the top morphism is a weak equivalence by the inductive hypothesis.
  Since the left vertical morphism is a cofibration,
  it follows that the bottom morphism is also a weak equivalence.

  Finally, we assume that condition (1) holds for a given even $k$ and prove
  that condition (2) holds for $k+1$. We have the following
  diagram of \hore{} direct categories
  \begin{ctikzpicture}
    \matrix[diagram,column sep=4em,row sep=3em]
    {
      |(h)| \Sd\hornm{k+2,k+2} & |(b)| \Sd\bdsimp{k+2} &
       |(s1)| \pfilt{0^{k+2},[0]} \\
      |(d1)| \pfilt{0^{k+1} 1,\hat{[1]}} &
        |(?)| \pfilt{{0^{k+1} 1,0^{k+2}},\hat{[1]}} &
        |(s2)| \pfilt{0^{k+2},[0]} \\
      & |(d2)| \pfilt{0^{k+1} 1,\hat{[1]}} & |(p)| \pfilt{0^{k+1},[0]} \\
    };

    \draw[->]  (h)  to (d1);
    \draw[->]  (b)  to (?);
    \draw[inj] (h)  to (b);
    \draw[inj] (d1) to (?);

    \draw[inj] (s2) to (?);
    \draw[inj] (p)  to (d2);
    \draw[inj] (p)  to (s2);
    \draw[inj] (d2) to (?);

    \draw[->] (b)  to (s1);
    \draw[->] (?)  to (s1);

    \draw[->] (s2) to node[right] {$\id$} (s1);
  \end{ctikzpicture}
  where the indicated maps are sieves,
  the top left and bottom right squares are pushouts
  and all functors respect Reedy cofibrant diagrams by \Cref{sieve-exact}.
  (The functor on the very left is induced by $0^{k+2} 1 \from [k+2] \to [1]$.)
  Hence there is an induced diagram in $\cat{C}$
  \begin{ctikzpicture}
    \matrix[diagram]
    {
      |(h)| \colim_{\Sd\hornm{k+2,k+2}} t^* X &
        |(b)| \colim_{\Sd\bdsimp{k+2}} t^* X &
        |(s1)| \colim_{\pfilt{0^{k+2},[0]}} X \\
      |(d1)| \colim_{\pfilt{0^{k+1} 1,\hat{[1]}}} t^* X &
        |(?)| \colim_{\pfilt{{0^{k+1} 1,0^{k+2}},\hat{[1]}}} t^* X &
        |(s2)| \colim_{\pfilt{0^{k+2},[0]}} X \\
      & |(d2)| \colim_{\pfilt{0^{k+1} 1,\hat{[1]}}} t^* X &
        |(p)| \colim_{\pfilt{0^{k+1},[0]}} X \\
    };

    \draw[->]  (h)  to (d1);
    \draw[->]  (b)  to (?);
    \draw[cof] (h)  to (b);
    \draw[cof] (d1) to (?);

    \draw[cof] (s2) to (?);
    \draw[cof] (p)  to (d2);
    \draw[cof] (p)  to (s2);
    \draw[cof] (d2) to (?);

    \draw[->] (b)  to (s1);
    \draw[->] (?)  to (s1);

    \draw[->] (s2) to node[right] {$\id$} (s1);
  \end{ctikzpicture}
  where the indicated maps are cofibrations
  and the top left and bottom right squares are pushouts.
  Thus the proof will be completed when we verify that both morphisms
  \begin{align*}
    \colim_{\Sd\hornm{k+2,k+2}} t^* X
      \to \colim_{\pfilt{0^{k+1} 1,\hat{[1]}}} t^* X \\
    \colim_{\pfilt{0^{k+1},[0]}} t^* X
      \to \colim_{\pfilt{0^{k+1} 1,\hat{[1]}}} X
  \end{align*}
  are weak equivalences.
  For the former we use
  \Cref{cone-filt-contractible,equiv-cofinal,special-horns-nf}.
  For the latter we use \Cref{cone-filt-contractible,equiv-cofinal}
  and the inductive assumption.
\end{proof}

In the next two lemmas we generalize the filtration of $D[0]$ to $D[m]$
for all $m \ge 0$.

\begin{lemma}\label{odd-sieves}
  Let $\cat{C}$ be a cofibration category.
  Assume that every fiber of $\phi \from [k] \to [m]$
  has an odd number of elements and let $X \from D^\phi[m] \to \cat{C}$
  be a \hore{} Reedy cofibrant diagram.
  Then $X_\phi \to \colim X$ is a weak equivalence.
\end{lemma}

\begin{proof}
  We proceed by induction \wrt{} $m$ (simultaneously for all $\cat{C}$ and $X$).
  For $m = 0$ the conclusion follows by \Cref{filt-even-odd}.

  If $m > 0$, we will prolong $X$ to the augmented sieve $D_\aug^\phi[m]$
  by setting the missing value to an initial object of $\cat{C}$
  which does not change the colimit.
  If the fiber of $\phi$ over $m$ has $k+1$ elements for some even $k$,
  then $D_\aug^\phi[m] \iso D_\aug^{\phi'}[m-1] \times D_\aug^{0^{k+1}}[0]$.
  (Here, $\phi'$ is the restriction of $\phi$ to $\phi^{-1}[m-1]$.)
  By applying \Cref{filt-even-odd} in the category
  $\cat{C}^{D_\aug^{\phi'}[m-1]}_\Reedy$ to the corresponding diagram
  $\tilde{X} \from D_\aug^{0^{k+1}}[0] \to \cat{C}^{D_\aug^{\phi'}[m-1]}_\Reedy$
  we obtain a weak equivalence
  $\tilde{X}_k \to \colim_{D_\aug^{0^{k+1}}[0]} \tilde{X}$
  and hence by the inductive assumption the composite
  \begin{align*}
    X_\phi = \tilde{X}_{k,\phi'} \to \colim_{D_\aug^{\phi'}[m-1]} \tilde{X}_k
      \to \colim_{D_\aug^{\phi'}[m-1]} \colim_{D_\aug^{0^{k+1}}[0]} \tilde{X}
      \cong \colim X
  \end{align*}
  is also a weak equivalence.
\end{proof}

For each $k, m \ge 0$ we define sets $A_{k,m}$ and $B_{k,m}$
of objects of $D[m]$.
We proceed by induction \wrt{} $m$.
First, we set $A_{k,0} = B_{k,0} = \{ [2k] \to [0]\}$.
For $m > 0$ we set
\begin{align*}
  B_{k,m} & = \{ \phi \from [2k-m] \to [m] \mid \text{each fiber of } \phi
              \text{ has an odd number of elements} \} \\
  A_{k,m} & = B_{k,m} \union \bigunion_{i \in [m]} \face_i A_{k,m-1} \text{.}
\end{align*}

We set $\filt{k,[m]} = D^{A_{k,m}}[m]$.
(In particular, we have $\filt{k,[0]} = \pfilt{[2k],[0]}$.)

\begin{lemma}\label{filt-simp-op}
  For every simplicial operator $\chi \from [m] \to [n]$ and $k \ge 0$
  we have an inclusion $\chi \filt{k,[m]} \subseteq \filt{k,[n]}$.
\end{lemma}

\begin{proof}
  It suffices to verify the statement when $\chi$
  is an elementary face or degeneracy operator.
  For the elementary face operators it follows directly from the definition.
  Hence assume that $\chi = \dgn_j$ for some $j \in [n]$.
  We will check that $\dgn_j A_{k,n+1} \subseteq \filt{k,[n]}$
  by induction \wrt{} $n$.

  If $\phi \from [2k-n-1] \to [n+1]$ has all fibers of odd cardinality,
  then the same holds for $\dgn_j \phi$ except at the fiber over $j$.
  Then $\dgn_j \phi$ is in the sieve generated by $\phi' \from [2k-n] \to [n]$
  obtained by adding one extra element to the fiber of $\dgn_j \phi$ over $j$
  (so that $\phi' \in A_{k,n}$).

  If $\psi \in A_{k,n}$, then $\dgn_j \face_i \psi$ is either equal to $\psi$
  or is of the form $\face_{i'} \dgn_{j'} \psi$.
  In the first case the conclusion holds trivially,
  in the second one it follows by the inductive hypothesis.
\end{proof}

Now, we can generalize the filtration of $D[m]$ to $DK$
for arbitrary finite $K$.
Let $x \in K_m$ and $k \ge 0$.
We define a sieve $\filt{k,K}$ in $DK$ as follows.
Write $x = x^\sharp x^\flat$ with $x^\sharp$ non-degenerate and $x^\flat$
a degeneracy operator.
Define $x$ to be an element of $\filt{k,K}$ if $x^\flat \in \filt{k,[n]}$
(where $n$ is the dimension of $x^\sharp$).
It follows from \Cref{filt-simp-op} that this definition coincides
with the previous one when $K$ is a simplex.

\begin{lemma}
  Every simplicial map $f \from K \to L$ carries $\filt{k,K}$ to $\filt{k,L}$
  for all $k \ge 0$.
\end{lemma}

\begin{proof}
  Let $x \in \filt{k,K}$.
  Then we have a diagram of simplicial sets
  \begin{ctikzpicture}
    \matrix[diagram]
    {
      |(m)| \simp{m} & |(n)|  \simp{n}  & |(K)| K \\
                     & |(n')| \simp{n'} & |(L)| L \\
    };

    \draw[->] (m) to node[above] {$x^\flat$}  (n);
    \draw[->] (n) to node[above] {$x^\sharp$} (K);

    \draw[->] (m) to node[below left] {$(fx)^\flat$}        (n');
    \draw[->] (n) to node[right]      {$(fx^\sharp)^\flat$} (n');
    \draw[->] (K) to node[right]      {$f$}                 (L);

    \draw[->] (n') to node[below] {$(fx^\sharp)^\sharp$} (L);
  \end{ctikzpicture}
  and by definition $x^\flat \in \filt{k,[n]}$.
  \Cref{filt-simp-op} implies that $(fx)^\flat \in \filt{k,[n']}$
  so that $fx \in \filt{k,L}$.
\end{proof}

\begin{lemma}\label{filt-simp}
  For all $k \ge m$, a cofibration category $\cat{C}$
  and a \hore{} Reedy cofibrant diagram $X \from \filt{k,[m]} \to \cat{C}$
  the morphism $X_{[m]} \to \colim_{\filt{k,[m]}} X$ is a weak equivalence.
\end{lemma}

\begin{proof}
  First, we will check that the morphism $X_{[m]} \to \pfilt{B_{k,m},[m]}$
  is a weak equivalence.
  Indeed, let $P$ be the subposet of $\nat^{m+1}$ consisting of tuples
  $x = (x_0, \ldots, x_m)$ such that each $x_i$ is odd
  and $x_0 + \ldots + x_m \le 2k - m + 1$.
  Let $\phi_x$ be the unique object of $D[m]$ whose fiber over each $i \in [m]$
  has cardinality $x_i$.
  Then we have $\pfilt{B_{k,m},[m]} = \colim_{x \in P} \pfilt{\phi_x,[m]}$.
  It follows from \Cref{odd-sieves} that for each $x \in P$
  the morphism $X_{[m]} \to \colim_{\pfilt{\phi_x,[m]}} X$ is a weak equivalence.
  The sequence $(1,\ldots,1)$ is the bottom element of $P$,
  hence if we consider $P$ as a \hore{} poset
  with all maps as weak equivalences,
  then $\{(1,\ldots,1)\} \to P$ is a homotopy equivalence.
  It follows by \Cref{equiv-cofinal} that $X_{[m]} \to \pfilt{B_{k,m},[m]}$
  is a weak equivalence.

  We are ready to prove the lemma by induction \wrt{} $m$.
  We have
  \begin{align*}
    \pfilt{B_{k,m},[m]} \inter \pfilt{\face_i A_{k,m-1},[m]}
      = \pfilt{\face_i A_{k-1,m-1},[m]} \iso \pfilt{A_{k-1,m-1},[m-1]}
  \end{align*}
  and
  $\filt{k,\bdsimp{m}}
    = \colim_{\phi \in \Sd \bdsimp{m}} \pfilt{\phi A_{k,m-1},[m]}$
  (and the same with $k-1$ in place of $k$).
  Hence by the inductive assumption the morphism
  $\colim_{\filt{k-1,\bdsimp{m}}} X \to \colim_{\filt{k,\bdsimp{m}}} X$
  is an acyclic cofibration.
  This along with the first part of the proof and the pushout square
  \begin{ctikzpicture}
    \matrix[diagram]
    {
      |(k1b)| \colim_{\filt{k-1,\bdsimp{m}}} X &
        |(Bk)| \colim_{\pfilt{B_{k,m},[m]}} X \\
      |(kb)|  \colim_{\filt{k,\bdsimp{m}}} X & |(k)|  \colim{\filt{k,[m]}} \\
    };

    \draw[->] (k1b) to (Bk);
    \draw[->] (kb)  to (k);

    \draw[cof] (k1b) to (kb);
    \draw[cof] (Bk)  to (k);
  \end{ctikzpicture}
  finishes the proof.
\end{proof}

\begin{lemma}\label{Dk-preserves-colimits}
  For each $k$ the functor $D^{(k)} \from \sSet \to \Cat$
  (i.e.\ when we disregard the \hore{} structures of $\filt{k,K}$s)
  preserves colimits.
\end{lemma}

\begin{proof}
  If $K$ is any simplicial set, then $D^{(k)}$ preserves the colimit
  of its simplices by \Cref{D-preserves-colimits} and the definition
  of $\filt{k,K}$.
  Hence for every small category $J$ we have the following sequence
  of isomorphisms natural in both $K$ and $J$.
  \begin{align*}
    \Cat(\filt{k,K},J) & \iso \Cat(\filt{k,\colim_{\simp{m} \to K} \simp{m}}, J) \\
      & \iso \lim_{\simp{m} \to K} \Cat(\filt{k,[m]}, J) \\
      & \iso \lim_{\simp{m} \to K} \sSet(\simp{m}, \Cat(\filt{k,[\uvar]}, J)) \\
      & \iso \sSet(K, \Cat(\filt{k,[\uvar]}, J))
  \end{align*}
  It follows that $J \mapsto \Cat(\filt{k,[\uvar]}, J)$ is a right adjoint
  of $D^{(k)}$ and the conclusion follows.
\end{proof}

Finally, we are ready to start translating the results of
\Cref{sec:cocompleteness-infinite} to the case of $\kappa = \aleph_0$.
The following is a counterpart to \Cref{colim-simplex-join}.

\begin{lemma}\label{colim-simplex-join-finite}
  Let $\cat{C}$ be a cofibration category and $K$ a finite simplicial set.
  For every \hore{} Reedy cofibrant diagram
  $X \from D(K \join \Delta[m]) \to \cat{C}$ and all $k \ge \dim K + 1 + m$,
  the induced morphism
  \begin{align*}
    X_{[m]} \to \colim_{\filt{k,(K \join \Delta[m]})} X
  \end{align*}
  is a weak equivalence.
\end{lemma}

\begin{proof}
  The morphism in question factors as
  \begin{align*}
    X_{[m]} \to \colim_{\filt{k,[m]}} X
      \to \colim_{\filt{k,(K \join \simp{m})}} X
  \end{align*}
  where the first morphism is a weak equivalence by \Cref{filt-even-odd}.
  Thus it will be enough to check that the second one is.

  It will suffice to verify that this statement holds when $K$
  is empty or a simplex and is preserved under pushouts along monomorphisms.
  For $K = \emptyset$ the morphism in question is an isomorphism.

  Let $K = \simp{n}$ and let $\iota$ be the composite
  $[m] \ito [n] \join [m] \iso [n+1+m]$.
  Then we have a commutative square
  \begin{ctikzpicture}
    \matrix[diagram]
    {
      |(m)|  X_\iota           & |(cm)|  \colim_{\filt{k,[m]}} X \\
      |(mn)| X_{\id_{[n+1+m]}} & |(cmn)| \colim_{\filt{k,[n+1+m]}} X \\
    };

    \draw[->] (m)  to (cm);
    \draw[->] (mn) to (cmn);

    \draw[->] (m)  to (mn);
    \draw[->] (cm) to (cmn);
  \end{ctikzpicture}
  where the left morphism is a weak equivalence since $X$ is \hore{}
  and so are the horizontal ones by \Cref{filt-simp}.
  Thus the right morphism is also a weak equivalence.

  Next, consider a pushout square
  \begin{ctikzpicture}
    \matrix [diagram]
    {
      |(A)| A & |(K)| K \\
      |(B)| B & |(L)| L \\
    };

    \draw[->] (A) to (K);
    \draw[->] (B) to (L);

    \draw[inj] (A) to (B);
    \draw[inj] (K) to (L);
  \end{ctikzpicture}
  \st{} the conclusion holds for $A$, $B$ and $K$.
  The functor $\uvar \join \simp{m}$ preserves pushouts by \Cref{join-colimits}
  and so does $D^{(k)}$ by \Cref{Dk-preserves-colimits}.
  Thus in the cube
  \begin{ctikzpicture}
    \matrix[diagram,row sep=2em,column sep=1.2em]
    {
      |(A)| \filt{k,A} & & |(K)| \filt{k,K} \\
      & |(Am)| \filt{k,(A \join \simp{m})} & & |(Km)| \filt{k,(K \join \simp{m})} \\
      |(B)| \filt{k,B} & & |(L)| \filt{k,L} \\
      & |(Bm)| \filt{k,(B \join \simp{m})} & & |(Lm)| \filt{k,(L \join \simp{m})} \\
    };

    \draw[->] (A) to (K);
    \draw[->] (B) to (L);

    \draw[inj] (A) to (B);
    \draw[inj] (K) to (L);

    \draw[->,cross line] (Am) to (Km);
    \draw[->] (Bm) to (Lm);

    \draw[inj,cross line] (Am) to (Bm);
    \draw[inj] (Km) to (Lm);

    \draw[->] (A) to (Am);
    \draw[->] (K) to (Km);
    \draw[->] (B) to (Bm);
    \draw[->] (L) to (Lm);
  \end{ctikzpicture}
  both the front and the back faces are pushouts along sieves
  and the conclusion follows by \cite{rb}*{Theorem 9.4.1(1a)}
  and the Gluing Lemma
  (since $\dim L = \max\{\dim B, \dim K\}$).
\end{proof}

For a cofibration category $\cat{C}$ we introduce a new cofibration category
$\cat{C}^{\tilde\nat}_\Reedy$. (Here, $\tilde\nat$ does not refer
to any \hore{} structure on $\nat$, $\cat{C}^{\tilde\nat}_\Reedy$
should be seen as an atomic notation.)
Its objects are Reedy cofibrant diagrams $X \from \nat \to \cat{C}$
(i.e.\ sequences of cofibrations in $\cat{C}$)
that are \emph{eventually (homotopically) constant},
i.e.\ \st{} there is a number $k$ \st{} for all $l \ge k$ the morphism
$X_k \to X_l$ is a weak equivalence.
A morphism $f \from X \to Y$ of such diagrams is called
an \emph{eventual weak equivalence} if there is $k$ \st{} for all $l \ge k$
the morphism $f_l$ is a weak equivalence in $\cat{C}$.
This cofibration category is designed as an enlargement
of the cofibration category $\cat{C}^{\hat\nat}_\Reedy$
of \emph{(homotopically) constant} sequences.
It is necessary since sequences arising as colimits over filtrations
$\filt{\uvar,K}$ are only eventually constant.

\begin{lemma}\label{eventually-constant}
  If $\cat{C}$ is a cofibration category, then the category
  $\cat{C}^{\tilde\nat}_\Reedy$ with Reedy cofibrations
  and eventual weak equivalences is also a cofibration category.
  Moreover, the inclusion
  $\cat{C}^{\hat\nat}_\Reedy \ito \cat{C}^{\tilde\nat}_\Reedy$
  is a weak equivalence.
\end{lemma}

\begin{proof}
  The construction of the cofibration category $\cat{C}^{\tilde\nat}_\Reedy$
  is a straightforward modification of the construction
  of $\cat{C}^\nat_\Reedy$, see e.g.\ \cite{rb}*{Theorem 9.3.5(1)}.

  We will verify the approximation properties.
  By ``2 out of 3'' a morphism between homotopically constant sequences
  is a levelwise weak equivalence if and only if
  it is an eventual weak equivalence.
  Hence (App1) holds.

  Next, let $X \to Y$ be a morphism with $X$ homotopically constant
  and $Y$ eventually constant.
  Assume that $Y$ is homotopically constant from degree $k$ on.
  Let $\tilde Y$ be $Y$ shifted down by $k$.
  Then $\tilde Y$ is homotopically constant and iterated structure morphisms
  of $Y$ yield a morphism $Y \to \tilde Y$ which is an eventual weak equivalence
  (starting from $k$).
  This yields a commutative square
  \begin{ctikzpicture}
    \matrix[diagram]
    {
      |(X)|   X        & |(Y)|   Y \\
      |(tY0)| \tilde Y & |(tY1)| \tilde Y \\
    };

    \draw[->] (X) to (Y);
    \draw[->] (X) to (tY0);

    \draw[->] (Y)   to node[right] {$\we$} (tY1);
    \draw[->] (tY0) to node[below] {$\id$} (tY1);
  \end{ctikzpicture}
  which proves (App2).
\end{proof}

We define a functor $|\uvar| \from DK \to \nat$ by sending $x \in DK$
to the smallest $k \in \nat$ \st{} $x \in \filt{k,K}$.
We call $|x|$ the \emph{filtration degree} of $x$.
Here, we do not consider any particular \hore{} structure on $\nat$
so $|\uvar|$ is not a \hore{} functor.
We will be interested in the left Kan extension
of a \hore{} Reedy cofibrant diagram $X \from DK \to \cat{C}$ along $|\uvar|$.
It can be computed as
\begin{align*}
  (\Lan_{|\uvar|} X)_k = \colim_{\filt{k,K}} X \text{.}
\end{align*}
We will denote $(\Lan_{|\uvar|} X)_k$ by $\Phi^{(k)} X$ and when $k$ varies
$\Phi^{(\uvar)} X$ will stand for the resulting sequence $\nat \to \cat{C}$.

Just as colimits can be defined in terms of cones,
left Kan extensions can be defined in terms of certain generalized cones.
We describe such cones for Kan extensions along $|\uvar|$.
Let $DK \join_{|\uvar|} \nat$ denote the \emph{cograph} (or \emph{collage})
of $|\uvar|$ defined as the category whose set of objects
is the disjoint union of the sets of objects of $DK$ and $\nat$ and
\begin{align*}
  (DK \join_{|\uvar|} \nat) (x, y) =
  \begin{cases}
    DK(x, y)     & \text{when } x, y \in DK \text{,} \\
    \nat(x, y)   & \text{when } x, y \in \nat \text{,} \\
    \nat(|x|, y) & \text{when } x \in DK \text{ and } y \in \nat \text{,} \\
    \emptyset    & \text{otherwise.}
  \end{cases}
\end{align*}
The left Kan extension of $X \from DK \to \cat{C}$ along $|\uvar|$
is nothing but an initial extension of $X$ to $DK \join_{|\uvar|} \nat$
so that morphisms $\Phi^{(\uvar)} X \to Y$ in $\cat{C}^\nat$
correspond to diagrams on $DK \join_{|\uvar|} \nat$ restricting to $X$ and $Y$
on $DK$ and $\nat$ respectively.
Such an extension of $X$ is a family of cones under the restrictions of $X$
to all $\filt{k,K}$s.
We will compare them to extensions to $D(K^\ucone)$ using a functor
$p_K \from D(K^\ucone) \to DK \join_{|\uvar|} \nat$ defined as follows.
Write an object of $D(K^\ucone)$ as $x \join \phi$ with $x \in D^a K$
and $\phi \in D^a [0]$ and set
\begin{align*}
  p_K(x \join \phi) =
  \begin{cases}
    |x \join \phi| & \text{when } \phi \in D[0] \text{,} \\
    x              & \text{otherwise.}
  \end{cases}
\end{align*}
This allows us to state and prove a version of \Cref{cone-factorization}
for finitely cocomplete cofibration categories.

\begin{lemma}\label{cone-factorization-finite}
  Let $\cat{C}$ be a cofibration category, $K$ a finite simplicial set
  and $X \from DK \to \cat{C}$ a \hore{} Reedy cofibrant diagram.
  Consider a morphism $f \from \Phi^{(\uvar)} X \to Y$
  and the corresponding cone
  $\tilde T \from DK \join_{|\uvar|} \nat \to \cat{C}$.
  If $T$ is any Reedy cofibrant replacement of $p_K^* \tilde T$ relative to $DK$
  (which exists by \Cref{relative-factorization}), then $f$ factors as
  \begin{align*}
    \Phi^{(\uvar)} X \to \Phi^{(\uvar)} T \weto Y
  \end{align*}
  where the latter morphism is an eventual weak equivalence
  (starting at $\dim K + 1$).
\end{lemma}

\begin{proof}
  To verify that the above composite agrees with $f$ it suffices to check that
  at each level $k$ it agrees upon precomposition with $X_x \to \Phi^{(k)} X$
  for all $x \in \filt{k,K}$.
  That's indeed the case since $T|DK = X$.

  It remains to check that the latter morphism
  is an eventual weak equivalence.
  For $i \ge \dim K + 1$ in the diagram
  \begin{ctikzpicture}
    \matrix[diagram]
    {
      |(KT)| \colim_{\filt{i,(K^\ucone)}} T & |(Y)| Y_i \\
      |(Ti)| T_{0^{2i+1}} \\
    };

    \draw[->] (KT) to (Y);
    \draw[->] (Ti) to (KT);

    \draw[->] (Ti) to (Y);
  \end{ctikzpicture}
  the left morphism is a weak equivalence by \Cref{colim-simplex-join-finite}
  and so is the diagonal one since $T$
  is a cofibrant replacement of $p_K^* \tilde T$.
  Therefore the top morphism is also a weak equivalence.
\end{proof}

For every $m \ge 0$ each object of $D(K \join \simp{m})$
can be uniquely written as $x \join \phi$ with $x \in D^a K$
and $\phi \in D^a[m]$.
This yields a functor $r_K \from D(K \join \simp{m}) \to D^a[m]$ sending
$x \join \phi$ to $\phi$ and to which we can associate
the ``filtered'' left Kan extension functor
\begin{align*}
  \Lan_{r_K}^{\mathrm{filt}} \from \cat{C}^{D(K \join \simp{m})}_\Reedy
  \to (\cat{C}^{\tilde\nat}_\Reedy)^{D^a[m]}_\Reedy
\end{align*}
defined as $(\Lan_{r_K}^{\mathrm{filt}} X)_\phi = \Phi^{(\uvar)} \phi^* X$
for $\phi \in D^a[m]$ which is exact by \cite{rb}*{Theorem 9.4.3(1)}.
Similarly we have
\begin{align*}
  \Lan_{s_K}^{\mathrm{filt}}
  \from (\cat{C}^{\tilde\nat}_\Reedy)^{D(K \join \bdsimp{m})}_\Reedy
  \to (\cat{C}^{\tilde\nat}_\Reedy)^{D^a\bdsimp{m}}_\Reedy \text{.}
\end{align*}
We form pullbacks (the front and back squares of the cube)
\begin{ctikzpicture}
  \matrix[diagram]
  {
    |(ju)| \cat{C}^{\tilde D(K \join \simp{m})}_\Reedy & &
    |(au)|  (\cat{C}^{\tilde\nat}_\Reedy)^{\tilde D^a[m]}_\Reedy \\
    & |(bju)| \cat{C}^{\tilde D(K \join \bdsimp{m})}_\Reedy & &
    |(bau)| (\cat{C}^{\tilde\nat}_\Reedy)^{\tilde D^a\bdsimp{m}}_\Reedy \\
    |(j)| \cat{C}^{D(K \join \simp{m})}_\Reedy & &
    |(a)|  (\cat{C}^{\tilde\nat}_\Reedy)^{D^a[m]}_\Reedy \\
    & |(bj)| \cat{C}^{D(K \join \bdsimp{m})}_\Reedy & &
    |(ba)| (\cat{C}^{\tilde\nat}_\Reedy)^{D^a\bdsimp{m}}_\Reedy \text{.} \\
  };

  \draw[->] (ju) to (au);
  \draw[->] (j)  to node[below right] {$\Lan_{r_K}^\mathrm{filt}$} (a);

  \draw[fib] (ju) to (j);
  \draw[fib] (au) to (a);

  \draw[->,cross line]   (bju) to (bau);
  \draw[fib, cross line] (bju) to (bj);

  \draw[->]  (bj)  to node[below right] {$\Lan_{s_K}^\mathrm{filt}$} (ba);
  \draw[fib] (bau) to (ba);

  \draw[->,shorten >=6pt] (ju) to node[above right] {$P_K$} (bju);

  \draw[->,shorten >=3pt] (au) to (bau);
  \draw[->,shorten >=6pt] (j)  to (bj);
  \draw[->,shorten >=4pt] (a)  to (ba);
\end{ctikzpicture}
Observe that $\cat{C}^{\tilde D(K \join \simp{m})}$
and $\cat{C}^{\tilde D(K \join \bdsimp{m})}$ are just atomic notations
for the pullbacks above, i.e.\ $\tilde D(K \join \simp{m})$
and $\tilde D(K \join \bdsimp{m})$ are \emph{not} \hore{} categories
for general $K$, although they will be interpreted as such
when $K$ is a simplex.

The following is a finite variant of \Cref{universal-cones}.

\begin{lemma}\label{universal-cones-finite}
  The functor
  $P_K \from \cat{C}^{\tilde D(K \join \simp{m})}_\Reedy
    \to \cat{C}^{\tilde D(K \join \bdsimp{m})}_\Reedy$ is an acyclic fibration
  for every finite simplicial set $K$.
\end{lemma}

\begin{proof}
  The proof is virtually identical to the proof of \Cref{universal-cones} except
  that now we do not consider the cases of coproducts and colimits of sequences
  of monomorphisms and we use \Cref{colim-simplex-join-finite} in the place of
  \Cref{colim-simplex-join}.
\end{proof}

Finally, we can characterize colimits in $\nf \cat{C}$ in terms of
homotopy colimits in $\cat{C}$ in a manner similar to \Cref{nerve-cones}.

\begin{proposition}\label{nerve-cones-finite}
  Let $\cat{C}$ be cofibration category, $K$ a finite simplicial set.
  A cone $S \from K^\ucone \to \nf \cat{C}$ is universal \iff{} the induced morphism
  \begin{align*}
    \Phi^{(\uvar)}(S|K) \to \Phi^{(\uvar)} S
  \end{align*}
  is an eventual weak equivalence (where $S$ is seen as
  a \hore{} Reedy cofibrant diagram $D(K^\ucone) \to \cat{C}$
  by \Cref{nf-representation}).
  Such a cone exists under every diagram $K \to \nf \cat{C}$.
\end{proposition}

\begin{proof}
  The proof is almost identical to the proof of \Cref{nerve-cones}
  except that we use
  \Cref{universal-cones-finite,cone-factorization-finite,colim-simplex-join-finite}
  in the place of \Cref{universal-cones,cone-factorization,colim-simplex-join}
  respectively.
\end{proof}

The more specific criteria for initial objects and pushouts discussed
in \Cref{ex:init,ex:push} are valid in the finitely cocomplete case
in exactly the same form.
This can be justified by observing that $\Phi^{(k)}$ stabilizes at $k = 0$
over $D[0]$ and at $k = 2$ over $D([1] \times [1])$
by \Cref{colim-simplex-join-finite}.

\begin{proof}[Proof of \Cref{nerve-exact}]
  Since we have already verified \Cref{nerve-limits,nf-inner,nf-iso}
  it remains to check that $\nf$ takes values
  in finitely cocomplete quasicategories and exact functors.

  It takes values in quasicategories by \Cref{nf-inner}
  and they are finitely cocomplete by \Cref{nerve-cones-finite}.

  Similarly, colimits in quasicategories of frames were characterized
  in \Cref{nerve-cones-finite} by certain morphisms being weak equivalences
  and weak equivalences are preserved by exact functors by \Cref{Ken-Brown}.
\end{proof}

  \section{Cofibration categories of diagrams\\in quasicategories}
  \label{ch:cofcats-of-dgrms}
  In this section we will prove our main result, i.e.\ that $\nf$
is a weak equivalence of fibration categories.
This will be achieved by defining a functor $\dgrm_\kappa$
from the category of $\kappa$-cocomplete quasicategories
to the category of $\kappa$-cocomplete cofibration categories.
The functor $\dgrm_\kappa$ fails to be exact
(e.g.\ it doesn't preserve the terminal object), but it will be verified
to induce an inverse to $\nf$ on the level of homotopy categories
which is sufficient to complete the proof.

\subsection{Construction}
\label{category-of-diagrams}

Let $\sSet_\kappa$ denote the category of $\kappa$-small simplicial sets.
If $\qcat{C}$ is a $\kappa$-cocomplete quasicategory we consider
the slice category $\sSet_\kappa \slice \qcat{C}$,
we denote it by $\dgrm_\kappa \qcat{C}$
and call the \emph{category of $\kappa$-small diagrams in $\qcat{C}$}.
Then we define a morphism
\begin{ctikzpicture}
  \matrix [diagram]
  {
    |(K)| K & & |(L)| L \\
    & |(C)| \qcat{C} \\
  };

  \draw[->] (K) to node[above] {$f$} (L);

  \draw[->] (K) to node[below left]  {$X$} (C);
  \draw[->] (L) to node[below right] {$Y$} (C);
\end{ctikzpicture}
to be
\begin{itemize}
  \item a \emph{weak equivalence} if the induced morphism
    $\colim_K X \to \colim_L Y$ is an equivalence in $\qcat{C}$
    (more precisely, if for any universal cone
    $S \from L^\ucone \to \qcat{C}$ under $Y$
    the induced cone $S f^\ucone$ is universal under $X$),
  \item a \emph{cofibration} if $f$ is injective.
\end{itemize}

In particular, such a morphism is a weak equivalence whenever $f$ is cofinal,
but there are of course many weak equivalences with $f$ not cofinal.
We will make use of the class of \emph{right anodyne maps} which is generated
by the \emph{right horn} inclusions $\horn{m,i} \ito \simp{m}$
(i.e. the ones with $0 < i \le m$) under coproducts,
pushouts along arbitrary maps, sequential colimits and retracts.

\begin{lemma}\label{right-anodyne}
  Every right anodyne map is cofinal.
\end{lemma}

\begin{proof}
  \cite{l}*{Proposition 4.1.1.3(4)}
\end{proof}

\begin{proposition}
  With weak equivalences and cofibrations as defined above
  $\dgrm_\kappa \qcat{C}$ is a $\kappa$-cocomplete cofibration category.
\end{proposition}

\begin{proof}
  \mbox{}
  \begin{enumerate}[label=\thm]
    \item[(C0)] Weak equivalences satisfy ``2 out of 6'' since equivalences
      in $\qcat{C}$ do.
    \item[(C1)] Isomorphisms are weak equivalences since isomorphisms
      of simplicial sets are cofinal.
    \item[(C2-3)] The empty diagram is an initial object
      and hence every object is cofibrant.
    \item[(C4)] Pushouts are created by the forgetful functor
      $\dgrm_\kappa \qcat{C} \to \sSet_\kappa$
      thus pushouts along cofibrations exist
      and cofibrations are stable under pushouts.
      By \cite{rb}*{Lemma 1.4.3(1)} it suffices to verify that
      the Gluing Lemma holds which follows by \cite{l}*{Proposition 4.4.2.2}.
    \item[(C5)] It will suffice to verify that
      in the usual mapping cylinder factorization
      \begin{align*}
        K \to M f \to L
      \end{align*}
      the second map is cofinal. Indeed, we have a diagram
      \begin{ctikzpicture}
        \matrix[diagram]
        {
          |(K)|  K \times \simp{0} & |(L)| L \\
          |(K1)| K \times \simp{1} & |(Mf)| M f \\
          & & |(L')| L \\
        };

        \draw[->] (K)  to node[above] {$f$} (L);
        \draw[->] (K1) to (Mf);

        \draw[->] (K) to node[left] {$K \times \face_0$} (K1);
        \draw[->] (L) to node[left] {$j$} (Mf);

        \draw[->] (Mf) to (L');
        \draw[->] (L)  to[bend left] node[right] {$\id_L$} (L');
      \end{ctikzpicture}
      where the square is a pushout.
      The map $K \times \face_0$ is right anodyne by \cite{jo}*{Theorem 2.17}
      and thus so is $j$.
      Hence it is cofinal by \Cref{right-anodyne}.
    \item[(C6-7-$\kappa$)] The proof is similar to that of (C4).
      (But there is no analogue of \cite{l}*{Proposition 4.4.2.2}
      for sequential colimits explicitly stated in \cite{l}.
      Instead, it follows from more general
      \cite{l}*{Proposition 4.2.3.10 and Remark 4.2.3.9}.)
      \qedhere
  \end{enumerate}
\end{proof}

\begin{lemma}
  A $\kappa$-cocontinuous functor $F \from \qcat{C} \to \qcat{D}$
  induces a $\kappa$-cocontinuous functor
  $\dgrm_\kappa F = \dgrm_\kappa \qcat{C} \to \dgrm_\kappa \qcat{D}$
  and thus we obtain a functor
  $\dgrm_\kappa \from \QCat_\kappa \to \CofCat_\kappa$.
\end{lemma}

\begin{proof}
  Colimits in both $\dgrm_\kappa \qcat{C}$ and $\dgrm_\kappa \qcat{D}$
  are created in $\sSet_\kappa$ and thus are preserved by $\dgrm_\kappa F$.
  Cofibrations are clearly preserved and so are weak equivalences
  since $F$ preserves $\kappa$-small colimits.
\end{proof}

  \subsection{Proof of the main theorem: the infinite case}
\label{sec:proof-infinite}

For a $\kappa$-cocomplete cofibration category $\cat{C}$ we define a functor
$\Phi_{\cat{C}} \from \dgrm_\kappa \nf \cat{C} \to \cat{C}$
by sending a diagram $X \from K \to \nf \cat{C}$ to $\colim_{DK} X$
(observe that $DK$ is $\kappa$-small since $K$ is and $\kappa > \aleph_0$,
so this colimit exists in $\cat{C}$).
It is clear that $\Phi_{\cat{C}}$ is a functor.
While we may not be able to choose colimits so that $\Phi_{\cat{C}}$
is natural in $\cat{C}$, it is $2$-natural,
i.e.\ natural up to coherent natural isomorphism.

\begin{lemma}\label{Phi-equivalence}
  The functor $\Phi_{\cat{C}}$ is $\kappa$-cocontinuous and a weak equivalence.
\end{lemma}

\begin{proof}
  Preservation of cofibrations follows by \cite{rb}*{Theorem 9.4.1(1a)}
  since if $K \ito L$ is an injective map of simplicial sets,
  then the induced functor $DK \ito DL$ is a sieve.

  \Cref{nerve-cones,colim-simplex-join} imply that a morphism $f$
  in $\dgrm_\kappa \nf \cat{C}$ is a weak equivalence
  \iff{} $\Phi_{\cat{C}} f$ is.
  Therefore $\Phi_{\cat{C}}$ preserves weak equivalences and satisfies (App1).

  Colimits in $\cat{C}$ are compatible with colimits of indexing categories
  and thus $\Phi_{\cat{C}}$ is $\kappa$-cocontinuous.

  It remains to check (App2),
  but it follows directly from \Cref{cone-factorization}.
\end{proof}

Next, we need a functor $\qcat{D} \to \nf \dgrm_\kappa \qcat{D}$
for every $\kappa$-cocomplete quasicategory $\qcat{D}$.
Let's start with unraveling the definition of $\nf \dgrm_\kappa \qcat{D}$.

An $m$-simplex of $\nf \dgrm_\kappa \qcat{D}$ consists of
a Reedy cofibrant diagram $K \from D[m] \to \sSet_\kappa$
and for each $\phi \in D[m]$ a diagram $X_\phi \from K_\phi \to \qcat{D}$.
These diagrams are compatible with each other in the sense that
they form a cone under $K$ with the vertex $\qcat{D}$.
Moreover, the entire structure is \hore{} as a diagram
in $\dgrm_\kappa \qcat{D}$, i.e.\ if $\phi, \psi \in D[m]$
and $\chi \from \phi \to \psi$ is a weak equivalence, then the induced morphism
$\colim_{K_\phi} X_\phi \to \colim_{K_\psi} X_\psi$
is an equivalence in $\qcat{D}$.

If $\mu \from [n] \to [m]$, then $(K, X) \mu = (K \mu, X \mu)$
is defined simply by $(K \mu)_\phi = K_{\mu \phi}$
and $(X \mu)_\phi = X_{\mu \phi}$.

We can now define a functor
$\Psi_{\qcat{D}} \from \qcat{D} \to \nf \dgrm_\kappa \qcat{D}$ as follows.
For $x \in \qcat{D}_m$ we set the underlying simplicial diagram of
$\Psi_{\qcat{D}} x$ to $\phi \mapsto \simp{k}$ where $\phi \from [k] \to [m]$
and the corresponding diagram in $\qcat{D}$
to $x \phi \from \simp{k} \to \qcat{D}$.
Then $\Psi_{\qcat{D}} x$ is \hore{} as a diagram
$D[m] \to \dgrm_\kappa \qcat{D}$ since any weak equivalence in $D[m]$ induces
a right anodyne (and hence cofinal by \Cref{right-anodyne}) map of simplices.
Clearly, $\Psi_{\qcat{D}}$ is a functor and is natural in $\qcat{D}$.

We will check that $\Psi_{\qcat{D}}$ is a categorical equivalence
by using the following criterion.
A suitable generalization of this criterion holds in any model category,
see \cite{v11}.

\begin{lemma}\label{equivalence-rlpu}
  A functor $F \from \qcat{C} \to \qcat{D}$ between quasicategories
  is a categorical equivalence provided that for every commutative square
  of the form
  \begin{ctikzpicture}
    \matrix [diagram]
    {
      |(bdm)| \bdsimp{m} & |(C)| \qcat{C} \\
      |(m)|   \simp{m}   & |(D)| \qcat{D} \\
    };

    \draw[inj] (bdm) to (m);
    \draw[->]  (C) to node[right] {$F$} (D);

    \draw[->] (bdm) to node[above] {$u$} (C);
    \draw[->] (m)   to node[below] {$v$} (D);
  \end{ctikzpicture}
  there exists a map $w \from \simp{m} \to \qcat{C}$ \st{} $w|\bdsimp{m} = u$
  and $F w$ is $\nwiso$-homotopic to $v$ relative to $\bdsimp{m}$.
\end{lemma}

\begin{proof}
  The class of simplicial maps $K \to L$ with the lifting property \wrt{} $F$
  as in the statement is closed under coproducts, pushouts
  and sequential colimits and thus contains all monomorphisms.
  In particular, if we consider the diagram
  \begin{ctikzpicture}
    \matrix [diagram]
    {
      & |(C)| \qcat{C} \\
      |(D0)| \qcat{D} & |(D1)| \qcat{D} \\
    };

    \draw[->] (C)  to node[right] {$F$}   (D1);
    \draw[->] (D0) to node[below] {$\id$} (D1);
  \end{ctikzpicture}
  we obtain a functor $G \from \qcat{D} \to \qcat{C}$ and an $\nwiso$-homotopy $H$
  from $F G$ to $\id_{\qcat{D}}$ which in turn yields a diagram
  \begin{ctikzpicture}
    \matrix [diagram]
    {
      |(Cbd)| \qcat{C} \times \bdsimp{1} & |(C)| \qcat{C} \\
      |(CE)|  \qcat{C} \times E[1]       & |(D)| \qcat{D} \text{.} \\
    };

    \draw[inj] (Cbd) to (CE);
    \draw[->]  (C)   to node[right] {$F$} (D);

    \draw[->] (Cbd) to node[above] {$[G F, \id]$} (C);
    \draw[->] (CE)  to node[below] {$F H$}        (D);
  \end{ctikzpicture}
  This time a lift is an $\nwiso$-homotopy from $G F$ to $\id_{\cat{C}}$.
  Thus $F$ is an $\nwiso$-equivalence.
\end{proof}

To apply this criterion in our situation we need a method of constructing
relative $\nwiso$-homotopies in quasicategories of the form $\nf \cat{C}$.

\begin{lemma}\label{nat-we-eq}
  \index{E1-homotopy@$\nwiso$-homotopy}
  Let $K \ito L$ be an inclusion of marked simplicial complexes, $X$ and $Y$
  \hore{} Reedy cofibrant diagrams $DL \to \cat{C}$
  and $f \from X|\Sd L \to Y|\Sd L$ a natural weak equivalence
  \st{} $f|\Sd K$ is an identity transformation.
  Then $X$ and $Y$ are $\nwiso$-homotopic relative to $K$ as diagrams
  in $\nf \cat{C}$.
\end{lemma}

\begin{proof}
  By \Cref{eqs-in-nf} it suffices to construct a \hore{} Reedy cofibrant diagram
  $D(L \times \hat{[1]}) \to \cat{C}$ that restricts to $[X, Y]$
  on $D(L \times \bdsimp{1})$ and to the identity on $D(K \times \hat{[1]})$
  (i.e.\ to a degenerate edge of $(\nf \cat{C})^K$).

  First, observe that we have a homotopical diagram
  $[f, \id] \from (\Sd L \union DK) \times \hat{[1]} \to \cat{C}$
  which is Reedy cofibrant when seen as a diagram
  $\Sd L \union DK \to \cat{C}^{\hat{[1]}}$.
  Hence \Cref{relative-D-Sd} implies that it extends to
  a Reedy cofibrant diagram $DL \to \cat{C}^{\hat{[1]}}$.
  We consider it as a diagram $DL \times \hat{[1]} \to \cat{C}$ and pull it back
  to $D(L \times \hat{[1]}) \to \cat{C}$.
  It restricts to $[X, Y]$ on $D(L \times \bdsimp{1})$ and to the identity on
  $D(K \times \hat{[1]})$.
  Thus it can be replaced Reedy cofibrantly relative to
  $D(L \times \bdsimp{1} \union K \times \hat{[1]})$
  by \Cref{relative-factorization} which finishes the proof.
\end{proof}

\begin{proposition}\label{Psi-equivalence}
  For every $\kappa$-cocomplete quasicategory $\qcat{D}$ the functor
  $\Psi_{\qcat{D}}$ is a categorical equivalence.
\end{proposition}

\begin{proof}
  Consider a square
  \begin{ctikzpicture}
    \matrix [diagram]
    {
      |(bdm)| \bdsimp{m} & |(C)| \qcat{D} \\
      |(m)|   \simp{m}   & |(D)| \nf \dgrm_\kappa \qcat{D} \text{.} \\
    };

    \draw[->] (bdm) to (m);
    \draw[->] (C) to node[right] {$\Psi_{\qcat{D}}$} (D);

    \draw[->] (bdm) to node[above] {$x$} (C);
    \draw[->] (m)   to node[below] {$Y$} (D);
  \end{ctikzpicture}
  By \Cref{equivalence-rlpu} it will be enough to extend $x$ to
  a simplex $\hat x \from \simp{m} \to \qcat{D}$
  and construct an $\nwiso$-homotopy from $\Psi_{\qcat{D}} \hat x$ to $Y$
  relative to $\bdsimp{m}$.

  Let's start by finding $\hat x$.
  Consider $Y_{[m]} \from A_{[m]} \to \qcat{D}$.
  Since $Y$ agrees with $\Psi_{\qcat{D}} x$ over $\bdsimp{m}$
  the $[m]$th latching object of $Y$ is $x \from \bdsimp{m} \to \qcat{D}$,
  i.e. we have an induced injective map $\bdsimp{m} \ito A_{[m]}$
  and $Y_{[m]}|\bdsimp{m} = x$.
  Choose a universal cone $\tilde Y_{[m]} \from A_{[m]}^\ucone \to \qcat{D}$
  under $Y_{[m]}$ and consider $\tilde Y_{[m]}|\bdsimp{m}^\ucone$.
  We have $\bdsimp{m}^\ucone \iso \horn{m+1,m+1}$ which is an outer horn.
  However, $\tilde Y_{[m]}|\bdsimp{m}^\ucone$ is special
  since $\Psi_{\qcat{D}} x$ is \hore{} and thus extends to
  $z \from \simp{m}^\ucone \to \qcat{D}$ by \Cref{special-horns}.
  We set $\hat x = z|\simp{m}$.

  By \Cref{nf-representation} finding an $\nwiso$-homotopy
  from $\Psi_{\qcat{D}} \hat x$ to $Y$ translates into constructing
  a \hore{} Reedy cofibrant diagram
  $D([m] \times \wiso) \to \dgrm_\kappa \qcat{D}$
  restricting to $[\Psi_{\qcat{D}} \hat x, Y]$
  on $D(\simp{m} \times \bdsimp{1})$.
  By \Cref{eqs-in-nf} it will be sufficient to construct such a diagram
  on $D([m] \times \hat{[1]})$ and by \Cref{relative-D-Sd} it will suffice to
  define it on $\Sd([m] \times \hat{[1]})$.

  We form a pushout on the left
  \begin{ctikzpicture}
    \matrix [diagram]
    {
      |(Ybd)| \tilde Y|\bdsimp{m}^\ucone & |(Y)| \tilde Y &
        |(bd)| \bdsimp{m} & |(A)| A_{[m]} \\
      |(z)| z & |(Z)| Z & |(s)| \simp{m} & |(B)| B \\
    };

    \draw[inj] (bd) to (A);
    \draw[inj] (s)  to (B);

    \draw[inj] (bd) to (s);
    \draw[inj] (A)  to (B);

    \draw[cof] (Ybd) to (Y);
    \draw[cof] (z)   to (Z);

    \draw[cof] (Ybd) to (z);
    \draw[cof] (Y)   to (Z);
  \end{ctikzpicture}
  in $\dgrm_\kappa \qcat{D}$.
  Its underlying square of simplicial sets is $(\uvar)^\ucone$ applied to
  the square on the right.

  This yields the following sequence of morphisms of $\dgrm_\kappa \qcat{D}$
  (with morphisms of the underlying simplicial sets displayed below).
  \begin{ctikzpicture}
    \matrix[diagram]
    {
      |(x)| \hat x & |(z)| z & |(Z)| Z &
        |(tY)| \tilde Y_{[m]} & |(Y)| Y_{[m]} \\[-1em]
      |(s)| \simp{m} & |(sc)| \simp{m}^\ucone & |(B)| B^\ucone &
        |(Ac)| A_{[m]}^\ucone & |(A)| A_{[m]} \\
    };

    \draw[cof] (x)  to (z);
    \draw[cof] (z)  to (Z);
    \draw[cof] (Y)  to (tY);
    \draw[cof] (tY) to (Z);

    \draw[inj] (s)  to (sc);
    \draw[inj] (sc) to (B);
    \draw[inj] (A)  to (Ac);
    \draw[inj] (Ac) to (B);
  \end{ctikzpicture}
  The first morphism is a weak equivalence since $z$ is a filler
  of a special horn.
  So are the middle two since the underlying maps of simplicial sets preserve
  the cone points.
  The last one is also a weak equivalence since $\tilde Y_{[m]}$ is universal.
  All these morphisms are maps of cones under
  $Y|\Sd\bdsimp{m} = \Psi_{\cat{D}} x|\Sd\bdsimp{m}$ and hence can be seen as
  transformations of diagrams over $\Sd[m]$ which restrict to identities
  over $\Sd\bdsimp{m}$.
  The conclusion follows by \Cref{nat-we-eq}.
\end{proof}

Before we can prove the main theorem we need to know that $\dgrm_\kappa$
is a \hore{} functor.
This in turn requires two technical lemmas.
The first one is about left homotopies in cofibration categories.
Even though cofibrations in a cofibration category do not necessarily satisfy
any lifting property, they can still be shown to have a version
of the ``homotopy extension property'' \wrt{} left homotopies.

\begin{lemma}\label{HEP}
  Let $i \from A \cto B$ be a cofibration in $\cat{C}$.
  Let $f \from A \to X$ and $g \from B \to X$ be morphisms \st{} $g i$
  is left homotopic to $f$.
  Then there exist a weak equivalence $s \from X \to \hat{X}$
  and a morphism $\tilde{g} \from B \to \tilde X$ \st{} $\tilde{g}$
  is left homotopic to $s g$ and $\tilde{g} i = s f$.
\end{lemma}

\begin{proof}
  Pick compatible cylinders on $A$ and $B$, i.e.\ a diagram
  \begin{ctikzpicture}
    \matrix[diagram]
    {
      |(AA)| A \coprod A & |(IA)| IA & |(A)| A \\
      |(BB)| B \coprod B & |(IB)| IB & |(B)| B \\
    };

    \draw[cof] (AA) to (IA);
    \draw[cof] (BB) to (IB);

    \draw[->] (IA) to node[above] {$\we$} (A);
    \draw[->] (IB) to node[above] {$\we$} (B);

    \draw[->] (AA) to node[left] {$i \coprod i$} (BB);

    \draw[->] (IA) to (IB);

    \draw[->] (A) to node[right] {$i$} (B);
  \end{ctikzpicture}
  \st{} the induced morphism $IA \push_{(A \coprod A)} (B \coprod B) \to IB$
  is a cofibration.
  Let $\face_0$ and $\face_1$ denote the two structure morphisms $A \cto IA$.

  Pick a left homotopy
  \begin{ctikzpicture}
    \matrix[diagram]
    {
      |(A)|  A \coprod A & |(X)|  X \\
      |(IA)| IA          & |(tX)| \tilde{X} \\
    };

    \draw[->] (A) to node[above] {$[f, g i]$} (X);

    \draw[cof] (A) to node[left] {$[\face_0, \face_1]$} (IA);

    \draw[->] (IA) to node[below] {$H$} (tX);

    \draw[cof] (X) to node[right] {$\we$} node[left] {$j$} (tX);
  \end{ctikzpicture}
  between $f$ and $g i$.
  Then we have in particular $j g i = H \face_1$ and thus
  there is an induced morphism $[H, jg] \from IA \push_A B \to \tilde{X}$
  so we can take a pushout
  \begin{ctikzpicture}
    \matrix[diagram]
    {
      |(M)|  IA \push_A B & |(tX)| \tilde{X} \\
      |(IB)| IB           & |(hX)| \hat{X} \text{.} \\
    };

    \draw[->] (M) to node[above] {$[H, j g]$} (tX);

    \draw[cof] (M) to node[left] {$\we$} (IB);

    \draw[->] (IB) to node[below] {$\tilde{H}$} (hX);

    \draw[cof] (tX) to node[right] {$\we$} node[left] {$\tilde{j}$} (hX);
  \end{ctikzpicture}
  Set $s = \tilde{j} j$ and $\tilde{g} = \tilde{H}$.
  We have $s f = \tilde{g} i$ and $\tilde{H}$ and $\id_{\hat{X}}$
  constitute a left homotopy between $\tilde{g}$ and $s g$.
\end{proof}

The second lemma says that up to equivalence all frames
are Reedy cofibrant replacements of constant diagrams.

\begin{lemma}\label{frame-pullback}
  Any object of $X \in \nf \cat{C}$ is equivalent
  to a Reedy cofibrant replacement of $p_{[0]}^* X_0$.
\end{lemma}

\begin{proof}
  Let $f \from [0] \to D[0]$ and $s \from D[0] \to D[0]$ be as in the proof of
  \Cref{Dm-m-he} so that $p_{[0]} f = \id_{[0]}$ and there are weak equivalences
  \begin{ctikzpicture}
    \matrix[diagram]
    {
      |(id)| \id & |(s)| s & |(fp)| f p_{[0]} \text{.} \\
    };

    \draw[->] (id) to node[above] {$\we$} (s);
    \draw[->] (fp) to node[above] {$\we$} (s);
  \end{ctikzpicture}
  These equivalences evaluated at $X$ form a diagram
  $D[0] \times \Sd\hat{[1]} \to \cat{C}$ which we can pull back along
  $D\hat{[1]} \to D[0] \times \Sd\hat{[1]}$ and then replace Reedy cofibrantly
  to obtain a \hore{} Reedy cofibrant diagram $Y \from D\hat{[1]} \to \cat{C}$
  \st{} $Y\face_1 = X$ by \Cref{relative-factorization}.
  By \Cref{eqs-in-nf} $Y$ is an equivalence and by the construction $Y\face_0$
  is a Reedy cofibrant replacement of $p_{[0]}^* X_0$.
\end{proof}

\begin{lemma}\label{dgrm-relative}
  The functor $\dgrm_\kappa$ is \hore{}.
\end{lemma}

\begin{proof}
  We begin by constructing a natural equivalence
  $\Theta_{\cat{C}} \from \Ho \nf \cat{C} \to \Ho \cat{C}$
  for every cofibration category $\cat{C}$.
  We send an object $X \from D[0] \to \cat{C}$ to $X_0$ and a morphism
  $Y \from D[1] \to \cat{C}$ to the composite $[\upsilon_1]^{-1} [\upsilon_0]$
  where $\upsilon_0$ and $\upsilon_1$ are the structure morphisms
  \begin{ctikzpicture}
    \matrix[diagram]
    {
      |(0)| Y_0 & |(01)| Y_{01} & |(1)| Y_1 \text{.} \\
    };

    \draw[->] (0) to node[above] {$\upsilon_0$} (01);
    \draw[->] (1) to node[above] {$\upsilon_1$} node[below] {$\we$} (01);
  \end{ctikzpicture}
  This assignment is well-defined and functorial
  by \Cref{calculus-of-fractions}.

  We check that $\Theta_{\cat{C}}$ is an equivalence.
  It is surjective and full since both $\Sd[0] \ito D[0]$
  and $D\bdsimp{1} \union \Sd[1] \ito D[1]$ have the \Rllp{} \wrt{}
  all cofibration categories by \Cref{relative-D-Sd}.
  For faithfulness, consider $X, \tilde{X} \from D[1] \to \cat{C}$
  \st{} $X|D\bdsimp{1} = \tilde{X}|D\bdsimp{1}$
  and $\Theta_{\cat{C}}(X) = \Theta_{\cat{C}}(\tilde{X})$.
  Since we have already verified that $\Theta_{\cat{C}}$
  is essentially surjective \Cref{frame-pullback} allows us to assume that
  $X \face_0$ is a Reedy cofibrant replacement of $p_{[0]}^* X_1$ so that
  the structure morphisms of $X$ fit into a cylinder
  \begin{align*}
    X_1 \coprod X_1 \cto X_{11} \weto X_1 \text{.}
  \end{align*}
  By \Cref{calculus-of-fractions}\ref{fractions-eq} we have a diagram
  \begin{ctikzpicture}
    \matrix[diagram]
    {
      & |(01)| X_{01} \\
      |(0)| X_0 & |(Y)| Y & |(1)| X_1 \\
      & |(t01)| \tilde{X}_{01} \\
    };

    \draw[cof] (0) to (01);
    \draw[cof] (0) to (t01);

    \draw[cof] (1) to node[above right] {$\we$}
                      node[below left]  {$\nu$} (01);
    \draw[cof] (1) to node[below right] {$\we$}
                      node[above left]  {$\tilde{\nu}$} (t01);

    \draw[->] (01)  to node[left] {$\phi$}         node[right] {$\we$} (Y);
    \draw[->] (t01) to node[left] {$\tilde{\phi}$} node[right] {$\we$} (Y);
  \end{ctikzpicture}
  where both squares commute up to left homotopy.
  By \Cref{HEP} we can assume that the left square commutes strictly.
  Let
  \begin{ctikzpicture}
    \matrix[diagram]
    {
      |(1)|  X_1 \coprod X_1 & |(Y)|  Y \\
      |(11)| X_{11}          & |(hY)| \hat{Y} \\
    };

    \draw[->] (1) to node[above] {$[\phi \nu, \tilde{\phi} \tilde{\nu}]$} (Y);

    \draw[cof] (1) to node[left] {$[\face_0, \face_1]$} (11);

    \draw[->] (11) to node[below] {$\chi$} (hY);

    \draw[cof] (Y) to node[right] {$\we$} node[left] {$\psi$} (hY);
  \end{ctikzpicture}
  be a left homotopy.
  Then we can form a diagram
  \begin{ctikzpicture}
    \matrix [diagram]
    {
      & & |(b)| X_1 \\
      \\ \\
      |(a)| X_0 & & & & |(c)| X_1 \\
    };

    \node (ab)  at (barycentric cs:a=1,b=1,c=0) {$X_{01}$};
    \node (bc)  at (barycentric cs:a=0,b=1,c=1) {$X_{11}$};
    \node (ac)  at (barycentric cs:a=1,b=0,c=1) {$\tilde{X}_{01}$};
    \node (abc) at (barycentric cs:a=1,b=1,c=1) {$\hat{Y}$};

    \draw[->] (a) to (ab);
    \draw[->] (b) to node[above left]  {$\we$} node[below right] {$\nu$} (ab);
    \draw[->] (b) to node[above right] {$\we$} (bc);
    \draw[->] (c) to node[above right] {$\we$} (bc);
    \draw[->] (a) to (ac);
    \draw[->] (c) to node[below] {$\we$} node[above] {$\tilde{\nu}$} (ac);

    \draw[->] (ab) to node[below left] {$\we$}
                      node[above right] {$\psi \phi$} (abc);
    \draw[->] (bc) to node[above left] {$\we$} node[below right] {$\chi$} (abc);
    \draw[->] (ac) to node[right] {$\we$}
                      node[left]  {$\psi \tilde{\phi}$} (abc);
  \end{ctikzpicture}
  which is a \hore{} diagram on $\Sd[2]$
  and Reedy cofibrant over $\Sd\bdsimp{2}$.
  Thus it can be replaced Reedy cofibrantly without modifying it over
  $\Sd\bdsimp{2}$ by \Cref{relative-factorization}.
  Then $X$, $\tilde{X}$ and $X \face_0 \dgn_0$ provide an extension
  over $D\bdsimp{2}$.
  We know that the inclusion $D\bdsimp{2} \union \Sd[2] \ito D[2]$
  has the \Rllp{} \wrt{} all cofibration categories by \Cref{relative-D-Sd}
  so we can find an extension to $D[2]$ which is a homotopy between $X$
  and $\tilde{X}$ in $\nf\cat{C}$.

  Since equivalences of quasicategories induce equivalences
  of homotopy categories, it follows that $\nf$ reflects equivalences.
  Thus $\dgrm_\kappa$ is \hore{} by \Cref{Psi-equivalence}.
\end{proof}

Finally, we are ready to prove the main theorem.

\begin{theorem}\label{nf-equivalence}
  The functor $\nf \from \CofCat_\kappa \to \QCat_\kappa$ is a weak equivalence
  of fibration categories.
\end{theorem}

\begin{proof}
  By \Cref{nerve-continuous} $\nf$ is continuous.
  The functor $\dgrm_\kappa$ is \hore{} by \Cref{dgrm-relative}
  and thus induces a functor on the homotopy categories.
  Since $\Psi$ is a natural categorical equivalence by \Cref{Psi-equivalence}
  the induced transformation $\Ho \Psi$ is a natural isomorphism
  $\id \to (\Ho \nf) (\Ho \dgrm_\kappa)$.
  The transformation $\Phi$ is merely $2$-natural, but natural isomorphisms
  of exact functors induce right homotopies in $\CofCat_\kappa$
  (by the construction of path objects in the proof
  of \Cref{fibcat-of-cofcats}).
  Therefore $\Ho \Phi$ is a natural transformation and by \Cref{Phi-equivalence}
  it is an isomorphism $(\Ho \dgrm_\kappa) (\Ho \nf) \to \id$.
  Hence $\Ho \nf$ is an equivalence.
\end{proof}

  \subsection{Proof of the main theorem: the finite case}
\label{sec:proof-finite}

The only part of the previous subsection that does not work for $\kappa = \aleph_0$
is the construction of a natural weak equivalence
$\Phi_{\cat{C}} \from \dgrm_\kappa \nf \cat{C} \to \cat{C}$
for every cofibration category $\cat{C}$.
Indeed, $\Phi_{\cat{C}}$ was defined using colimits over categories $DK$
which are infinite even for finite simplicial sets $K$.
Instead, we will define a zig-zag of ($2$-natural) weak equivalences
connecting $\dgrm_{\aleph_0} \nf \cat{C}$ to $\cat{C}$, namely,
\begin{ctikzpicture}
  \matrix[diagram]
  {
    |(dnf)| \dgrm_{\aleph_0} \nf \cat{C} & |(CtN)| \cat{C}^{\tilde\nat}_\Reedy &
      |(ChN)| \cat{C}^{\hat\nat}_\Reedy & |(C)| \cat{C} \text{.} \\
  };

  \draw[->] (dnf) to node[above] {$\Phi^{(\uvar)}_{\cat{C}}$} (CtN);

  \draw[inj] (ChN) to (CtN);
  
  \draw[->] (ChN) to node[above] {$\ev_0$} (C);
\end{ctikzpicture}

We have already verified that $\cat{C}^{\hat\nat}_\Reedy
\ito \cat{C}^{\tilde\nat}_\Reedy$ is a weak equivalence
in \Cref{eventually-constant}.
Moreover, $\ev_0 \from \cat{C}^{\tilde\nat}_\Reedy \to \cat{C}$ is induced by
a homotopy equivalence $[0] \to \hat\nat$ hence it is a weak equivalence, too.

It remains to define $\Phi^{(\uvar)}_{\cat{C}}$ and prove that it is also
a weak equivalence.
For each $k$ and an object $X \from DK \to \nf \cat{C}$ we set
$\Phi^{(k)}_{\cat{C}} X = \colim_{\filt{k,K}} X$.
This colimit exists since $\filt{k,K}$ is finite if $K$ is finite.

\begin{lemma}\label{Phi-equivalence-finite}
  For a cofibration category $\cat{C}$ the formula above defines
  an exact functor
  $\Phi^{(\uvar)}_{\cat{C}}
    \from \dgrm_{\aleph_0} \nf \cat{C} \to \cat{C}^{\tilde\nat}_\Reedy$.
  Moreover, it is a weak equivalence.
\end{lemma}

\begin{proof}
  First, we need to verify that $\Phi^{(\uvar)}_{\cat{C}} X$ is
  an eventually constant sequence for all
  $(K, X) \in \dgrm_{\aleph_0} \nf \cat{C}$.
  Consider $X$ as a diagram in $\nf \cat{C}$ and choose a universal cone
  $S \from K^\ucone \to \nf \cat{C}$.
  Then \Cref{colim-simplex-join-finite} implies that
  $\Phi^{(\uvar)}_{\cat{C}} S$ is eventually constant
  and \Cref{nerve-cones-finite} implies that the induced morphism
  $\Phi^{(\uvar)}_{\cat{C}} S \to \Phi^{(\uvar)}_{\cat{C}} S$
  is an eventual weak equivalence.
  Thus $\Phi^{(\uvar)}_{\cat{C}} S$ is eventually constant.

  Preservation of cofibrations follows by \cite{rb}*{Theorem 9.4.1(1a)}
  since if $K \ito L$ is an injective map of simplicial sets,
  then the induced functors $\filt{k,K} \union \filt{k-1,L} \to \filt{k,L}$
  are sieves.

  \Cref{nerve-cones-finite,colim-simplex-join-finite} imply that a morphism $f$
  in $\dgrm_{\aleph_0} \nf \cat{C}$ is a weak equivalence
  \iff{} $\Phi^{(\uvar)}_{\cat{C}} f$ is an eventual weak equivalence.
  Therefore $\Phi^{(\uvar)}_{\cat{C}}$ preserves weak equivalences
  and satisfies (App1).

  Colimits in $\cat{C}$ are compatible with colimits of indexing categories
  and thus $\Phi^{(\uvar)}_{\cat{C}}$ is exact.

  It remains to check (App2),
  but it follows directly from \Cref{cone-factorization-finite}.
\end{proof}

This yields the proof of of \Cref{nf-equivalence} in the case of
$\kappa = \aleph_0$ since the three weak equivalences described above induce
a natural isomorphism $(\Ho \dgrm_\kappa) (\Ho \nf) \to \id$ and the rest of
the argument applies verbatim.

\end{document}